\documentclass[UTF-8,reqno]{amsart}
\usepackage{enumerate}
\usepackage{mhequ, tikz, wasysym, stmaryrd}
\usepackage{amssymb,url,color, booktabs,nccmath}
\usepackage[left=3.2cm,right=3.2cm,top=4cm,bottom=4cm]{geometry}
\usepackage{mathrsfs}
\usepackage{enumitem,dsfont}

\usepackage{color}
\usepackage[colorlinks=true]{hyperref}
\hypersetup{
	linkcolor=blue,          
	citecolor=red,        
	filecolor=blue,      
	urlcolor=cyan
}

\definecolor{darkergreen}{rgb}{0.0, 0.5, 0.0}

\numberwithin{equation}{section}
\def\theequation{\arabic{section}.\arabic{equation}}
\newcommand{\be}{\begin{eqnarray}}
\newcommand{\ee}{\end{eqnarray}}
\newcommand{\ce}{\begin{eqnarray*}}
	\newcommand{\de}{\end{eqnarray*}}
\newtheorem{theorem}{Theorem}[section]
\newtheorem{lemma}[theorem]{Lemma}
\newtheorem{proposition}[theorem]{Proposition}
\newtheorem{Examples}[theorem]{Example}
\newtheorem{corollary}[theorem]{Corollary}

\newtheorem{definition}[theorem]{Definition}
\theoremstyle{definition}
\newtheorem{remark}[theorem]{Remark}

\newcommand{\mathd}{\mathrm{d}}
\newcommand{\tmcolor}[2]{{\color{#1}{#2}}}
\newcommand{\tmmathbf}[1]{\ensuremath{\mathbf{#1}}}
\newcommand{\tmop}[1]{\ensuremath{\operatorname{#1}}}

\def\eps{\varepsilon}

\def\p{\partial}

\def\<{{\langle}}
\def\>{{\rangle}}
\def\({{\Big(}}
\def\){{\Big)}}

\def\bx{{\mathbf{x}}}

\def\dif{{\mathord{{\rm d}}}}

\def\={&\!\!=\!\!&}

\def\cF{{\mathcal F}}

\def\cI{{\mathcal I}}

\def\mN{{\mathbb N}}

\def\mP{{\mathbb P}}

\def\mR{{\mathbb R}}

\def\1{{\mathbf{1}}}

\def\E{\mathbf E}

\def\geq{\geqslant}
\def\leq{\leqslant}

\def\div{\mathord{{\rm div}}}

\def\eps{\varepsilon}

\def\p{\partial}

\def\<{{\langle}}
\def\>{{\rangle}}
\def\({{\Big(}}
\def\){{\Big)}}

\def\bx{{\mathbf{x}}}

\def\dif{{\mathord{{\rm d}}}}

\def\={&\!\!=\!\!&}
\def\bt{\begin{theorem}}
	\def\et{\end{theorem}}
\def\bl{\begin{lemma}}
	\def\el{\end{lemma}}
\def\br{\begin{remark}}
	\def\er{\end{remark}}
\def\bx{\begin{Examples}}
	\def\ex{\end{Examples}}
\def\bd{\begin{definition}}
	\def\ed{\end{definition}}
\def\bp{\begin{proposition}}
	\def\ep{\end{proposition}}
\def\bc{\begin{corollary}}
	\def\ec{\end{corollary}}

\def\geq{\geqslant}
\def\leq{\leqslant}

\def\div{\mathord{{\rm div}}}

\def\Id{\textrm{Id}}

 \def\R{\mathbb R}
 \def\R{\mathbb R}    
\def\N{\mathbb N}  
   
\def\<{\langle} \def\>{\rangle}

\allowdisplaybreaks

\newcommand{\oprec}{\hspace{.09cm}\Circle\hspace{-.395cm}\prec}
\newcommand{\osucc}{\hspace{.09cm}\Circle\hspace{-.35cm}\succ}

\newcommand{\opreccurlyeq}{\hspace{.09cm}\Circle\hspace{-.4cm}\preccurlyeq\hspace{.1cm}}
\newcommand{\osucccurlyeq}{\hspace{.09cm}\Circle\hspace{-.35cm}\succcurlyeq\hspace{.05cm}}

\def\drawx{\draw[-,solid] (-3pt,-3pt) -- (3pt,3pt);\draw[-,solid] (-3pt,3pt) -- (3pt,-3pt);}
\tikzset{
	root/.style={circle, draw=black, fill=white, inner sep=0pt, minimum size=0.7mm},
	dot/.style={circle,fill=black,draw=black, solid,inner sep=0pt,minimum size=0.5mm},
	square/.style={rectangle,fill=black,draw=black, solid,inner sep=0pt,minimum size=1mm},
	empty/.style={circle,fill=white,draw=white, solid,inner sep=0pt,minimum size=0.5mm},
	var/.style={circle,fill=black!10,draw=black,inner sep=0pt, minimum size=
		2mm},
	symb/.style={circle,fill=symbols,draw=symbols, solid,inner sep=0pt,minimum size=0.5mm},
	yy/.style={circle,fill=gray!20,draw=black,inner sep=0pt,minimum size=0.8mm},
	>=stealth,
	dotred/.style={circle,fill=black!50,inner sep=0pt, minimum size=2mm},
	generic/.style={semithick,shorten >=1pt,shorten <=1pt},
	dist/.style={ultra thick,draw=testcolor,shorten >=1pt,shorten <=1pt},
	testfcn/.style={ultra thick,testcolor,shorten >=1pt,shorten <=1pt,<-},
	testfcnx/.style={ultra thick,testcolor,shorten >=1pt,shorten <=1pt,<-,
		postaction={decorate,decoration={markings,mark=at position 0.6 with {\drawx}}}},
	kprime/.style={semithick,shorten >=1pt,shorten <=1pt,densely dashed,->},
	kprimex/.style={semithick,shorten >=1pt,shorten <=1pt,densely dashed,->,
		postaction={decorate,decoration={markings,mark=at position 0.4 with {\drawx}}}},
	kernel/.style={semithick,shorten >=1pt,shorten <=1pt,->},
	multx/.style={shorten >=1pt,shorten <=1pt,
		postaction={decorate,decoration={markings,mark=at position 0.5 with {\drawx}}}},
	kernelx/.style={semithick,shorten >=1pt,shorten <=1pt,->,
		postaction={decorate,decoration={markings,mark=at position 0.4 with {\drawx}}}},
	kernel1/.style={->,semithick,shorten >=1pt,shorten <=1pt,postaction={decorate,decoration={markings,mark=at position 0.45 with {\draw[-] (0,-0.1) -- (0,0.1);}}}},
	kernel2/.style={->,semithick,shorten >=1pt,shorten <=1pt,postaction={decorate,decoration={markings,mark=at position 0.45 with {\draw[-] (0.05,-0.1) -- (0.05,0.1);\draw[-] (-0.05,-0.1) -- (-0.05,0.1);}}}},
	kernelBig/.style={semithick,shorten >=1pt,shorten <=1pt,decorate, decoration={zigzag,amplitude=1.5pt,segment length = 3pt,pre length=2pt,post length=2pt}},
	rho/.style={dotted,semithick,shorten >=1pt,shorten <=1pt},
	renorm/.style={shape=circle,fill=white,inner sep=1pt},
	labl/.style={shape=rectangle,fill=white,inner sep=1pt},
	xi/.style={circle,fill=symbols!10,draw=symbols,inner sep=0pt,minimum size=1.2mm},
	xix/.style={crosscircle,fill=symbols!10,draw=symbols,inner sep=0pt,minimum size=1.2mm},
	xib/.style={circle,fill=symbols!10,draw=symbols,inner sep=0pt,minimum size=1.6mm},
	xibx/.style={crosscircle,fill=symbols!10,draw=symbols,inner sep=0pt,minimum size=1.6mm},
	not/.style={circle,fill=symbols,draw=symbols,inner sep=0pt,minimum size=0.5mm},
	>=stealth,
}

\colorlet{symbols}{blue!90!black}

\makeatletter
\def\DeclareSymbol#1#2#3{\expandafter\gdef\csname MH@symb@#1\endcsname{\tikz[baseline=#2,scale=0.15]{#3}}%
	\expandafter\gdef\csname MH@symb@#1s\endcsname{\scalebox{0.6}{\tikz[baseline=#2,scale=0.15]{#3}}}}
\def\<#1>{\csname MH@symb@#1\endcsname}
\makeatother

\DeclareSymbol{2}{0}{\draw (-0.5,1.2) node[dot] {} -- (0,0.2) -- (0.5,1.2) node[dot] {};}

\DeclareSymbol{20}{0}{\draw (0,0) -- (0,-.7);\draw (-.6,.9) node[dot] {} -- (0,0) -- (.6,.9) node[dot] {};}

\DeclareSymbol{21}{-2}{\draw (-.6,.9) node[dot] {} -- (0,0) -- (.6,.9) node[dot] {}; \draw (0,0) -- (0,-.7) -- (.7,0) node[dot] {};}

\DeclareSymbol{21l}{-2}{\draw (-.6,.9) node[dot] {} -- (0,0) -- (.6,.9) node[dot] {}; \draw (0,0) -- (0,-.7) -- (-.7,0) node[dot] {};}

\DeclareSymbol{211}{0}{\draw (0,0) -- (0,-.7);\draw (-.6,.7) node[dot] {} -- (0,0) -- (.6,.7) node[dot] {}; \draw (0,0) -- (0,-.7) -- (.7,0) node[dot] {}; \draw (0,0) -- (0,-.7) -- (0,-1.4);}

\DeclareSymbol{2111}{0}{\draw (0,0) -- (0,-.7);\draw (-.6,.7) node[dot] {} -- (0,0) -- (.6,.7) node[dot] {}; \draw (0,0) -- (0,-.7) -- (.7,0) node[dot] {}; \draw (0,0) -- (0,-.8) -- (0,-1.4); \draw (0,0) -- (0,-1.4) -- (.7,-.7) node[dot] {};\node [root] at (0,-1.4) {};}

\DeclareSymbol{21l1}{0}{\draw (0,0) -- (0,-.7);\draw (-.6,.7) node[dot] {} -- (0,0) -- (.6,.7) node[dot] {}; \draw (0,0) -- (0,-.7) -- (-.7,0) node[dot] {}; \draw (0,0) -- (0,-.7) -- (0,-1.4);}

\DeclareSymbol{101}{0}{\draw (0,0) -- (0,-.7); \draw (-.6,.9) node[dot] {} -- (0,0);\draw (0,0) -- (0,-.7) -- (.7,0) node[dot]{};\node [root] at (0,-.7) {};}

\DeclareSymbol{2020}{-2}{\draw (-.5,0) -- (0,-.7) -- (.5,0) ;\draw (-1.1,.9) node[dot] {} -- (-.5,0) -- (-.3,1) node[dot] {}; \draw (1.1,.9) node[dot] {} -- (.5,0) -- (.3,1) node[dot] {};}

\DeclareSymbol{X}{-2.4}{\node[dot] {};}
\DeclareSymbol{1}{0}{\draw[white] (-.4,0) -- (.4,0); \draw (0,0)  -- (0,1.2) node[dot] {};}

\DeclareSymbol{3}{0}{\draw (0,0) -- (0,1.2) node[dot] {}; \draw (-.7,1) node[dot] {} -- (0,0) -- (.7,1) node[dot] {};}

\DeclareSymbol{31}{-3}{\draw (0,0) -- (0,-0.7) node[root] {} -- (0.9,0.1) node[dot] {}; \draw (0,0) -- (0,1.2) node[dot] {}; \draw (-.7,1) node[dot] {} -- (0,0) -- (.7,1) node[dot] {};}
\DeclareSymbol{30}{-3}{\draw (0,0) -- (0,-0.7); \draw (0,0) -- (0,1.2) node[dot] {}; \draw (-.7,1) node[dot] {} -- (0,0) -- (.7,1) node[dot] {};}

\DeclareSymbol{10}{-3}{\draw (0,0) -- (0,-1); \draw (-.8,1) node[dot] {} -- (0,0);}
\DeclareSymbol{32}{-3}{\draw (0,0) -- (0,-0.7) -- (0.7,0) node[dot] {}; \draw  (0,-0.7)  -- (-0.7,0) node[dot] {}; \draw (0,0) -- (0,1) node[dot] {}; \draw (-.7,0.8) node[dot] {} -- (0,0) -- (.7,0.8) node[dot] {}; \node [root] at (0,-0.7) {};}
\DeclareSymbol{12}{-3}{\draw (0,0) -- (0,-1) -- (1,0) node[dot] {}; \draw (0,0) -- (0,-1) -- (-1,0) node[dot] {}; \draw (-.8,1) node[dot] {} -- (0,0);}
\DeclareSymbol{22}{-3}{\draw (0,0) -- (0,-0.7) -- (0.7,0) node[dot] {}; \draw (0,0) -- (0,-0.7) -- (-0.7,0) node[dot] {};\draw (-.5,0.8) node[dot] {} -- (0,0) -- (.5,0.8) node[dot] {};  \node [root] at (0,-0.7) {};}
\DeclareSymbol{32W}{-3}{\draw  [densely dotted, semithick] (-1.5,0)node[dot] {}  -- (0,-1.5) node[yy]{} -- (1.5,0) node[dot] {}; \draw (.8,-2.3)node{\tiny $y$} ;\draw[semithick] (0,0) -- (0,-1.5) node[yy] {}; \draw [densely dotted, semithick] (0,0) -- (0,1.8) node[dot] {}; \draw[densely dotted, semithick] (-1,1.5) node[dot] {} -- (0,0) -- (1,1.5) node[dot] {};}
\DeclareSymbol{32WW}{-3}{\draw  [densely dotted, semithick] (-1.5,0)node[dot] {}  -- (0,-1.5) node[yy]{} -- (1.5,0) node[dot] {}; \draw (.8,-2.3)node{\tiny $\bar{y}$} ;\draw[semithick] (0,0) -- (0,-1.5) node[yy] {}; \draw [densely dotted, semithick] (0,0) -- (0,1.8) node[dot] {}; \draw[densely dotted, semithick] (-1,1.5) node[dot] {} -- (0,0) -- (1,1.5) node[dot] {};}

\DeclareSymbol{s1}{0}{\draw[white] (-.4,0) -- (.4,0); \draw[symbols] (0,0)  -- (0,1.2) node[symb] {};}
\DeclareSymbol{s2}{0}{\draw[symbols] (-0.5,1.2) node[symb] {} -- (0,0) -- (0.5,1.2) node[symb] {};}
\DeclareSymbol{s3}{0}{\draw[symbols] (0,0) -- (0,1.2) node[symb] {}; \draw[symbols] (-.7,1) node[symb] {} -- (0,0) -- (.7,1) node[symb] {};}
\DeclareSymbol{s31}{-3}{\draw[symbols] (0,0) -- (0,-1) -- (1,0) node[symb] {}; \draw[symbols] (0,0) -- (0,1.2) node[symb] {}; \draw[symbols] (-.7,1) node[symb] {} -- (0,0) -- (.7,1) node[symb] {};}
\DeclareSymbol{s30}{-3}{\draw[symbols] (0,0) -- (0,-1); \draw[symbols] (0,0) -- (0,1.2) node[symb] {}; \draw[symbols] (-.7,1) node[symb] {} -- (0,0) -- (.7,1) node[symb] {};}
\DeclareSymbol{s32}{-3}{\draw[symbols] (0,0) -- (0,-1) -- (1,0) node[symb] {}; \draw[symbols] (0,0) -- (0,-1) -- (-1,0) node[symb] {}; \draw[symbols] (0,0) -- (0,1.2) node[symb] {}; \draw[symbols] (-.7,1) node[symb] {} -- (0,0) -- (.7,1) node[symb] {};}
\DeclareSymbol{s22}{-3}{\draw[symbols] (0,0.3) -- (0,-1) -- (1,0) node[symb] {}; \draw[symbols] (0,0.3) -- (0,-1) -- (-1,0) node[symb] {};\draw[symbols] (-.7,1) node[symb] {} -- (0,0.3) -- (.7,1) node[symb] {};}
\DeclareSymbol{s20}{-3}{\draw[symbols] (0,0) -- (0,-1);\draw[symbols] (-.7,1) node[symb] {} -- (0,0) -- (.7,1) node[symb] {};}
\DeclareSymbol{s21}{-3}{\draw[symbols] (0,0) -- (0,-1);\draw[symbols] (-.7,1) node[symb] {} -- (0,0) -- (.7,1) node[symb] {}; \draw[symbols] (0,0) -- (0,-1) -- (1,0) node[symb] {};}
\DeclareSymbol{s12}{-3}{\draw[symbols] (0,0) -- (0,-1) -- (1,0) node[symb] {}; \draw[symbols] (0,0) -- (0,-1) -- (-1,0) node[symb] {}; \draw (-.8,1) node[symb] {} -- (0,0);}
\DeclareSymbol{s11}{-3}{\draw[symbols] (0,0) -- (0,-1) -- (-1,0) node[symb] {}; \draw (-.8,1) node[symb] {} -- (0,0);}
\DeclareSymbol{s10}{-3}{\draw[symbols] (0,0) -- (0,-1); \draw[symbols] (-.8,1) node[symb] {} -- (0,0);}

\begin{document}

	\title[Global existence and non-uniqueness for 3D NSE with space-time white noise]{Global existence and non-uniqueness for 3D Navier--Stokes equations with space-time white noise}

	\author{Martina Hofmanov\'a}
	\address[M. Hofmanov\'a]{Fakult\"at f\"ur Mathematik, Universit\"at Bielefeld, D-33501 Bielefeld, Germany}
	\email{hofmanova@math.uni-bielefeld.de}
	
	\author{Rongchan Zhu}
	\address[R. Zhu]{Department of Mathematics, Beijing Institute of Technology, Beijing 100081, China; Fakult\"at f\"ur Mathematik, Universit\"at Bielefeld, D-33501 Bielefeld, Germany}
	\email{zhurongchan@126.com}
	
	\author{Xiangchan Zhu}
	\address[X. Zhu]{ Academy of Mathematics and Systems Science,
		Chinese Academy of Sciences, Beijing 100190, China; Fakult\"at f\"ur Mathematik, Universit\"at Bielefeld, D-33501 Bielefeld, Germany}
	\email{zhuxiangchan@126.com}
	\thanks{
This project has received funding from the European Research Council (ERC) under the European Union's Horizon 2020 research and innovation programme (grant agreement No. 949981).  The financial support by the DFG through the CRC 1283 ``Taming uncertainty and profiting
 from randomness and low regularity in analysis, stochastics and their applications'' is greatly acknowledged.
 R.Z. is grateful to the financial supports of the NSFC (No.  11922103).
  X.Z. is grateful to the financial supports  in part by National Key R\&D Program of China (No. 2020YFA0712700) and the NSFC (No. 11771037,  12090014, 11688101) and
  the support by key Lab of Random Complex Structures and Data Science,
 Youth Innovation Promotion Association (2020003), Chinese Academy of Science.
}

\begin{abstract}
We establish global-in-time existence and non-uniqueness of probabilistically strong solutions to the three dimensional Navier--Stokes system driven by space-time white noise.  In this setting, solutions are expected to have  space regularity at most $-1/2-\kappa$ for any $\kappa>0$. Consequently, the convective term is  ill-defined analytically and probabilistic renormalization is required. Up to now, only local well-posedness has been known. With the help of paracontrolled calculus we decompose  the system in a way which makes it amenable to convex integration. By a careful analysis of the regularity of each term, we develop an iterative procedure which yields  global non-unique probabilistically strong paracontrolled solutions.
Our result applies to  any divergence free initial condition in $L^{2}\cup B^{-1+\kappa}_{\infty,\infty}$, $\kappa>0$, and implies also non-uniqueness in law.
\end{abstract}


\date{\today}

\maketitle

\tableofcontents

	\section{Introduction}

Thermal fluctuations are   omnipresent on the molecular  level of fluids. As such,  fluids are not deterministic but rather stochastic and continually changing. In order to incorporate these effects into the description of large scale dynamics, Landau and Lifshitz \cite{LL87} proposed a Navier--Stokes system perturbed by stochastic flux terms. These are given by a divergence of delta correlated space-time Gaussian random fields included in the momentum equation.
Mathematically,    it is extremely challenging to make sense of such a system. Indeed, due to the  irregularity of the noise  combined with the  nonlinearity of the  system, the equations become critical in dimension 2 and  supercritical in dimension higher in the sense of the theory of regularity structures by Hairer  \cite{Hai14}: Introducing an ultraviolet cut-off, i.e. mollifying  the noise, and trying to remove the cut-off would  require an infinite number of renormalizations.
Therefore, neither the classical stochastic analysis tools nor the  pathwise theories of regularity structures \cite{Hai14} and paracontrolled distributions by Gubinelli, Imkeller and Perkowski \cite{GIP15} are applicable.
Nevertheless, despite these theoretical difficulties, the Landau--Lifshitz--Navier--Stokes system  has been  verified numerically for various equilibrium as well as non-equilibrium systems (see e.g. \cite{DVEGB10} and references therein).

In this paper, we contribute to the rigorous mathematical understanding of microscopic perturbations in fluid dynamics. To avoid the issue of criticality mentioned above, we consider the delta correlated space-time Gaussian noise as a forcing, not as a flux.
This models forcing acting on the molecular level, which translated to the macroscopic level necessarily becomes delta correlated.  Indeed, any two points in the large scale dynamics are extremely far apart on the molecular level, so their associated noises must be uncorrelated. Moreover, such a noise also appears in a scaling limit of point vortex approximation and  the vorticity form of the 2D Euler equations perturbed by a certain transport type noise (cf. \cite{FL20, FL21, LZ21}). In fact, the scaling limit is given by the vorticity form of the 2D Navier--Stokes system driven by the curl of space-time white noise, which in the velocity-pressure variables reads as 2D Navier--Stokes equations driven by space-time white noise.

The corresponding gain of one derivative in comparison to the Landau--Lifshitz setting makes the system mathematically  accessible as it remains subcritical up to   space dimension 3.
However, even this problem has resisted rigorous mathematical analysis for a long time due to its irregularity.  More precisely, the space-time white noise in spatial dimension $d$ can be shown to be a random distribution  of space-time regularity $-(d+2)/2-\kappa$ under the parabolic scaling for any $\kappa>0$. Accordingly, in view of  Schauder's estimates a solution is expected to be  two degrees of  regularity better, i.e. at most $-d/2+1-\kappa$. Hence, already in $d=2$ solutions are not functions. Consequently, the product in the convective term is  analytically ill-defined and probabilistic arguments are required in order to make sense of the equations.

We consider the three dimensional Navier--Stokes system with periodic boundary conditions driven by a space-time white noise
\begin{equation}\label{main1}
\aligned
\dif u+\div(u\otimes u)\,\dif t+\nabla p\,\dif t&=\Delta u \,\dif t+\dif B,
\\
\div u&=0,
\\ u(0)&=u_{0},
\endaligned
\end{equation}
where $B$ is a cylindrical Wiener process on some stochastic basis $(\Omega,\mathcal{F},(\mathcal{F}_t)_{t\geq0},\mathbf{P})$. The time derivative of $B$ is the space-time white noise. Our main results read as follows.

\begin{theorem}\label{thm:1.1}
	For any given divergence free initial condition $u_{0}\in L^{2}\cup B^{-1+\kappa}_{\infty,\infty}$  $\mathbf{P}$-a.s., $\kappa>0$, there exist infinitely many global-in-time probabilistically strong  solutions solving \eqref{main1} in a paracontrolled sense.
\end{theorem}

The main ideas behind the definition of paracontrolled solution are explained in  Section~\ref{sec:1.3} and the detailed presentation can be found in  Section~\ref{s:stoch}. Probabilistically strong  solutions means that the solutions are adapted to the normal filtration $(\mathcal{F}_t)_{t\geq0}$ generated by the cylindrical Wiener process.

\begin{corollary}\label{cor:1.2}
	Non-uniqueness in law holds for \eqref{main1} for every given initial law  supported on divergence free vector fields in  $L^{2}\cup B_{\infty,\infty}^{-1+\kappa}$, $\kappa>0$.
\end{corollary}

\subsection{Singular SPDEs}

In two space dimensions, the problem was solved locally in time in the seminal paper by Da~Prato, Debussche	\cite{DD02}. Furthermore, using the properties of the Gaussian invariant measure, it was possible to obtain global-in-time existence for a.e. initial condition with respect to the invariant measure. By the strong Feller property in \cite{ZZ17}  global-in-time existence for every initial condition could be derived. Recently, for a related two dimensional critical problem, tightness of approximate stationary solutions and non-triviality of the limit has been established in a weak coupling regime  by Cannizzaro and Kiedrowski  \cite{CK21}.

The more irregular three dimensional setting remained open for much longer as substantially new ideas were required. These came in a parallel development with the theory of regularity structures by Hairer \cite{Hai14} and with the paracontrolled distributions introduced by Gubinelli, Imkeller and Perkowski \cite{GIP15}. These theories permit to treat a large number of singular subcritical SPDEs (cf. \cite{BHZ19, CH16, BCCH21}) including the Kardar--Parisi--Zhang (KPZ) equation, the generalized parabolic Anderson model and the stochastic quantization equations for quantum fields (see \cite{Hai13, CC15, GP17,  CCHS20} and references therein). In particular they led to a local well-posedness theory for the Navier--Stokes system \eqref{main1} in three dimensions by Zhu, Zhu \cite{ZZ15}.

The question of global existence is even more challenging. Roughly speaking, in the field of singular SPDEs the only available global existence results rely either on a strong drift  present in the system or a particular transform for certain nonlinearities or on  properties of an invariant measure:

\begin{itemize}
	
	\item Suitable a priori estimates  have been established
	for the dynamical $\Phi^4$ stochastic quantization model by Mourrat and Weber \cite{MW17,MW18} and Gubinelli and Hofmanov\'a \cite{GH18, GH18a} (see also \cite{SSZZ20, SZZ21} for the vector valued case). All these results make an essential use of the strong damping term $-\phi^3$.
	
	\item In \cite{GP17, PR18, ZZZ20}, a priori estimates and paracontrolled solutions to
	the KPZ equation and singular Hamilton--Jacobi--Bellman equation
	were obtained by using Cole--Hopf's  or Zvonkin's transform and maximum principle.
	
	\item Moreover, using the probabilistic notion of energy solutions \cite{GJ14, GJ13,GP18}
	or studying the associated  infinitesimal generator and Kolmogorov equation \cite{GP18a},
	it is possible to construct global solutions to
	KPZ equation, but this  depends on the invariant Gaussian measure  (i.e. the law of Brownian motion or spatial white noise).
	\item We mention that in \cite{RWZZ20,CWZZ18}
	global martingale solutions
	were constructed
	for geometric stochastic heat equations
	by using a Dirichlet form approach. This
	relies on an integration by parts formula for the known invariant measure. 	
\end{itemize}

No such results are available
for the 3D Navier--Stokes system with space-time white noise:
\begin{itemize}
	\item There is no strong drift helping to stabilize the evolution.
	\item Due to the appearance of the divergence free condition and the corresponding pressure term, it is impossible to apply maximum principle or Cole--Hopf's transform.
	\item The existence of an invariant measure is an open problem.
	\item No global energy (or other) estimates are available due to irregularity  of solutions (see below for more explanation on this point).
\end{itemize}

Furthermore, even if the regularity of the noise is increased, global existence is not known. More precisely, consider the following regularized problem which interpolates between the case of a trace-class noise and the space-time white noise
\begin{equation}\label{main3}
\aligned
\dif u+\div(u\otimes u)\,\dif t+\nabla p\,\dif t&=\Delta u \,\dif t+(-\Delta)^{-\gamma/2}\dif B,
\\
\div u&=0.
\endaligned
\end{equation}
Here $B$ is the cylindrical Wiener process. The case of $\gamma>3/2$ corresponds to a trace-class noise whereas $\gamma=0$ is the space-time white noise.

For $\gamma>3/2$, global  existence of  probabilistically weak Leray solutions is classical (see Flandoli, Gatarek \cite{FG94}). Their uniqueness remains an outstanding open problem. The authors in \cite{FG94} also constructed stationary solutions via Krylov--Bogoliubov's argument and  energy estimates.  Ergodicity and  strong Feller property was proved by Da~Prato and Debussche \cite{DD03} and also by Flandoli and Romito \cite{FR08}. A Markov transition semigroup was then constructed   by Debussche, Odasso \cite{DO06} and by Flandoli and Romito \cite{FR08}. However, further  structure and properties of the invariant measure  are still unclear.  Recently, in \cite{HZZ19} we established  non-uniqueness in law in a class of analytically weak (not Leray) solutions using the method of convex integration. In \cite{HZZ21}, we additionally presented  non-uniqueness of Markov solutions and global-in-time existence and non-uniqueness of probabilistically strong and analytically weak solutions.

If $\gamma\leq 3/2$ the usual energy estimates pertinent to the notion of Leray solution are not available. Indeed,  as the noise is no longer trace-class, It\^o's formula cannot be applied.
However, for $\gamma\in (1/2,3/2]$ it is still possible to prove global existence of probabilistically weak solutions. Namely, decomposing  the velocity $u$ into the sum of its stochastic part $z$ and its nonlinear part $v$ (as shown also in  Section~\ref{s:3.1} below), one can derive a global energy estimate for $v$. It can be then combined  with compactness and  Skorokhod representation theorem to deduce global existence. Nevertheless, since it is necessary to change the probability space in the course of the proof, these solutions are only probabilistically weak.

Up to now, even global existence of probabilistically weak solutions was open in the case of $\gamma\leq 1/2$. Only local well-posedness in the spirit of Zhu, Zhu \cite{ZZ15} seemed to be possible.
The difficulty can be seen as follows. If e.g.  $\gamma=1/2$, we decompose $u=v+z$ with $z$ solving the linear stochastic equation
$$
\dif z +\nabla p^{z}\, \dif t = \Delta z\, \dif t +(-\Delta)^{-1/4}\dif B,\qquad \div z=0,
$$
and $v$ solving the nonlinear  equation
$$\p_tv+\div((v+z)\otimes (v+z))+\nabla p^{v}=\Delta v,\qquad \div v=0.$$
As $z\in C_TC^{-\kappa}$, the best regularity for $v$ is given by $C_TC^{1-\kappa}$. Hence we cannot expect the energy estimate for $v$ since this would require the $L^2_TH^{1}$-norm of $v$. We may include further decomposition like in the case of the $\Phi^{4}$ model  \cite{MW17, MW18,GH18, GH18a}, but this would lead to new nonlinear terms which cannot be absorbed as in the $\Phi^4$ model. Due to this, we also cannot obtain energy inequality and uniform estimates as in the case of a trace-class case. Consequently, we cannot derive global solutions by a usual compactness argument. Note that it is also possible to consider global solutions for  small initial data (cf. \cite{BR17}). However, this would destroy adaptedness of the solutions since the initial data would depend on the whole path of the driving noise.

\subsection{Convex integration}

In the present paper, we focus on the case of space-time white noise, i.e. $\gamma=0$. Simplified versions of our proofs as outlined in Section~\ref{s:reg} also provide the results in the more regular cases of $\gamma>0$.
Our idea is to apply the method of  convex integration in order to construct global-in-time solutions. This is an iterative procedure which permits to construct solutions explicitly scale by scale. It makes an essential use of the form of the nonlinearity which propagates oscillations and reduces an error term, the so-called Reynolds stress, in order to  approach a solution as one proceeds through the iteration. As typical for the convex integration constructions, the same method gives raise to infinitely many solutions.

Convex integration was introduced into fluid dynamics by De~Lellis and Sz{\'e}kelyhidi~Jr. \cite{DelSze2,DelSze3,DelSze13}. This method has already led to a number of groundbreaking results concerning the incompressible Euler equations,  culminating  in the proof of  Onsager's conjecture by Isett \cite{Ise} and by Buckmaster, De~Lellis, Sz{\'e}kelyhidi~Jr. and Vicol \cite{BDSV}.
Also the question of well/ill-posedness of the three dimensional  Navier--Stokes equations has experienced an immense  breakthrough: Buckmaster and Vicol \cite{BV19a} established non-uniqueness of weak solutions with finite kinetic energy,
Buckmaster, Colombo and Vicol \cite{BCV18} were  able to connect two arbitrary strong solutions via a weak solution. Burczak, Modena and Sz{\'e}kelyhidi~Jr. \cite{BMS20} then obtained a number of new ill-posedness results for power-law fluids and in particular also non-uniqueness of weak solutions to the Navier--Stokes equations for every given divergence free initial condition in $L^{2}$. Sharp non-uniqueness results for the Navier--Stokes equations in dimension $d\geq 2$ were obtained by Cheskidov and Luo \cite{CL20, CL21}.
We  refer to the  reviews \cite{BV19, BV20} for further details and references.

All these convex integration results very much rely on the $L^{2}$-setting. Namely, the constructed solutions belong (at least) to $L^{2}$ and the iteration converges strongly in $L^{2}$, hence  one can pass to the limit in the quadratic nonlinearity. As the energy inequality is available in this setting, it is also understood as a natural selection criterion for physical solutions. Convex integration yields  such  solutions  for Euler equations (see \cite{DelSze3}, and \cite{HZZ20} for the stochastic setting), or when the diffusion is  weak,  e.g. for power-law fluids with small parameter $p$ (see \cite{BMS20}) and for hypodissipative Navier--Stokes equations for small $\alpha$ (see \cite{CDD18}). However, constructing Leray solutions by convex integration to the Navier--Stokes system seems  to be out of reach at the moment. By a different method,   a first non-uniqueness result for Leray solutions was  established recently by Albritton, Bru\'e and Colombo \cite{ABC21} for the Navier--Stokes system with a force.

Compared to the  classical uniform estimates and the compactness argument, convex integration provides a new way of constructing solutions. This turns out to be  particularly useful in the stochastic setting as uniqueness of Leray solutions is unknown and there has been no result of existence of global probabilistically strong solutions before. In \cite{HZZ21}, we proved  such a result for a trace-class noise by  convex integration.
In less regular settings, as for instance for $\gamma\leq 1/2$ discussed above, there are no Leray solutions to compete with in the first place. Furthermore, there are no alternative globally defined solutions whatsoever (neither probabilistically strong nor probabilistically weak).
In this paper, we use convex integration  to construct global probabilistically strong solutions in this setting when the energy inequality is out of reach.
The question of how to select physical solutions in this case remains   open.

\subsection{Decomposition}\label{sec:1.3}

We introduce a decomposition of the Navier--Stokes system \eqref{main1}, which makes also this singular setting amenable to convex integration. Recall that solutions are only expected to have regularity $-1/2-\kappa$, $\kappa>0$, hence the quadratic term $u\otimes u$ is far from being well-defined analytically. The common idea in the field of singular SPDEs is to prescribe a particular form of a solution $u$ so that the nonlinearity can be  made sense of. In the first step,  we write
$$
u=z+z^{\<20>}+h.
$$
The first term $z$ solves the stochastic heat equation
$$
\dif z +\nabla p^{z}\, \dif t = \Delta z\, \dif t +\dif B,\qquad \div z=0,
$$
and permits to isolate the most irregular part of $u$, the rest being more regular. Note that by Schauder estimates, $z$ belongs to $B^{-1/2-\kappa}_{\infty,\infty}$ and hence the product $z\otimes z$ is not well defined in the classical sense. As $z$ is Gaussian, we can understand $z\otimes z$ as a Wick product using renormalization. In particular, for a suitable mollification $z_{\varepsilon}$ of $z$, there are diverging constants $C_{\varepsilon}\in \mathbb{R}^{3\times 3}$, $C^{i j}_{\varepsilon}\to\infty$, so that
$$
z^{\<2>}=\lim_{\varepsilon\to 0} z_{\varepsilon}\otimes z_{\varepsilon}-C_{\varepsilon}
$$
is well defined and belongs to $B^{-1-\kappa}_{\infty,\infty}$. More details of the probabilistic constructions are included in Section~\ref{s:stoch}.
In order to isolate the corresponding (still irregular) part of the solution $u$, we then define
$$
\partial_{t} z^{\<20>} + \div(z^{\<2>})+\nabla p^{2} = \Delta z^{\<20>} ,\qquad \div z^{\<20>}=0,
$$
which by Schauder estimates belongs to $B^{-\kappa}_{\infty,\infty}$.

It is then seen that the remainder $h$ is already function valued but its regularity is necessarily limited by $1/2-\kappa$. Even further decomposition cannot improve this regularity since products of  $z$ with the unknown always appear in the equation. A way how to overcome this issue stems from the work by Gubinelli, Imkeller and Perkowski \cite{GIP15} on the parabolic Anderson model and was applied to \eqref{main1}  by Zhu, Zhu in \cite{ZZ15}: one postulates a paracontrolled ansatz which describes a further structure of the solution $h$. It reads as
$$
h=-\mathbb{P} [h \prec  \mathcal{I} \nabla z] +
\vartheta -(z^{\<21l1>}+z^{\<211>}).
$$
Here, $z^{\<21l1>}$, $z^{\<211>}$ are additional stochastic objects constructed by renormalization, $\mathcal{I}$ is the heat operator, $\mathbb{P}$ the Helmholtz projection and finally $\prec$ denotes a paraproduct as introduced by Bony \cite{Bon81}. This permits to cancel the two most irregular terms in the equation for $h$ so that the remainder $\vartheta$ becomes more regular, namely,  $1-\kappa$. Furthermore, by a commutator lemma it permits to make sense of the analytically ill-defined product $h\otimes z$. We refer to Section~\ref{s:para} for basic definitions of paracontrolled calculus and to Section~\ref{s:sol} for more details on the notion of our paracontrolled solution. In view of the  regularity of solutions,  we also see that energy inequality is impossible in this case.

The above decomposition is sufficient to prove  local well-posedness as done in \cite{ZZ15}.
However, a much more refined analysis is indispensable to apply convex integration. Therefore, we split further $h=v^{1}+v^{2}$ where $v^{1}$ represents the irregular part and $v^{2}$ the regular one. In addition, the equation for $v^{1}$ is linear whereas the one for $v^{2}$ contains the quadratic nonlinearity. Similarly to the above discussion of the more regular cases \eqref{main3} with $\gamma\leq 1/2$, even with this decomposition into $v^{1}+v^{2}$, it is not possible to derive global estimates via the energy method.
Our idea is instead to apply convex integration on the level of $v^{2}$. However, the equation for $v^{2}$ is  coupled with the equation for $v^{1}$. Therefore, we put forward a joint iterative procedure approximating both equations at once. The~Reynolds stress $\mathring{R}_{q}$ is only included in the equation for $v_{q}^{2}$, where $q\in\mathbb{N}_{0}$ is the iteration parameter. Consequently, the construction of the new iteration $v^{2}_{q+1}$ relies only on the previous stress $\mathring{R}_{q}$. Here, we employ the intermittent jets by Buckmaster, Colombo and Vicol \cite{BCV18} (see also \cite{BV19} and our previous works \cite{HZZ19,HZZ21}). As the next step, we solve the equation for $v_{q+1}^{1}$ exactly by a fixed point argument. See Figure~\ref{f:1} for a sketch of our procedure.

 In order to make this strategy possible, it is necessary to find the decomposition of the equation for $h$ into the system for $v^{1}$ and $v^{2}$ and  to define the corresponding equations for the iterations $v^{1}_{q}$ and $v^{2}_{q}$. This together with the construction of each approximate velocity $v^{2}_{q+1}$ through the intermittent jets has to be done in a way to decrease the corresponding Reynolds stress $\mathring{R}_{q+1}$ as $q\to\infty$. Especially the control of $\mathring{R}_{q+1}$ requires a careful analysis of each of the terms appearing in the equation for $h$. We have to balance various competing factors such as regularity, integrability, blow-up as $t\to0$ and blow-up as $q\to\infty$ of various terms. The divergencies need to be compensated by small quantities. We rely on a decomposition of each product into the two paraproducts and the resonant term, because each of these parts behaves differently and requires a different treatment. Roughly speaking, irregular terms are included into $v^{1}$ while regular ones into $v^{2}$, but the precise splitting is delicate. Further issues and required ideas are summarized as follows:
\begin{itemize}

\item In the first step, we aim at constructing convex integration solutions up to a stopping time. The stopping time can be chosen arbitrarily large with arbitrarily large probability and  is used  to have uniform in $\omega$ bounds for the stochastic objects. The reason for this is adaptedness: without a stopping time, the parameters would depend on the whole path and the constructed solutions could not be  adapted to the given filtration $(\mathcal{F}_{t})_{t\geq0}$. Thus, they would not be probabilistically strong.

\item In the second step, we  overcome this limitation by extending the constructed solutions by other convex integration solutions. To this end, it is necessary to obtain convex integration solutions for any given divergence free initial condition in $L^{2}$. Hence, we aim at starting $v^{1}$ as well as $v^{1}_{q}$ from the given initial condition, whereas $v^{2}$ and $v^{2}_{q}$ all  start from zero. This simplifies the begin of the iterative procedure. To keep the same initial value during the iteration, the oscillations can only be added for positive times and we approach $t=0$ as $q\to\infty$.

\item A paracontrolled ansatz for $v^{1}$ and accordingly also for each $v^{1}_{q}$ needs to be included in order to make sense of the resonant part of the product $v^{1}\otimes z$ and $v^{1}_{q}\otimes z$. At each iteration step $q$, the equation for $v^{1}_{q}$ is coupled with the equation for the corresponding remainder $v^{\sharp}_{q}$ (taking the role of $\vartheta$ above).

\item Since the initial value of $v^{1}$, $v^{1}_{q}$ and $v^{\sharp}_{q}$ is only in $L^{2}$, they have singularity at time zero when considered in more regular function spaces. But higher regularity is indispensable in order to control various terms, both in the equation for $v^{1}_{q}$ and $v^{\sharp}_{q}$ but also in the formula for $\mathring{R}_{q+1}$. For this reason, we work with blow-up norms in time, but this brings an extra blow-up  in the convex integration part and in particular in the estimate of $\mathring{R}_{q+1}$.

\item We introduce additional (uniform in $q$) localizers $\Delta_{>R}+\Delta_{\leq R}$ in terms of Littlewood--Paley blocks.
The~part $\Delta_{>R}$ is always included into $v^{1}$ and helps to simplify the estimates of $v^{1}_{q}$ and $v^{\sharp}_{q}$ as it provides an arbitrarily  small constant and  avoids  the need for a Gronwall lemma.

\item Some terms are regular and could be, in principle, included into $v^{2}$, but they require  regularity of $v_{q}^{1}$ which does not hold  true uniformly as $t\to 0$. While the fixed point equation for $v^{1}_{q}$ can be solved using suitable blow-up norms in time to overcome the singularity at $t=0$, having such blow-ups in the equation for $v^{2}_{q}$ cannot be controlled in the convex integration. Hence,  we further decompose these terms by using $\Delta_{>R}+\Delta_{\leq R}$ and include their irregular parts into $v_{q}^{1}$.

\item Moreover, in order to control the blow-up of certain terms, we include $q$-dependent localizers $\Delta_{\leq f(q)}$ for a suitable $f(q)\to\infty$ into the equation for $v^{1}_{q}$. Thanks to this, we are able to add irregularity scale by scale in a controlled way. The opposite approach, namely including the irregular terms fully  at the beginning of the iteration, does not seem to be  possible.

\end{itemize}

The detailed decomposition is presented in Section~\ref{s:dec}. As a preparation, we explain the main ideas on the simpler settings of \eqref{main3} with $\gamma>0$ in Section~\ref{s:reg}.

\subsection{Final remarks}

Previously, convex integration has always been  used to deduce non-uniqueness of solutions in  settings where energy inequality is available.  It has also been used to obtain first existence results for weak solutions   in situations where compactness does not guarantee the passage to the limit in  the convective term,  namely, for Euler equations (see \cite{DelSze2,DelSze3,DelSze13}), and  for power-law fluids with small parameter $p$ (see \cite{BMS20}).  As already mentioned above, these results even lead to infinitely many weak solutions satisfying the energy inequality.

Our result shows that convex integration can also be used to construct global solutions when the energy inequality is out of reach, in particular in the field of singular SPDEs. We  hope that our technique could also be applied  to other singular SPDEs, especially in  cases where no strong damping  is at hand. We also mention that the convex integration can be applied to other PDEs like transport and continuity equation \cite{MS17, MS18},  3D Hall--MHD system \cite{Da18} and hyperviscous Navier--Stokes equations \cite{LT18}. Our approach may also be applied to the corresponding singular versions of the latter two.

The nonlinearity in the Navier--Stokes system looks similarly to the one in the Langevin dynamic for the Yang--Mills measure, i.e. stochastic quantization of the Yang--Mills field. This was  considered in \cite{She18, CCHS20} where local-in-time solutions were constructed. However,  existence of global solutions remains open. The idea  is to use the dynamics and PDE techniques to study properties of the field. Formally, these equations have the law of the associated field as an invariant measure. In the case of the stochastic quantization of the Euclidean $\Phi^{4}$ field theory, it was indeed possible to use the dynamics to construct and study  properties of the corresponding measure (see \cite{GH18a} and \cite{SZZ21a}). As global existence for the Langevin dynamic for the Yang--Mills measure is out of reach by the classical PDE techniques, we hope that our technique can shed some light on this problem.

Finally, we point out that in the field of regularization by noise, it is believed that more noise, in the sense of more irregular noise, in the Navier--Stokes equations may imply uniqueness. Certain regularizing effect could indeed be proved on the level of the so-called strong Feller property established for \eqref{main1} by Zhu, Zhu \cite{ZZ17}.
Nevertheless, our results show that even in this case we still have non-uniqueness. Furthermore, also a weaker notion of uniqueness, namely, uniqueness in law is disproved.

\subsubsection*{Organization of the paper} In Section~\ref{sec:pre} we introduce our notation and present preliminary results on Besov spaces, paraproducts and paracontrolled calculus. Section~\ref{s:reg} is devoted to the more regular settings of \eqref{main3} with $\gamma>0$ and we discuss the main ideas of our decomposition. In particular, it is shown how decreasing the parameter $\gamma$, i.e. making the problem more irregular, necessarily requires further ideas and a more refined decomposition. In Section~\ref{s:pa} we recall the construction of stochastic objects and introduce the notion of paracontrolled solution. We also present a formal decomposition of the system into the system for $v^{1}$ and $v^{2}$ as discussed above. The set-up of the iterative convex integration procedure and proofs of our main results are shown in Section~\ref{s:con}. We give estimates of $v^1_q$ and $v^\sharp_q$   in Section \ref{s:v1vs}. Section~\ref{ss:it2} is devoted to the core of the convex integration construction, namely, the iteration Proposition~\ref{p:iteration p}. Finally, in Appendix~\ref{s:B} we recall the construction of intermittent jets and in Appendix~\ref{s:sch} we prove auxiliary Schauder estimates.

\section{Preliminaries}
\label{sec:pre}

Throughout the paper, we use the notation $a\lesssim b$ if there exists a constant $c>0$ such that $a\leq cb$, and we write $a\simeq b$ if $a\lesssim b$ and $b\lesssim a$.

\subsection{Function spaces}
Given a Banach space $E$ with a norm $\|\cdot\|_E$ and $T>0$, we write $C_TE=C([0,T];E)$ for the space of continuous functions from $[0,T]$ to $E$, equipped with the supremum norm $\|f\|_{C_TE}=\sup_{t\in[0,T]}\|f(t)\|_{E}$.  For $p\in [1,\infty]$ we write $L^p_TE=L^p([0,T];E)$ for the space of $L^p$-integrable functions from $[0,T]$ to $E$, equipped with the usual $L^p$-norm.
We use $(\Delta_{i})_{i\geq -1}$ to denote the Littlewood--Paley blocks corresponding to a dyadic partition of unity.
Besov spaces on the torus with general indices $\alpha\in \R$, $p,q\in[1,\infty]$ are defined as
the completion of $C^\infty(\mathbb{T}^{d})$ with respect to the norm
$$
\|u\|_{B^\alpha_{p,q}}:=\left(\sum_{j\geq-1}2^{j\alpha q}\|\Delta_ju\|_{L^p}^q\right)^{1/q}.$$
The H\"{o}lder--Besov space $C^\alpha$ is given by $C^\alpha=B^\alpha_{\infty,\infty}$
and we also set $H^\alpha=B^\alpha_{2,2}$, $\alpha\in \mathbb{R}$.
To deal with the singularity at time zero we introduce the following blow-up norms: for $\alpha\in(0,1),$ $p\in[1,\infty]$
$$
\|f\|_{C_{T,\gamma}^{\alpha}L^p}:=\sup_{0\leq t\leq T}t^{\gamma}\|f(t)\|_{L^{p}}+\sup_{0\leq s<t\leq T}s^\gamma \frac{\|f(t)-f(s)\|_{L^p}}{|t-s|^\alpha},
$$
$$
 \|f\|_{C_{T,\gamma}B^\alpha_{p,\infty}}:=\sup_{0\leq t\leq T}t^\gamma \|f(t)\|_{B^\alpha_{p,\infty}}.$$
For $T>0$ and a domain $D\subset\R^{+}$ we denote by  $C^{N}_{T,x}$ and $C^{N}_{D,x}$, respectively, the space of $C^{N}$-functions on $[0,T]\times\mathbb{T}^{3}$ and on $D\times\mathbb{T}^{3}$, respectively,  $N\in\N_{0}:=\N\cup \{0\}$. The spaces are equipped with the norms
$$
\|f\|_{C^N_{T,x}}=\sum_{\substack{0\leq n+|\alpha|\leq N\\ n\in\N_{0},\alpha\in\N^{3}_{0} }}\|\partial_t^n D^\alpha f\|_{L^\infty_T L^\infty},\qquad \|f\|_{C^N_{D,x}}=\sum_{\substack{0\leq n+|\alpha|\leq N\\ n\in\N_{0},\alpha\in\N^{3}_{0} }}\sup_{t\in D}\|\partial_t^n D^\alpha f\|_{ L^\infty}.
$$

Set $\Lambda= (1-\Delta)^{{1}/{2}}$. For $s\geq0$, $p\in [1,+\infty]$ we use $W^{s,p}$ to denote the subspace of $L^p$, consisting of all  $f$   which can be written in the form $f=\Lambda^{-s}g$, $g\in L^p$ and the $W^{s,p}$ norm of $f$ is defined to be the $L^p$ norm of $g$, i.e. $\|f\|_{W^{s,p}}:=\|\Lambda^s f\|_{L^p}$. For $s<0$, $p\in (1,\infty)$, $W^{s,p}$ is the dual space of $W^{-s,q}$ with $\frac{1}{p}+\frac{1}{q}=1$.

The following embedding results will  be frequently used (we refer to e.g.  \cite[Lemma~A.2]{GIP15} for the first one and to \cite[Theorem 4.6.1]{Tri78} for the second one).

\bl\label{lem:emb}
~
\begin{enumerate}
\item Let $1\leq p_1\leq p_2\leq\infty$ and $1\leq q_1\leq q_2\leq\infty$, and let $\alpha\in\mathbb{R}$. Then $B^\alpha_{p_1,q_1} \subset B^{\alpha-d(1/p_1-1/p_2)}_{p_2,q_2}$.

\item Let $s\in \R$, $1<p<\infty$, $\epsilon>0$. Then $W^{s,2}=B^s_{2,2}=H^s$, and
$B^s_{p,1}\subset W^{s,p}\subset B^{s}_{p,\infty}\subset B^{s-\epsilon}_{p,1}$.
\end{enumerate}

\el

\subsection{Paraproducts, commutators and localizers}
\label{s:para}

Paraproducts were introduced by Bony in \cite{Bon81} and they permit to decompose a product of two distributions into three parts which behave differently in terms of regularity. More precisely, using the Littlewood-Paley blocks, the product $fg$ of two Schwartz distributions $f,g\in\mathcal{S}'(\mathbb{T}^{d})$ can be formally decomposed as
$$fg=f\prec g+f\circ g+f\succ g,$$
with $$f\prec g=g\succ f=\sum_{j\geq-1}\sum_{i<j-1}\Delta_if\Delta_jg, \quad f\circ g=\sum_{|i-j|\leq1}\Delta_if\Delta_jg.$$
Here, the paraproducts $\prec$ and $\succ$ are always well-defined and critical is the resonant product denoted by $\circ$.
In general, it is only well-defined provided  the sum of the regularities of $f$ and $g$ in terms of Besov spaces is strictly positive.
Moreover, we have the following paraproduct estimates from \cite{Bon81} (see also \cite[Lemma~2.1]{GIP15},  \cite[Proposition~A.7]{MW18}).

\begin{lemma}\label{lem:para}
	Let  $\beta\in\R$, $p, p_1, p_2, q\in [1,\infty]$ such that $\frac{1}{p}=\frac{1}{p_1}+\frac{1}{p_2}$. Then  it holds
	\begin{equation*}
	\|f\prec g\|_{B^\beta_{p,q}}\lesssim\|f\|_{L^{p_1}}\|g\|_{B^{\beta}_{p_2,q}},
	\end{equation*}
	and if $\alpha<0$ then
	\begin{equation*}
	\|f\prec g\|_{B^{\alpha+\beta}_{p,q}}\lesssim\|f\|_{B^{\alpha}_{p_1,q}}\|g\|_{B^{\beta}_{p_2,q}}.
	\end{equation*}
	If  $\alpha+\beta>0$ then it holds
	\begin{equation*}
	\|f\circ g\|_{B^{\alpha+\beta}_{p,q}}\lesssim\|f\|_{B^{\alpha}_{p_1,q}}\|g\|_{B^{\beta}_{p_2,q}}.
	\end{equation*}
\end{lemma}

We denote $\succcurlyeq\,=\circ\, +\!\succ$, $\preccurlyeq \,=\circ\, +\!\prec.$
The key tool of the paracontrolled calculus introduced in \cite{GIP15} is  the following commutator lemma from \cite[Lemma 2.4]{GIP15} (see also  \cite[Proposition~A.9]{MW18}).

\begin{lemma}\label{lem:com1}
	Assume that $\alpha\in (0,1)$ and $\beta,\gamma\in \R$ are such that $\alpha+\beta+\gamma>0$ and $\beta+\gamma<0$ and   $p,p_1,p_2\in [1,\infty]$ satisfy $\frac{1}{p}=\frac{1}{p_1}+\frac{1}{p_2}$. Then there exist a bounded trilinear  operator
	$$
	 \mathrm{com}(f,g,h):B^\alpha_{p_1,\infty}\times C^\beta\times B^\gamma_{p_2,\infty}\to B^{\alpha+\beta+\gamma}_{p,\infty}
	 $$ satisfying
	$$
	\|\mathrm{com}(f,g,h)\|_{B^{\alpha+\beta+\gamma}_{p,\infty}}\lesssim \|f\|_{B^\alpha_{p_1,\infty}}\|g\|_{C^\beta}\|h\|_{B^\gamma_{p_2,\infty}}
	$$
	such that for smooth functions $f,g,h$ it holds
	$$
	\mathrm{com}(f,g,h)=(f\prec g)\circ h - f(g\circ h).
	$$
\end{lemma}

We also recall the following two lemmas for the Helmholtz projection $\mathbb{P}$ from \cite[Lemma 3.5, Lemma 3.6]{ZZ15}.

\begin{lemma}\label{lem:com2}
	Assume that $\alpha\in (0,1), \beta\in \R$ and   $p\in [1,\infty]$. Then for every $k,l=1,2,3$
	$$
	\|[\mathbb{P}^{kl}, f\prec] g \|_{B^{\alpha+\beta}_{p,\infty}}\lesssim \|f\|_{B^\alpha_{p,\infty}}\|g\|_{C^\beta}.
	$$
\end{lemma}

\begin{lemma}\label{lem:Leray}
	Assume that $\alpha\in  \R$ and   $p\in [1,\infty]$. Then for every $k,l=1,2,3$
	$$
	\|\mathbb{P}^{kl}f \|_{B^{\alpha}_{p,\infty}}\lesssim \|f\|_{B^\alpha_{p,\infty}}.
	$$
\end{lemma}

Analogously to the  the real-valued case, we may define paraproducts for vector-valued distributions. More precisely, for two vector-valued distributions $f,g\in \mathcal{S}'(\mathbb{T}^{d};\mathbb{R}^{m})$, we use the following tensor paraproduct notation
$$
\begin{aligned}
f\otimes g = (f_{i} g_{j})_{i,j=1}^{m}& =f\oprec g+ f\varocircle g +f\osucc g \\
&= (f_{i}\prec g_{j})_{i,j=1}^{m}+(f_{i}\circ g_{j})_{i,j=1}^{m}+(f_{i}\succ g_{j})_{i,j=1}^{m}
\end{aligned}
$$
and note that Lemma~\ref{lem:para} carries over mutatis mutandis. We also denote
$$
\osucccurlyeq=\varocircle +\osucc,\qquad\opreccurlyeq =\varocircle +\oprec.
$$
When there is no danger of confusion, we apply the simple
paraproducts also within matrix-vector multiplication, i.e. for $f \in
\mathcal{S}' (\mathbb{T}^d ; \mathbb{R}^{m \times m})$ and $g \in \mathcal{S}'
(\mathbb{T}^d ; \mathbb{R}^m)$ we define using the Einstein summation convention
\[ (f \succ g)_{i=1}^{m} = (g \prec f)_{i=1}^{m} = \left( f^{i j} \succ g^j\right)_{i=1}^{m}, \qquad (f \circ g)_{i=1}^{m} =
   (g \circ f)_{i=1}^{m} = \left( f^{i j} \circ g^j \right)_{i=1}^{m}. \]
Similarly to Lemma~\ref{lem:com1}, we may also define a matrix-valued commutator as  continuous extensions of
\[ (f,g,h)\mapsto (\tmop{com} (f, g, h))_{i,j=1}^{m} = (f \prec g) \varocircle h - f\cdot (g
   \varocircle h) = (f^k \prec g^{i k}) \circ h^j - f^k (g^{i k} \circ h^j) ,\]
\[ (f,g,h)\mapsto(\tmop{com}^{\ast} (f, g, h))_{i,j=1}^{m} = h \varocircle (f \prec g) - (h
   \varocircle g) \cdot f = h^i \circ (f^k \prec g^{j k}) - (h^i \circ g^{j k}) f^k,
\]
which is well-defined for smooth functions $f, h :\mathbb{T}^d \to \mathbb{R}^m$, $g :\mathbb{T}^d \to \mathbb{R}^{m \times m}$ and takes values in $\mathcal{S}' (\mathbb{T}^d ; \mathbb{R}^{m \times m})$. A counterpart of the bound in Lemma~\ref{lem:com1} holds true in this setting as well.

Finally, we introduce  localizers in terms of Littlewood--Paley expansions. Let $J\in\mathbb{N}_{0}$. For $f\in\mathcal{S}'(\mathbb{T}^{d})$ we define
$$
\begin{aligned}
\Delta_{> J}f=\sum_{j> J}\Delta_{j}f,\qquad \Delta_{\leq J}f=\sum_{j\leq J}\Delta_{j}f.
\end{aligned}
$$
Then it holds in particular for $\alpha\leq \beta\leq\gamma$
\begin{equation}\label{eq:loc}
\begin{aligned}
\|\Delta_{> J}f\|_{C^{\alpha}}\lesssim 2^{-J(\beta-\alpha)}\|f\|_{C^{\beta}},\qquad
\|\Delta_{\leq J}f\|_{C^{\gamma}}\lesssim 2^{J(\gamma-\beta)}\|f\|_{C^{\beta}}.
\end{aligned}
\end{equation}

\section{More regular stochastic perturbations}
\label{s:reg}

In our previous works {\cite{HZZ19,HZZ21}} we considered the case of a trace
class noise. We proved non-uniqueness in law as well as non-uniqueness of
Markov solutions and existence and non-uniqueness of global-in-time
probabilistically strong solutions. These results apply to the
Navier--Stokes system
\begin{equation}
  \mathd u + \tmop{div} (u \otimes u) \mathd t + \nabla p \mathd t = \Delta u
  \mathd t + (- \Delta)^{- \gamma / 2} \mathd B, \qquad \tmop{div} u = 0,
  \label{eq:gamma}
\end{equation}
where $B=(B^1,B^2,B^3)$ is a vector-valued $L^2$-cylindrical Wiener process on some stochastic basis $(\Omega,
\mathcal{F}, (\cF_t)_{t\geq0}, \tmmathbf{P})$ and $\gamma > 3 / 2$. Here the filtration $(\cF_t)_{t\geq0}$ is the normal filtration generated by the Wiener process $B$.  As a main result of the
present paper, we treat the case of space-time white noise, i.e. $\gamma = 0$.
But already the more regular case of $\gamma \in (0, 3 / 2]$ presents
interesting new challenges. In this section we want to outline some of the
main ideas on examples of these more regular noises.

If the noise is not trace class, It{\^o}'s formula cannot be applied in order
to obtain an energy inequality. The case of $\gamma = 3 / 2$ is therefore the
treshold where (and below which) stochastic counterparts of Leray solutions on the level of $u$
no longer make sense.  Furthermore, the expected regularity of solutions depends on
the regularity of the noise. This can be seen by looking at the linear
counterpart of \eqref{eq:gamma}
\begin{equation}
  \mathd z + \nabla p^z \mathd t = \Delta z \mathd t + (- \Delta)^{- \gamma /
  2} \mathd B, \qquad \tmop{div} z = 0, \label{eq:zz}
\end{equation}
and realizing that $z \in C_T C^{\gamma - 1 / 2
- \kappa} \cap C^{1 / 2 - \delta}_T C^{\gamma - 3 / 2 + 2 \delta -
\kappa}$ $\mathbf{P}$-a.s. for $\kappa, \delta > 0$ small. This can be obtained by Schauder estimates from the fact that  the space-time white noise $\mathrm{d} B/\mathrm{d} t$ is a random distribution of space-time regularity $-5/2-\kappa$ for $\kappa>0$  with parabolic scaling.  Thus, if $\gamma \leq 1 / 2$ the
solution is not even function-valued and the quadratic nonlinearity in
\eqref{eq:gamma} is not well-defined analytically. Nevertheless,   one can use
probability theory and renormalization to define this nonlinearity using
Wick products.

We distinguish the following cases with an increasing level of difficulty:
\begin{enumerate}
  \item $\gamma \in (1 / 2, 3 / 2]$: solutions are function-valued hence the
  convective term is well-defined but the energy inequality for $u$ cannot be
  obtained.

  \item $\gamma \in (1 / 6, 1 / 2]$: solutions become distribution-valued,
  renormalization is needed to define the product.

  \item $\gamma \in (0, 1 / 6]$: further decomposition is required in order to
  make sense of all the required products.

  \item $\gamma=0$: the case of space-time white noise, a so-called paracontrolled ansatz is required and we present a detailed proof in  subsequent sections.
\end{enumerate}

We stress that while global existence of martingale (i.e. probabilistically weak) solutions in the case $\gamma \in (1 / 2, 3 / 2]$ can be obtained by compactness and Skorokhod representation theorem, existence of probabilistically strong solutions was unknown. For $\gamma\leq 1/2$ even global existence of martingale solutions was an open problem.

In the remainder of this section, we focus on the first three regimes and explain how
to make them amenable to convex integration and hence to global-in-time
existence and non-uniqueness results.

\subsection{The case of $\gamma \in (1 / 2, 3 / 2]$}
\label{s:3.1}

Even though the energy inequality cannot be computed here, the approach of
{\cite{HZZ19,HZZ21}} can be applied with minimal modifications. In particular,
one uses the decomposition $u = z + v$ where $z$ solves \eqref{eq:zz} and
\begin{equation}
  \mathcal{L}v + \tmop{div} ((v + z) \otimes (v + z)) + \nabla p^v = 0, \qquad
  \tmop{div} v = 0. \label{eq:vv}
\end{equation}
Here and in the sequel, we use the notation $\mathcal{L}= \partial_t -
\Delta$. The convex integration scheme is then given via the iteration system
\[ \mathcal{L}v_q + \tmop{div} ((v_q + z_q) \otimes (v_q + z_q)) + \nabla p^q
   = \tmop{div} \mathring{R}_q, \qquad \tmop{div} v_q = 0, \]
where $z_q = \Delta_{\leqslant f (q)} z$ with $\Delta_{\leqslant f (q)}$ being
a  cut-off of the Littlewood--Paley expansion. With a suitable
definition of the stopping times, this permits to add noise scale by scale as
one proceeds through the iteration. Setting the convex integration up with an
initial iteration $v_0$ as in {\cite{HZZ19}}, one can prove that the
constructed solution $v$ is distinct from the Leray solution to \eqref{eq:vv}
starting from the same initial value. This implies existence of an initial
value so that non-uniqueness holds for \eqref{eq:gamma} as well. Using the
probabilistic extension of solutions from {\cite{HZZ19}}, non-uniqueness in
law on any time interval $[0, T]$, $T > 0$, follows. Following the ideas of
{\cite{HZZ21}} it is also possible to construct solutions with a prescribed
$L^2$ initial condition and to obtain existence and non-uniqueness of
global-in-time probabilistically strong solutions.

\subsection{The case of $\gamma \in (1 / 6, 1 / 2]$}

This case is more delicate and further decomposition is required. More
precisely, in addition to $u = z + v$ as above, we split $v$ into its
irregular and regular part, i.e. $v = v^1 + v^2$. The equation for $v^1$
contains all the irregular terms of the product $(v + z) \otimes (v + z)$,
whereas the regular ones are put in $v^2$. Additionally, the equation for $v^{1}$ shall be linear so that it can be solved by a fixed point argument.
As a rule of thumb, we color the
irregular terms \tmcolor{magenta}{magenta} and the regular ones
\tmcolor{blue}{blue}. The decomposition can be done as follows. The product $z
\otimes z$ needs to be constructed by renormalization as a Wick product denoted by $z^{\<2>}$ and it is of spatial
regularity $C^{2 (\gamma - 1 / 2) - \kappa}$. For the moment, we ignore this fact and proceed formally. We come back to the rigorous definition of the stochastic objects in Section~\ref{s:stoch}.

So we have  the first magenta term
$\tmcolor{magenta}{z \otimes z}$. Then we write with the help of paraproducts
and Littlewood--Paley projectors
\[ (v^1 + v^2) \otimes z = \tmcolor{magenta}{(v^1 + v^2) \oprec \Delta_{> R}
   z} \tmcolor{blue}{+ (v^1 + v^2) \osucccurlyeq \Delta_{> R} z}
   \tmcolor{blue}{+ (v^1 + v^2) \otimes \Delta_{\leqslant R} z}, \]
and treat the symmetric term $z\otimes (v^{1} +v^{2})$ the same way. Here, $\osucccurlyeq=\varocircle +\osucc $ and we included a suitable
cut-off $R$ to be chosen appropriately. This eventually simplifies the fixed point argument used to establish, for a given convex integration iteration $v^{2}_{q}$, the  existence and uniqueness of $v^{1}_{q}$. 
Finally, we let $\tmcolor{blue}{(v^1 + v^2) \otimes (v^1
+ v^2)}$. This leads to
\[ \mathcal{L}v^1 + \tmop{div} \left( \tmcolor{magenta}{z \otimes z + (v^1 +
   v^2) \oprec \Delta_{> R} z + \Delta_{> R} z \osucc (v^1 + v^2)} \right) +
   \nabla p^1 = 0, \qquad \tmop{div} v^1 = 0, \]
\[ \mathcal{L}v^2 + \tmop{div} \left( \tmcolor{blue}{(v^1 + v^2)\osucccurlyeq
   \Delta_{> R} z + \Delta_{> R} z \opreccurlyeq (v^1 + v^2) + (v^1 + v^2)
   \otimes \Delta_{\leqslant R} z} \right) \]
\[ + \tmop{div} (\tmcolor{blue}{\Delta_{\leq R} z \otimes (v^1 + v^2) + (v^1 +
   v^2) \otimes (v^1 + v^2)}) + \nabla p^2 = 0, \qquad \tmop{div} v^2 = 0. \]

We set up a convex integration scheme as an approximation of the above system
of equations for $v^1$ and $v^2$. In particular, we include further
Littlewood--Paley projectors and let
\[ \mathcal{L}v_q^1 + \tmop{div} \left( \tmcolor{magenta}{z \otimes z + (v_q^1
   + v_q^2) \oprec \Delta_{\leqslant f (q)} \Delta_{> R} z + \Delta_{\leqslant
   f (q)} \Delta_{> R} z \osucc (v_q^1 + v_q^2)} \right) + \nabla p_q^1 = 0,
   \qquad \tmop{div} v_q^1 = 0, \]
\[ \mathcal{L}v_q^2 + \tmop{div} \left( \tmcolor{blue}{(v_q^1 + v_q^2)
 \osucccurlyeq \Delta_{> R} z + \Delta_{> R} z \opreccurlyeq (v_q^1 + v_q^2)
   + (v_q^1 + v_q^2) \otimes \Delta_{\leqslant R} z} \right) \]
\[ + \tmop{div} (\tmcolor{blue}{\Delta_{\leq R} z \otimes (v_q^1 + v_q^2) +
   (v_q^1 + v_q^2) \otimes (v_q^1 + v_q^2)}) + \nabla p_q^2 = \tmop{div} \mathring{R}_q,
   \qquad \tmop{div} v_q^2 = 0. \]
Note that the Reynolds stress $\mathring{R}_q$ is only included in the equation for
$v^2_q$.  Indeed, $v^{2}_{q}$ contains the quadratic nonlinearity which is used in the convex integration to reduce the stress. 

At each iteration step $q + 1$, we first use the previous stress $\mathring{R}_q$ in
order to define the principle part of the corrector $w^{(p)}_{q + 1}$, the
incompressibility corrector $w^{(c)}_{q + 1}$ and the time corrector
$w^{(t)}_{q + 1}$ in terms of the intermittent jets, see Appendix~\ref{s:B}. This gives  the next
iteration $v^2_{q + 1}$ and consequently we obtain $v^1_{q + 1}$ by a fixed
point argument, cf. Figure~\ref{f:1}. The localizers $\Delta_{\leqslant f (q)}$ in the equation of
$v^1_q$ are used to control the blow up of a certain norm of $v^1_q$ as $q
\rightarrow \infty$.

\subsection{The case of $\gamma \in (0, 1 / 6]$}

In this regime, also the resonant product $v^1 \varocircle z$ becomes
problematic. This can be overcome by introducing an additional stochastic
object which permits to cancel the worst term, i.e. $z \otimes z$, from the
equation for $v^1$. To be precise, let
\begin{equation}
  \mathcal{L}z_1 + \tmop{div} (z \otimes z) + \nabla p^{z_1} = 0, \qquad
  \tmop{div} z_1 = 0, \label{eq:z1}
\end{equation}
and define $u = z + v = z + z_1 + v^1 + v^2$.
Recall that in the rigorous analysis $z \otimes z$ needs to be replaced by the Wick product $z^{\<2>}$ of regularity $C^{2(\gamma-1/2)-\kappa}$. Consequently, $z_{1}$ then becomes our second stochastic object, later denoted as $z^{\<20>}$ with regularity $C^{2\gamma-\kappa}$. In addition,  also the products $z_{1}\otimes z$ and $z\otimes z_{1}$ need to be defined via renormalization as $z^{\<21>}$ and $z^{\<21l>}$, respectively. We again ignore this fact for a moment and continue with the formal decomposition. The reader is referred to Section~\ref{s:stoch} below for more details on the stochastic construction.

Proceeding as above, this leads
to
\[ \mathcal{L}v^1 + \tmop{div} \left( \tmcolor{magenta}{z_1 \otimes z + z
   \otimes z_1 + z_1 \otimes z_1 + (v^1 + v^2) \oprec \Delta_{> R} z +
   \Delta_{> R} z \osucc (v^1 + v^2)} \right) + \nabla p^1 = 0, \]
\[ \qquad \tmop{div} v^1 = 0, \]
\[ \mathcal{L}v^2 + \tmop{div} \left( \tmcolor{blue}{(v^1 + v^2) \osucccurlyeq
   \Delta_{> R} z + \Delta_{> R} z \opreccurlyeq (v^1 + v^2) + (v^1 + v^2)
   \otimes (\Delta_{\leqslant R} z + z_1)} \right) \]
\[ + \tmop{div} (\tmcolor{blue}{(z_1 + \Delta_{\leq R} z) \otimes (v^1 + v^2) +
   (v^1 + v^2) \otimes (v^1 + v^2)}) + \nabla p^2 = 0, \qquad \tmop{div} v^2 =
   0, \]
 together with the convex integration scheme
\[ \mathcal{L}v_q^1 + \tmop{div} \left( \tmcolor{magenta}{z_1 \otimes z + z
   \otimes z_1 + z_1 \otimes z_1 + (v_q^1 + v_q^2) \oprec \Delta_{\leqslant f
   (q)} \Delta_{> R} z + \Delta_{\leqslant f (q)} \Delta_{> R} z \osucc (v_q^1
   + v_q^2)} \right) \]
\[ + \nabla p_q^1 = 0, \qquad \tmop{div} v_q^1 = 0, \]
\[ \mathcal{L}v_q^2 + \tmop{div} \left( \tmcolor{blue}{(v_q^1 + v_q^2)
   \osucccurlyeq \Delta_{> R} z + \Delta_{> R} z \opreccurlyeq (v_q^1 + v_q^2)
   + (v_q^1 + v_q^2) \otimes (\Delta_{\leqslant R} z + z_1)} \right) \]
\[ + \tmop{div} (\tmcolor{blue}{(\Delta_{\leqslant R} z + z_1) \otimes (v_q^1
   + v_q^2) + (v_q^1 + v_q^2) \otimes (v_q^1 + v_q^2)}) + \nabla p_q^2 =
   \tmop{div} \mathring{R}_q, \qquad \tmop{div} v_q^2 = 0. \]

With this definition, one can obtain existence of an initial condition for
which there are non-unique solutions before a stopping time. Since there are no global Leray solutions, for global existence it is necessary to extend these solutions by other convex integration solutions. To this end, an improved convex integration construction is needed, which gives the result for any
prescribed initial condition in $L^2$. In particular, the term
$$
v^{1}\osucccurlyeq\Delta_{>R}z
$$
and its symmetric counterpart appearing in the equation for $v^{2}$ require regularity of $v^{1}$. Accordingly, this  introduces blow-up norms in time to overcome the singularity at $t=0$ into the convex integration, which is not convenient. We therefore refine the decomposition above by writing
$$
(v^1 + v^2) \osucccurlyeq
   \Delta_{> R} z = \tmcolor{magenta}{ v^1  \osucccurlyeq
   \Delta_{> R} z } \tmcolor{blue}{+  v^2 \osucccurlyeq
   \Delta_{> R} z}.
   $$
   We  include these terms into the equations for $v^{1}$, $v^{2}$ accordingly respecting the colors. Then in the equations for $v^{1}_{q}$ and $v^{2}_{q}$ we rewrite
 $$
 (v_q^1 + v_q^2)
   \osucccurlyeq \Delta_{> R} z = \tmcolor{magenta}{ v_{q}^1  \osucccurlyeq
  \Delta_{> R} z } \tmcolor{blue}{+  v_{q}^2 \osucccurlyeq
   \Delta_{> R} z}.
   $$
The symmetric terms are treated the same way.

\section{Paracontrolled solutions}\label{s:pa}

It turns out that even further expansion would not help in order to treat the
case of space-time white noise, i.e. $\gamma = 0$. Indeed, there would always
be  ill-defined products. As understood in the field of singular SPDEs, a
paracontrolled ansatz needs to be included. It  postulates a   particular
structure of solutions and permits to make sense of the analytically ill-defined products using probabilistic tools.
In the sequel, we first introduce the stochastic objects needed for the rigorous formulation of the Navier--Stokes system \eqref{main1}. Then we formulate the notion of paracontrolled solution  incorporating the paracontrolled ansatz. Finally, we  give a formal decomposition combined with paracontrolled ansatz in the spirit of Section~\ref{s:reg}.

\subsection{Stochastic objects}\label{s:stoch}

Let us recall that due to Theorem 1.1 {\cite{ZZ15}}, the equation
\eqref{eq:gamma} with $\gamma = 0$ is locally well-posed for initial
conditions in $C^{\eta}$ for $\eta \in (- 1, - 1 / 2)$. The solution $u$
belongs to $C ([0, \sigma) ; C^{\eta})$ where $\sigma$ is a strictly positive
stopping time so that
\[ \| u_{\varepsilon} - u \|_{C_{\sigma} C^{\eta}} \rightarrow 0 \]
in probability. Here, $u_{\varepsilon}$ denotes the solution to the
regularized Navier--Stokes system
\[ \partial_t u_{\varepsilon} + \tmop{div} (u_{\varepsilon} \otimes
   u_{\varepsilon}) + \nabla p_{\varepsilon} = \Delta u_{\varepsilon} +
   \zeta_{\varepsilon}, \qquad \tmop{div} u_{\varepsilon} = 0, \]
where $\zeta_{\varepsilon}$ is a mollification of the space-time white noise
$\zeta = \mathd B / \mathd t$. In particular, the stochastic objects needed in
	our proof here were constructed in {\cite{ZZ15}}.
	
	 To summarize the main ideas, let
	$z_{\varepsilon}$ be the stationary solution to
	\[ \partial_t z_{\varepsilon} + \nabla p^{z_{\varepsilon}} = \Delta
	z_{\varepsilon} + \zeta_{\varepsilon}, \qquad \tmop{div} z_{\varepsilon} =
	0. \]
	Then $z_{\varepsilon} \rightarrow z$ in $L^p \left( \Omega ; C_T {C^{- 1 / 2 -
			\kappa}}  \right)$ for every $p \in [1, \infty)$. Moreover, the renormalized
	product $z^{\<2>} $ can be defined as a Wick product in the sense that there
	exist diverging constants $C_{\varepsilon} \in \mathbb{R}^{3 \times 3}$,
$C^{i j}_{\varepsilon} \rightarrow \infty$, so that
	\[ z^{\<2>}_\eps:=z_{\varepsilon} \otimes z_{\varepsilon} - C_{\varepsilon} \]
	has a well-defined limit in $L^p \left( \Omega ; C_T {C^{- 1 - \kappa}}
	\right)$ for every $p \in [1, \infty)$. In fact, $C_{\varepsilon}=\E[ z_{\varepsilon} \otimes z_{\varepsilon}]$.  We also introduce  the following stochastic objects. Let
	\begin{align*}
	\partial_t z_{\varepsilon}^{\<20>}  - \Delta
	z_{\varepsilon}^{\<20>} =-\mathbb{P}\div(z_\eps^{\<2>}), \qquad \tmop{div} z_{\varepsilon}^{\<20>} =
	0,\qquad z_{\varepsilon}^{\<20>}(0)=0,
	\end{align*}
	with $\mP$ being the Leray projection operator, and  define
	\begin{align*}
	z^{\<21>}_\eps&:=z^{\<20>}_\eps\otimes z_\eps,\qquad z^{\<21l>}_\eps:= z_\eps \otimes z^{\<20>}_\eps,\qquad z^{\<101>}_\eps:=\mathbb{P}\cI(\nabla z_\eps)\varocircle z_\eps=\left(\mathbb{P}\cI(\nabla z_\eps)^{ik}\circ z^{j}_\eps\right)_{i,j,k=1}^{3},
	\\z^{\<211>}_\eps&:=\mathbb{P}\cI\div(z^{\<21>}_\eps),\qquad z^{\<21l1>}_\eps:=\mathbb{P}\cI\div(z^{\<21>}_\eps),
	\\z^{\<2020>}_\eps&:=z^{\<20>}_\eps\otimes z^{\<20>}_\eps-C_{1,\eps}(t),\qquad z^{\<2111>}_\eps:= z^{\<211>}_\eps\varocircle z_\eps+z^{\<21l1>}_\eps\varocircle z_\eps-C_{2,\eps}(t),
	\end{align*}
where $\mathcal{I}f(t)=\int_{0}^{t}e^{(t-r)\Delta}f(r)\dif r$	and
$$
C_{1,\eps}(t)=\E[(z^{\<20>}_\eps\otimes z^{\<20>}_\eps)(t,0)]\to \infty,\qquad C_{2,\eps}(t)=\E[(z^{\<211>}_\eps\varocircle z_\eps)(t,0)]+\E [(z^{\<21l1>}_\eps\varocircle z_\eps)(t,0)] \to \infty
$$
as $\varepsilon\to 0$.

For connecting the solutions in Section \ref{s:con} we also introduce  two stochastic objects
$z_\eps^{\<2111>}(r;s)$ and $z^{\<101>}(r;s)$. They are defined the same way as $z_\eps^{\<2111>}$ and $z^{\<101>}$ but replacing the last integration operator $\cI$ by $\cI_{s,r}=\int_s^re^{(r-l)\Delta}\dif l$ and the renormalized diverging constants $C_{i,\eps}(r;s)$ are defined as the expectation of the corresponding terms for $i=1,2$.

		We recall the following result from \cite{ZZ15}.
	\bp\label{p:sto}
For every $\kappa>0$ and some $0<\delta<1/4$,  there exist random distributions
\begin{align}\label{e:defZ}
\mathbb{Z}:=(z,z^{\<2>},z^{\<20>}, z^{\<21>}, z^{\<21l>}, z^{\<2020>}, z^{\<2111>},z^{\<101>})
\end{align}
such that
if $\tau_{\varepsilon}$ is a component in
\begin{align*}
\mathbb{Z}_{\varepsilon}:=(z_{\varepsilon},z^{\<2>}_{\varepsilon},z^{\<20>}_{\varepsilon}, z^{\<21>}_{\varepsilon}, z^{\<21l>}_{\varepsilon}, z^{\<2020>}_{\varepsilon}, z^{\<2111>}_{\varepsilon},z^{\<101>}_{\varepsilon})
\end{align*}
and $\tau$ is
the corresponding component  in
$\mathbb{Z}$  then
$\tau_{\varepsilon }\rightarrow \tau$ in $ C_TC^{\alpha_{\tau}}\cap C_T^{\delta/2}C^{\alpha_{\tau}-\delta}$ a.s. as $\varepsilon\to0$, where the regularities $\alpha_{\tau}$ are summarized in Table~\ref{t:reg}.
Furthermore, for every $p\in[1,\infty)$
\begin{align*}
\sup_{\eps\geq0}\E\|\tau_\eps\|_{C_TC^{\alpha_{\tau}}}^p+\sup_{\eps\geq0}\E\|\tau_\eps\|_{C_T^{\delta/2}C^{\alpha_{\tau}-\delta}}^p
&\lesssim1.
\end{align*}
For $\tau_\eps(r;s)=z_{\eps}^{\<101>}(r;s)$ and $\tau_\eps(r;s)=z_{\eps}^{\<2111>}(r;s)$ there exist random distributions $\tau(r;s)=z^{\<101>}(r;s)$ and $\tau(r;s)=z^{\<2111>}(r;s)$ such that
\begin{align*}
\sup_{\eps\geq0}\E\sup_{0\leq s\leq r\leq T}\|\tau_\eps\|_{C^{\alpha_{\tau}}}^p
&\lesssim1,
\end{align*}
and \begin{align*}
\sup_{\eps\geq0}\E\sup_{0\leq s\leq r\leq T}\|\tau_\eps-\tau\|_{C^{\alpha_{\tau}}}^p
&\to0,
\end{align*}
as $\eps\to0$.
	\ep
	
		\begin{table}[ht]
	  \begin{center}
	\[ \begin{array}{|c|c|c|c|c|c|c|c|c|}
	\hline
	\tau & z & z^{\<20>} & z^{\<2>}  & z^{\<21>} & z^{\<21l>}& z^{\<2020>}  &
	z^{\<2111>} & z^{\<101>}\\
	\hline
	\alpha_{\tau} & - \frac{1}{2} - \kappa & - \kappa & - 1 - \kappa &
	- \frac{1}{2} - \kappa &- \frac{1}{2} - \kappa & - \kappa & - \kappa & - \kappa\\
	\hline
	\end{array} \]
	\caption{Regularity of stochastic objects.}
  \label{t:reg}
\end{center}
\end{table}

The renormalization of $\tau(r;s)$ can be done by a similar argument as \cite{ZZ15}.

	\br We emphasize that the renormalization constants $C_\eps, C_{i,\eps}$, $i=1,2$, only depend on $t$. Hence, due to the divergence in the nonlinear term, they do not appear in the approximate Navier--Stokes system driven by the mollified noise $\zeta_{\varepsilon}$.  If we modified the $C_\eps, C_{i,\eps}$, $i=1,2$, by adding a finite constant, some of the limit random distributions $\tau$ would change. For the rest of the paper, we fix the stochastic objects $\tau$ and prove existence of infinitely many solutions $h$ to  \eqref{solution} with these fixed stochastic data.
	\er

Finally, we note that similarly to our decomposition, the local
solution $u$ obtained in Theorem~1.1 in {\cite{ZZ15}} decomposes as $u=z+z_{1}+h$ with a suitable paracontrolled ansatz for $h$. In particular, by Schauder estimates, the part $h$
possesses a positive spatial regularity at positive times, i.e. it belongs in particular to
$L^2$. For this reason, it is sufficient to restrict our attention to initial
conditions for the convex integration scheme in $L^2$. Indeed, for an initial
condition in $C^{\eta}$, $\eta \in (- 1, - 1 / 2)$, we can always start with
the local solution from \cite{ZZ15} and start with convex integration only at a
positive (random) time.

\subsection{Notion of solution}
\label{s:sol}

To begin with, we write  formally
\[ u = z + v = z + z_1 + h, \]
where $z$ and $z_1$ solve \eqref{eq:zz} and \eqref{eq:z1}, respectively.
Then the equation for $h$ reads as
\begin{align}\label{solution}
\mathcal{L}h+\div((h+z_1)\otimes (h+z_1)+z\otimes (h+z_1)+(h+z_1)\otimes z)+\nabla p^h&=0,
\\\div h & = 0,\nonumber
\end{align}
which shall be further rewritten as a system for $v^{1}$ and $v^{2}$ in Section \ref{s:dec}.
 Indeed, the multiplication of $h$
and $z$ is not well-defined as the expected sum of their regularities is not
strictly positive for the resonant product $h \varocircle z$ to be
well-defined, cf. Lemma~\ref{lem:para}.
Collecting the terms which make $h$ too irregular leads to the following paracontrolled ansatz
\begin{align*}
h=-\mathbb{P} [h \prec  \mathcal{I} \nabla z] +
\vartheta -\tmcolor{orange}{\mathbb{P}\mathcal{I} \tmop{div} (z \otimes z_1 + z_1 \otimes
	z)},
\end{align*}
where $\cI f(t)=\int_0^tP_{t-s}f(s)\dif s$.

Then $\vartheta$ becomes more regular than $h$ since
\begin{equation}\label{eq:varth}
\vartheta=h+\tmcolor{orange}{\mathbb{P}\mathcal{I} \tmop{div} (z \otimes z_1 + z_1 \otimes
	z)}+\mathbb{P}\mathcal{I}[h\prec \nabla z]-\mathbb{P}[\mathcal{I},h\prec]\nabla z,
\end{equation}
	where $[\mathcal{I},h\prec]\nabla z$ denotes the commutator between $\mathcal{I}$ and $h\prec$ given by
	$$
	[\mathcal{I},h\prec]\nabla z=\mathcal{I}[h\prec\nabla z]-h\prec \mathcal{I}\nabla z.
	$$
	We use the same notation for  other commutators as well.
The second and the third terms on the right hand side of \eqref{eq:varth}  cancel the irregular terms in $h$ whereas the last term has better regularity by Lemma \ref{commutator}.
 Using the commutator Lemma \ref{commutator}, Lemma~\ref{lem:com1} and Lemma \ref{lem:com2}, we write formally $h \varocircle z$ as
\begin{align*}
\begin{aligned}
h \varocircle z&= -\mathbb{P} [h\prec  \mathcal{I}
\nabla z] \varocircle z + \vartheta\varocircle z - (\mathbb{P}\mathcal{I}
\tmop{div} (z \otimes z_1 + z_1 \otimes z)) \varocircle z
\\
&= - ([\mathbb{P}, h \prec] \mathcal{I} \nabla z) \varocircle z -
\tmop{com} (h, \mathbb{P}\mathcal{I} \nabla z, z) - h \cdot(\tmcolor{orange} {\mathbb{P}\mathcal{I} \nabla z \varocircle z})
\\& \qquad+\vartheta \varocircle z- \tmcolor{orange} {
	(\mathbb{P}\mathcal{I} \tmop{div} (z \otimes z_1 + z_1 \otimes z))
	\varocircle z}.
\end{aligned}
\end{align*}
The above orange terms  are still ill-defined and   need to be replaced by the corresponding stochastic objects. Hence the rigorous paracontrolled ansatz and the definition of  $h \varocircle z$ read as
\begin{align}
\label{anastz}
h=-\mathbb{P} [h \prec  \mathcal{I} \nabla z] +
\vartheta -(z^{\<21l1>}+z^{\<211>}),
\end{align}
\begin{align}\label{product}
\begin{aligned}
h \varocircle z &:=
- ([\mathbb{P}, h \prec] \mathcal{I} \nabla z) \varocircle z -
\tmop{com} (h, \mathbb{P}\mathcal{I} \nabla z, z) - h\cdot z^{\<101>}
+ \vartheta \varocircle z-
z^{\<2111>},
\end{aligned}
\end{align}

Let us now formulate the definition of paracontrolled solution to \eqref{solution}. To this end, we recall  that the given cylindrical Wiener process $B$ is defined on  a stochastic basis $(\Omega,\mathcal{F},( \mathcal{F}_t)_{t\geq 0},\mathbf{P})$ and this is also where all the stochastic objects in Section~\ref{s:stoch} are constructed.

\begin{definition}\label{def:sol}
	We say that a pair of $( \mathcal{F}_t)_{t\geq 0}$-adapted processes
$$
	(h,\vartheta)\in \left(L^2(0,T,L^2)\cap C_{T,1/2-\kappa}^{1/10}L^{5/3}\cap L^1(0,T,B^{1/5}_{5/3,\infty})\cap C_TH^{-1}\right)\times \left(
	L^1(0,T,B^{3/5-\kappa}_{1,\infty})\cap C_TH^{-1}\right)
	$$
$\mathbf{P}$-a.s.	with $ \kappa>0$ is a paracontrolled solution to \eqref{solution} provided \eqref{anastz} holds
	and the equation \eqref{solution} holds in the analytically weak sense with the resonant product $h \varocircle z$  given by \eqref{product}.
\end{definition}

\begin{remark}
(i)	We note that paracontrolled solutions are probabilistically strong, that is, they are adapted to the given filtration $(\mathcal{F}_{t})_{t\geq0}$.

(ii) The space $C^{1/10}_{T,1/2-\kappa}$ with singularity at time zero is used in Definition~\ref{def:sol} in order to cover the case of irregular initial conditions, namely, divergence free initial conditions in $C^{-1+{2}\kappa}$ for $\kappa>0$.

(iii) The function spaces to which $(h,\vartheta)$ belong together with Proposition~\ref{p:sto} guarantee that  all the terms in \eqref{anastz}-\eqref{product} are well-defined. This in turn permits to make sense of all the terms in \eqref{solution}.
\end{remark}

\subsection{Formal decomposition}\label{s:dec}

We continue with the formal decomposition of the equation in order to find the desired paracontrolled ansatz. Moreover, in order to be able to apply convex integration, we need to introduce a further decomposition into $v^{1}$, $v^{2}$. In this subsection, we ignore the fact that various products of $z$ are not well-defined. We proceed as if these were well-defined and we replace their values in the end by the corresponding stochastic objects constructed in Section~\ref{s:stoch}.

Let $h=v^{1}+v^{2}$, \eqref{solution} rewrites as
\begin{equation}\label{eq:v1v2}
\begin{aligned}
\mathcal{L} (v^1 + v^2) + \tmop{div} ((v^1 + v^2 + z_1) \otimes (v^1 + v^2
+ z_1) + &z \otimes (v^1 + v^2 + z_1) + (v^1 + v^2 + z_1) \otimes z) \\
+ \nabla (p^1 + p^2) &= 0,\qquad \tmop{div} (v^{1}+v^{2})=0.
\end{aligned}
\end{equation}
The regular terms (encoded in \tmcolor{blue}{blue}) shall be put in the
equation for $v^2$, whereas the irregular ones (in \tmcolor{magenta}{magenta})
into $v^1$, so that in the end
\begin{equation}\label{eq:vv1}
\mathcal{L}v^1 + \tmop{div} (\tmcolor{magenta}{z_1 \otimes z_1 + V^1 +
	V^{1, \ast}}) + \nabla p^1 = 0, \qquad \tmop{div} v^1 = 0,
\end{equation}
\begin{equation}\label{eq:vv2}
\mathcal{L}v^2 + \tmop{div} (\tmcolor{blue}{(v^1 + v^2) \otimes (v^1 + v^2)
	+ V^2 + V^{2, \ast}}) + \nabla p^2 = 0, \qquad \tmop{div} v^2 = 0.
\end{equation}
Here $\tmcolor{magenta}{V^{1, \ast}}$ and $\tmcolor{blue}{V^{2, \ast}}$,
respectively, denotes the transpose of $\tmcolor{magenta}{V^1}$ and
$\tmcolor{blue}{V^2}$, respectively.
First, we
assign
\[ \tmcolor{blue}{(v^1 + v^2) \otimes (v^1 + v^2)}, \quad
\tmcolor{magenta}{z_1 \otimes z_1 + z \otimes z_1 + z_1 \otimes z} . \]
Then we decompose
\[ (v^1 + v^2) \otimes z_1 = (v^1 + v^2) \otimes \Delta_{> R} z_1 + (v^1 +
v^2) \otimes \Delta_{\leqslant R} z_1 \]
\[ = \tmcolor{magenta}{(v^1 + v^2) \oprec \Delta_{> R} z_1 + v^1 \osucccurlyeq
	\Delta_{> R} z_1} \tmcolor{blue}{+ v^2 \osucccurlyeq \Delta_{> R} z_1}
\tmcolor{blue}{+ (v^1 + v^2) \otimes \Delta_{\leqslant R} z_1} . \]
These terms are included into $\tmcolor{magenta}{V^1}$ and
$\tmcolor{blue}{V^2}$, whereas the symmetric counterparts
\[ z_1 \otimes (v^1 + v^2) = \tmcolor{magenta}{\Delta_{> R} z_1 \osucc (v^1 +
	v^2) + \Delta_{> R} z_1 \opreccurlyeq v^1} \tmcolor{blue}{+ \Delta_{> R}
	z_1 \opreccurlyeq v^2} \tmcolor{blue}{+ \Delta_{\leqslant R} z_1 \otimes
	(v^1 + v^2)} \]
are in $\tmcolor{magenta}{V^{1, \ast}}$ and $\tmcolor{blue}{V^{2, \ast}}$.
Note that we colored  $\tmcolor{magenta}{v^1 \osucccurlyeq
	\Delta_{> R} z_1}$ magenta. Indeed,  it requires regularity of $v^{1}$ which is not true uniformly as $t\to0$, since the initial condition is only in $L^{2}$. Hence if colored blue, this term cannot be controlled in the convex integration scheme. This applies also to other terms below.

In
the rest of the computation we only discuss the terms in
$\tmcolor{magenta}{V^1}$ and $\tmcolor{blue}{V^2}$, the symmetric terms being
included automatically.
We split
\[ (v^1 + v^2) \otimes z = (v^1 + v^2) \oprec z + (v^1 + v^2) \osucc z + (v^1
+ v^2) \varocircle z \]
\[ = \tmcolor{magenta}{(v^1 + v^2) \oprec \Delta_{> R} z + v^1 \osucc
	\Delta_{> R} z} \tmcolor{blue}{\,+ \,v^2 \osucc \Delta_{> R} z + (v^1 + v^2)
	\left( \oprec + \osucc \right) \Delta_{\leqslant R} z + v^2 \varocircle z}
+ v^1 \varocircle z. \]
For the resonant product $v^1 \varocircle z$ we shall use a paracontrolled
ansatz which permits
to cancel the worst terms in $\tmcolor{magenta}{V^1}$. In particular, the multiplication of $v^1$
and $z$ is not well-defined as the expected sum of their regularities is not
strictly positive making the resonant product $v^1 \varocircle z$
ill-defined. First of all, there is the stochastic term $\div (z \otimes z_1 + z_1
\otimes z)$ which makes $v^1$ too irregular. Second of all, it is the following
paraproduct between $v^1$ and $z$
\[ \tmop{div} \left(\Delta_{>R} z \osucc (v^1 + v^2) \right) = (v^1 + v^2) \prec \Delta_{>R}\nabla
z, \]
which creates problems. This leads to the paracontrolled ansatz
\begin{equation}\label{eq:para}
v^1 = -\mathbb{P} [(v^1 + v^2) \prec \Delta_{> R} \mathcal{I} \nabla z] +
v^{\sharp} -\mathbb{P}\mathcal{I} \tmop{div} (z \otimes z_1 + z_1 \otimes
z), \qquad \tmop{div} v^{\sharp} = 0,
\end{equation}
where $\mathcal{I}f (t) = \int_0^t e^{(t - s)\Delta} f (s) \mathd s$. Thus,
 \begin{align}
 \label{eq:vv3}
	v^\sharp=v^1+\mathbb{P}\mathcal{I} \tmop{div} (z \otimes z_1 + z_1 \otimes
z)+\cI\mathbb{P} [(v^1 + v^2) \prec \Delta_{>R}\nabla z] -\mathbb{P} [\cI, (v^1 +
	v^2)\prec ]\Delta_{>R}\nabla z.
	\end{align}

This permits to cancel the two terms
\[ -\mathbb{P}\mathcal{I} \tmop{div} (z \otimes z_1 + z_1 \otimes
z) -\cI\mathbb{P} [(v^1
+ v^2) \prec \Delta_{> R} \nabla z] \]
from $v^1$.  Consequently, $v^{\sharp}$ has a better regularity than
$v^1$ and hence, in view of the commutator lemma, Lemma~\ref{lem:com1}, the resonant product $v^1
\varocircle z$ can be rigorously defined.
To this end, we compute
\[ \mathbb{P} [(v^1 + v^2) \prec \Delta_{> R} \mathcal{I} \nabla z]
\varocircle z =\mathbb{P} [(v^1 + v^2) \prec \mathcal{I} \nabla z]
\varocircle z -\mathbb{P} [(v^1 + v^2) \prec \Delta_{\leqslant R}
\mathcal{I} \nabla z] \varocircle z \]
\[ = ([\mathbb{P}, (v^1 + v^2) \prec] \mathcal{I} \nabla z) \varocircle z +
[(v^1 + v^2) \prec \mathbb{P}\mathcal{I} \nabla z] \varocircle z
-\mathbb{P} [(v^1 + v^2) \prec \Delta_{\leqslant R} \mathcal{I} \nabla z]
\varocircle z \]
\[ = ([\mathbb{P}, (v^1 + v^2) \prec] \mathcal{I} \nabla z) \varocircle z +
\tmop{com} (v^1 + v^2, \mathbb{P}\mathcal{I} \nabla z, z) \]
\[ + (v^1 + v^2)\cdot (\mathbb{P}\mathcal{I} (\nabla z) \varocircle z) -\mathbb{P}
[(v^1 + v^2) \prec \Delta_{\leqslant R} \mathcal{I} \nabla z] \varocircle
z. \]
The above commutator between the Helmholtz projection and the paraproduct is
understood componentwise as follows
\[ ([\mathbb{P}, (v^1 + v^2) \prec] \mathcal{I} \nabla z)^i =\mathbb{P}^{i j}
[(v^1 + v^2)^k \prec \mathcal{I} \partial_k z^j] - (v^1 + v^2)^k \prec
\mathbb{P}^{i j} \mathcal{I} \partial_k z^j \]
and accordingly
\[ ((v^1 + v^2) \prec \mathbb{P}\mathcal{I} \nabla z)^i = (v^1 + v^2)^k \prec
\mathbb{P}^{i j} \mathcal{I} \partial_k z^j . \]
We deduce
\[ v^1 \varocircle z = -\mathbb{P} [(v^1 + v^2) \prec \Delta_{> R} \mathcal{I}
\nabla z] \varocircle z + v^{\sharp} \varocircle z - (\mathbb{P}\mathcal{I}
\tmop{div} (z \otimes z_1 + z_1 \otimes z)) \varocircle z \]
\[ = - ([\mathbb{P}, (v^1 + v^2) \prec] \mathcal{I} \nabla z) \varocircle z -
\tmop{com} (v^1 + v^2, \mathbb{P}\mathcal{I} \nabla z, z) - (v^1 + v^2)\cdot
(\mathbb{P}\mathcal{I} \nabla z \varocircle z) \]
\[ \tmcolor{blue}{-\mathbb{P} [(v^1 + v^2) \prec \Delta_{\leqslant R}
	\mathcal{I} \nabla z] \varocircle z} + v^{\sharp} \varocircle z  \tmcolor{magenta}{-
	(\mathbb{P}\mathcal{I} \tmop{div} (z \otimes z_1 + z_1 \otimes z))
	\varocircle z}. \]
In order to avoid the singularity at time zero in the convex integration,  we include an additional splitting into $\Delta_{> R}$ and
$\Delta_{\leqslant R}$ of the remaining (uncolored) terms above as follows
\[ - ([\mathbb{P}, (v^1 + v^2) \prec] \mathcal{I} \nabla z) \varocircle z =
\tmcolor{magenta}{- ([\mathbb{P}, v^1 \prec] \mathcal{I} \nabla z)
	\varocircle \Delta_{> R} z} \]
\[ \tmcolor{blue}{- ([\mathbb{P}, v^1 \prec] \mathcal{I} \nabla z) \varocircle
	\Delta_{\leqslant R} z} \tmcolor{blue}{- ([\mathbb{P}, v^2 \prec]
	\mathcal{I} \nabla z) \varocircle z}, \]
\[ - \tmop{com} (v^1 + v^2, \mathbb{P}\mathcal{I} \nabla z, z) =
\tmcolor{magenta}{- \tmop{com} (v^1, \mathbb{P}\mathcal{I} \nabla z,
	\Delta_{> R} z)} \]
\[ \tmcolor{blue}{- \tmop{com} (v^1, \mathbb{P}\mathcal{I} \nabla z,
	\Delta_{\leqslant R} z)} \tmcolor{blue}{- \tmop{com} (v^2,
	\mathbb{P}\mathcal{I} \nabla z, z)}, \]
\[ - (v^1 + v^2)\cdot (\mathbb{P}\mathcal{I} \nabla z \varocircle z) =
\tmcolor{magenta}{- (v^1 + v^2) \prec \Delta_{> R} (\mathbb{P}\mathcal{I}
	\nabla z \varocircle z)} \tmcolor{magenta}{- v^1 \succcurlyeq \Delta_{> R}
	(\mathbb{P}\mathcal{I} \nabla z \varocircle z)} \]
\[ \tmcolor{blue}{- v^2 \succcurlyeq \Delta_{> R} (\mathbb{P}\mathcal{I}
	\nabla z \varocircle z)} \tmcolor{blue}{- (v^1 + v^2) \cdot
	\Delta_{\leqslant R} (\mathbb{P}\mathcal{I} \nabla z \varocircle z)}, \]
\[ v^{\sharp} \varocircle z = \tmcolor{magenta}{v^{\sharp} \varocircle
	\Delta_{> R} z} \tmcolor{blue}{+ v^{\sharp} \varocircle \Delta_{\leqslant
		R} z} . \]

Finally, collecting all the terms leads us to
\[ \tmcolor{magenta}{V^1} = \tmcolor{magenta}{z_1 \otimes z}
\tmcolor{magenta}{+ (v^1 + v^2) \oprec \Delta_{> R} z_1 + v^1 \osucccurlyeq
	\Delta_{> R} z_1 + (v^1 + v^2) \oprec \Delta_{> R} z + v^1 \osucc \Delta_{>
		R} z} \]
\[ \tmcolor{magenta}{} \tmcolor{magenta}{- (\mathbb{P}\mathcal{I} \tmop{div}
	(z \otimes z_1 + z_1 \otimes z)) \varocircle z} \tmcolor{magenta}{-
	([\mathbb{P}, v^1 \prec] \mathcal{I} \nabla z) \varocircle \Delta_{> R} z}
\tmcolor{magenta}{} \tmcolor{magenta}{- \tmop{com} (v^1,
	\mathbb{P}\mathcal{I} \nabla z, \Delta_{> R} z)} \]
\[ \tmcolor{magenta}{- (v^1 + v^2) \prec \Delta_{> R} (\mathbb{P}\mathcal{I}
	\nabla z \varocircle z)} \tmcolor{magenta}{- v^1 \succcurlyeq \Delta_{> R}
	(\mathbb{P}\mathcal{I} \nabla z \varocircle z) + v^{\sharp} \varocircle
	\Delta_{> R} z}, \]
\[ \tmcolor{blue}{V^2} = \tmcolor{blue}{v^2 \osucccurlyeq \Delta_{> R} z_1}
\tmcolor{blue}{+ (v^1 + v^2) \otimes \Delta_{\leqslant R} z_1} \]
\[ \tmcolor{blue}{+ v^2 \osucc \Delta_{> R} z + (v^1 + v^2) \left( \oprec +
	\osucc \right) \Delta_{\leqslant R} z + v^2 \varocircle z} \]
\[ \tmcolor{blue}{-\mathbb{P} [(v^1+v^2) \prec \Delta_{\leqslant R} \mathcal{I}
	\nabla z] \varocircle z - ([\mathbb{P}, v^2 \prec] \mathcal{I} \nabla z)
	\varocircle z} \tmcolor{blue}{- ([\mathbb{P}, v^1  \prec]
	\mathcal{I} \nabla z) \varocircle \Delta_{\leqslant R} z} \]
\[ \tmcolor{blue}{\tmcolor{blue}{- \tmop{com} (v^1, \mathbb{P}\mathcal{I}
		\nabla z, \Delta_{\leqslant R} z) - \tmop{com} (v^2, \mathbb{P}\mathcal{I}
		\nabla z, z)} - v^2 \succcurlyeq \Delta_{> R} (\mathbb{P}\mathcal{I} \nabla
	z \varocircle z)} \]
\[ \tmcolor{blue}{- (v^1 + v^2) \cdot \Delta_{\leqslant R}
	(\mathbb{P}\mathcal{I} \nabla z \varocircle z) + v^{\sharp} \varocircle
	\Delta_{\leqslant R} z} . \]

It will be seen  below, that letting
\[
h=v^{1}+v^{2},\qquad \vartheta = v^{\sharp} + v^{2} + \mathbb{P}[ (v^{1}+v^{2})\prec \Delta_{\leq R}\mathcal{I}\nabla z]
\]
satisfies the requirements of Definition~\ref{def:sol}.
The idea then is to apply convex integration on the level of the equation \eqref{eq:vv2} for $v^{2}$. In particular, we need to make sure that the convex integration gives $v^{2}$ of the required regularity for $\vartheta$.

\section{Convex integration set-up and results}\label{s:con}

The goal of this section  is to construct infinitely many probabilistically strong paracontrolled solutions $(h,\vartheta)$ and deduce global existence and non-uniqueness for the system \eqref{solution}. The equation \eqref{solution} for $h$ splits formally into the coupled system \eqref{eq:vv1}, \eqref{eq:vv2} for $v^{1}$ and $v^{2}$ and it remains
to  include the stochastic objects. This leads to
\begin{align}\label{main4}
\begin{aligned}
\mathcal{L}v^1 + \tmop{div} (z^{\<2020>} + V^1 + V^{1, \ast}) +
\nabla p^1 &= 0, \\
\mathcal{L}v^2 + \tmop{div} ((v^1 + v^2) \otimes (v^1 + v^2) +
V^2 + V^{2, \ast}) + \nabla p^2 &= \tmop{div} \mathring{R}_q, \\
\tmop{div} v^1 = \tmop{div} v^2 = 0, \qquad v^1 (0) = v_{0}, \qquad  v^{2}(0)&=0,
\end{aligned}
\end{align}
with
\[ {V^1} =  z^{\<21>}
+ (v^1 + v^2) \oprec \Delta_{> R} z^{\<20>} + v^1 \osucccurlyeq
\Delta_{> R}  z^{\<20>}  + (v^1 + v^2) \oprec \Delta_{> R} z + v^1 \osucc \Delta_{>
	R} z \]
\[ - z^{\<2111>} -
([\mathbb{P}, v^1 \prec] \mathcal{I} \nabla z) \varocircle \Delta_{> R} z
- \tmop{com} (v^1,
\mathbb{P}\mathcal{I} \nabla z, \Delta_{> R} z) \]
\[- (v^1 + v^2) \prec \Delta_{> R} z^{\<101>} - v^1 \succcurlyeq \Delta_{> R}
z^{\<101>} + v^{\sharp} \varocircle
\Delta_{> R} z \]
and
\[ {V^2} = v^2 \osucccurlyeq \Delta_{> R} z^{\<20>}
+ (v^1 + v^2) \otimes \Delta_{\leqslant R} z^{\<20>} \]
\[ + v^2 \osucc \Delta_{> R} z + (v^1 + v^2) \left( \oprec +
\osucc \right) \Delta_{\leqslant R} z + v^2 \varocircle z \]
\[ -\mathbb{P} [(v^1+v^2) \prec \Delta_{\leqslant R} \mathcal{I}
\nabla z] \varocircle z - ([\mathbb{P}, v^2 \prec] \mathcal{I} \nabla z)
\varocircle z - ([\mathbb{P},v^1  \prec]
\mathcal{I} \nabla z) \varocircle \Delta_{\leqslant R} z \]
\[- \tmop{com} (v^1, \mathbb{P}\mathcal{I}
\nabla z, \Delta_{\leqslant R} z) - \tmop{com} (v^2, \mathbb{P}\mathcal{I}
\nabla z, z) - v^2 \succcurlyeq \Delta_{> R} z^{\<101>} \]
\[ - (v^1 + v^2) \cdot \Delta_{\leqslant R} z^{\<101>}
+ v^{\sharp} \varocircle
\Delta_{\leqslant R} z . \]
The paracontrolled ansatz for $v^{1}$ reads as
\begin{equation}\label{eq:para2}
v^{1}=-\mathbb{P} [(v^{1}+v^{2}) \prec \Delta_{>R} \mathcal{I} \nabla z] +
v^{\sharp} -(z^{\<21l1>}+z^{\<211>}).
\end{equation}
These  equations need to be considered together within the convex integration scheme and we put forward a joint iterative procedure.

The convex integration iteration is indexed by a parameter $q\in\mathbb{N}_{0}$. It will be seen that the Reynolds stress $\mathring{R}_{q}$ is only required for the approximations $v^{2}_{q}$ of  $v^{2}$, whereas the approximations $v^{1}_{q}$ of $v^{1}$ are obtained by a fixed point argument. We consider an increasing sequence $\{\lambda_q\}_{q\in\mathbb{N}_0}\subset \mathbb{N}$ which diverges to $\infty$, and a sequence $\{\delta_q\}_{q\in \mathbb{N}_0}\subset(0,1)$  which decreases to $0$. We choose $a\in\mathbb{N}$, $b\in\mathbb{N}$,  $\beta\in (0,1)$ and let
$$\lambda_q=a^{(b^q)}, \quad \delta_q=\lambda_1^{2\beta}\lambda_q^{-2\beta},$$
where $\beta$ will be chosen sufficiently small and $a$ as well as $b$ sufficiently large. At each step $q$, a triple $(v_q^1, v_q^2, \mathring{R}_q)$ is constructed solving the following system
\begin{align}\label{induction3}
\begin{aligned}
\mathcal{L}v_q^1 + \tmop{div} (z^{\<2020>} + V_q^1 + V_q^{1, \ast}) +
\nabla p_q^1 &= 0, \\
\mathcal{L}v_q^2 + \tmop{div} ((v_q^1 + v_q^2) \otimes (v_q^1 + v_q^2) +
V_q^2 + V_q^{2, \ast}) + \nabla p_q^2 &= \tmop{div} \mathring{R}_q, \\
\tmop{div} v_q^1 = \tmop{div} v_q^2 = 0, \qquad v^1_q (0) = v_{ 0}, \qquad  v^{2}_{q}(0)&=0,
\end{aligned}
\end{align}
where
\[ V_q^1 = z^{\<21>} + (v_q^1 + v_q^2) \oprec \Delta_{\leqslant f (q)}
\Delta_{> R} (z + z^{\<20>}) + v_{q}^1
\osucccurlyeq \Delta_{> R} z^{\<20>}+ v_q^1 \osucc \Delta_{\leqslant f (q)} \Delta_{> R} z \]
\[  - ([\mathbb{P}, v_q^1 \prec] \mathcal{I} \nabla z) \varocircle \Delta_{> R} z
\tmcolor{magenta}{} - \tmop{com} (v_q^1, \mathbb{P}\mathcal{I} \nabla z,
\Delta_{> R} z)
-z^{\<2111>} \]
\[ - (v_q^1 + v_q^2) \prec \Delta_{\leqslant f (q)}
\Delta_{> R} z^{\<101>} - v_{q}^1
\succcurlyeq \Delta_{> R} z^{\<101>}+ v_q^{\sharp} \varocircle \Delta_{> R} z \]
and
\[ V_q^2 = v_q^2 \osucccurlyeq \Delta_{> R} z^{\<20>}  + (v_q^1 + v_q^2) \otimes \Delta_{\leqslant R} z^{\<20>}
\]
\[ + v_q^2 \osucc \Delta_{> R} z + (v_q^1 + v_q^2) \left( \oprec + \osucc
\right) \Delta_{\leqslant R} z + v_{q}^2 \varocircle z \]
\[ -\mathbb{P} [(v_q^1 +v_q^2)\prec \Delta_{\leqslant R} \mathcal{I} \nabla z]
\varocircle z - ([\mathbb{P}, v_q^2 \prec] \mathcal{I} \nabla z)
\varocircle z - ([\mathbb{P}, v_q^1  \prec] \mathcal{I} \nabla z)
\varocircle \Delta_{\leqslant R} z \]
\[ - \tmop{com} (v_q^1, \mathbb{P}\mathcal{I} \nabla z, \Delta_{\leqslant R}
z) - \tmop{com} (v_q^2, \mathbb{P}\mathcal{I} \nabla z, z) -v_q^2 \succcurlyeq \Delta_{> R} z^{\<101>}  \]
\[
- (v_q^1 + v_q^2) \cdot \Delta_{\leqslant R} z^{\<101>} + v_q^{\sharp}
\varocircle \Delta_{\leqslant R} z.
\]
Here,  $V^1_q$ and $V^2_q$ are obtained from $V^1$ and $V^2$,
respectively, by replacing $v^1$, $v^2$, $v^{\sharp}$ by $v_q^1$, $v_q^2$,
$v_q^{\sharp}$ and adding the projector $\Delta_{\leqslant f
	(q)}$, where $2^{f(q)}=\lambda_{q}^{\theta}$, $\theta={10}/{21}$,  into the second, the fourth and the eighth term in $V^1_q$.  These projectors are used
to control the blow-up of certain norms of $v^1_q$ and $v^2_q$ and
also to prove the convergence of $v^1_q$ as $q \rightarrow \infty$.  We shall therefore require the parameter $a$ to be a power of $2^{21}$ and $b\in \mathbb{N}$. The parameter  $R$ is chosen in \eqref{c:R} below.

Analogously to the paracontrolled ansatz \eqref{eq:para2}, $v^{1}_{q}$, $v^{2}_{q}$ and $v_q^\sharp$ are linked via
\begin{align}
\label{vsharp}
v_q^1=-\mathbb{P}[(v_q^1+v_q^2)\prec\mathcal{I}(\nabla \Delta_{>R}\Delta_{\leq f(q)}z)]+v_q^\sharp-(z^{\<21l1>}+z^{\<211>}).
\end{align}
Our main goal is to prove convergence of $v^{1}_{q}$, $v^{2}_{q}$ and $v^{\sharp}_{q}$ as $q\to\infty$ and to show that  their limits satisfy \eqref{main4}, \eqref{eq:para2} in order to recover a paracontrolled solution to \eqref{solution} in the sense of Definition~\ref{def:sol}.
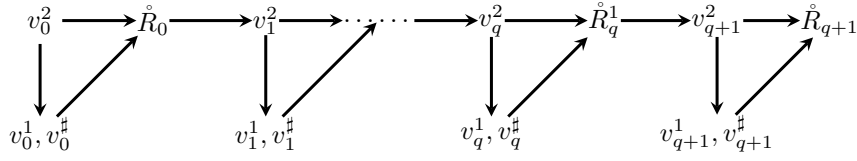
\begin{figure}
	\begin{tikzpicture}
	\node (lnd)  {$v^2_0$};
	
	\node (12) [below of=lnd, yshift=-0.5cm, inner sep = 0pt] {$v^1_0, v_0^\sharp$};
	
	\node (13) [right of=lnd, xshift=0.5cm, inner sep = 0pt] {$\mathring R_0$};
	
	\node (14) [right of=13, xshift=0.5cm, inner sep = 0pt] {$v^2_1$};
	
	\node (131) [below of=14, yshift=-0.5cm, inner sep = 0pt] {$v^1_1, v_1^\sharp$};
	
	\node (15) [right of=14, xshift=0.5cm, inner sep = 0pt] {$\dots \dots$};
	
	\node (16) [right of=15, xshift=0.5cm, inner sep = 0pt] {$v_q^2$};
	
	\node (17) [below of=16, yshift=-0.5cm, inner sep = 0pt] {$v^1_q, v_q^\sharp$};
	
	\node (18) [right of=16, xshift=0.5cm, inner sep = 0pt] {$\mathring R_q^1$};
	
	\node (19) [right of=18, xshift=0.5cm, inner sep = 0pt] {$v_{q+1}^2$};
	
	\node (20) [below of=19, yshift=-0.5cm, inner sep = 0pt] {$v_{q+1}^1,v_{q+1}^\sharp$};
	
	\node (21) [right of=19, xshift=0.5cm, inner sep = 0pt] {$\mathring R_{q+1}$};

	\draw[->,line width=0.4mm] (lnd) -- (12);
	\draw[->,line width=0.4mm] (lnd) -- (13);
	\draw[->,line width=0.4mm] (12) -- (13);
	\draw[->,line width=0.4mm] (13) -- (14);
	\draw[->,line width=0.4mm] (14) -- (131);
	\draw[->,line width=0.4mm] (14) -- (15);
	\draw[->,line width=0.4mm] (131) -- (15);
	\draw[->,line width=0.4mm] (15) -- (16);
	\draw[->,line width=0.4mm] (16) -- (17);
	\draw[->,line width=0.4mm] (16) -- (18);
	\draw[->,line width=0.4mm] (17) -- (18);
	\draw[->,line width=0.4mm] (18) -- (19);
	\draw[->,line width=0.4mm] (19) -- (20);
	\draw[->,line width=0.4mm] (20) -- (21);
	\draw[->,line width=0.4mm] (19) -- (21);
	\end{tikzpicture}	
	\caption{Iteration scheme.}
	\label{f:1}
\end{figure}
See Figure~\ref{f:1} for our iteration scheme. More precisely, we use $v^2_q$ to determine $v^1_q$ and $v^\sharp_q$ by Schauder estimates. Then $\mathring{R}_q$ is determined by $v^1_q, v^\sharp_q$ and $v^2_q$. The next velocity $v^2_{q+1}$ is only determined by $\mathring R_{q}$ via a convex integration argument.

As the next step, we define a stopping time which controls suitable norms of all the required stochastic objects. Namely, for  $L>1$
we let
\begin{align}\label{stopping time3}
T_L&:= T^1_L\wedge T^2_L\wedge T^3_L\wedge T^4_L\wedge T^5_L\wedge  L^{1/2},\\
T^1_L&:=\inf\left\{t\geq0, \|z(t)\|_{C^{-1/2-\kappa}}\geq L^{1/2}\right\}\wedge \inf\left\{t\geq0,\|z\|_{{C_t^{1/10}C^{-7/10-\kappa}}}\geq {L^{1/2}}\right\},\nonumber\\
T^2_L&:=\inf\left\{t\geq0, \|z^{\<20>}(t)\|_{C^{-\kappa}}\geq L\right\}{\wedge\inf\left\{t\geq0, \|z^{\<20>}\|_{C_t^{1/10}C^{-1/5-\kappa}}\geq L\right\}} ,\nonumber\\
T^3_L&:=\inf\left\{t\geq0, \|z^{\<2020>}(t)\|_{C^{-\kappa}}\geq L\right\}\wedge \inf\left\{t\geq0, \|z^{\<21>}(t)\|_{C^{-1/2-\kappa}}+\|z^{\<21l>}(t)\|_{C^{-1/2-\kappa}}\geq L\right\},\nonumber\\
T^4_L&:=\inf\left\{t\geq0, \sup_{0\leq s< r\leq t}\|z^{\<101>}(r;s)\|_{C^{-\kappa}}\geq L\right\} {\wedge\inf\left\{t\geq0, \|z^{\<101>}\|_{C_t^{1/10}C^{-1/5-\kappa}}\geq L\right\}},\nonumber
\\
T^5_L&:= \inf\left\{t\geq0, \sup_{0\leq s< r\leq t}\|z^{\<2111>}(r;s)\|_{C^{-\kappa}}\geq L\right\},\nonumber
\end{align}
where we denoted by
$
z^{\<2111>}(r;s)
$
and $z^{\<101>}(r;s)$
the stochastic objects obtained the same way as $z^{\<2111>}(r)$ and $z^{\<101>}(r)$ but replacing the last integration operator $\mathcal{I}=\mathcal{I}_{0,r}$ by
$\cI_{s,r}=\int_s^re^{(r-l)\Delta}\dif l$.
It follows from Proposition~\ref{p:sto} that the stopping time $T_{L}$ is $\mathbf{P}$-a.s. strictly positive and it holds that $T_{L}\uparrow\infty$ as $L\rightarrow\infty$ $\mathbf{P}$-a.s.

We intend to solve \eqref{main4} for any given divergence free initial condition $v_{0}\in L^{2}\cup C^{-1+\kappa}$ measurable with respect to $\mathcal{F}_0$.
However, in the first step, we take the following additional assumption:
Let $N\geq1$ be given and assume that $\mathbf{P}$-a.s.
\begin{equation}\label{eq:u0}
\|v_{0}\|_{L^{2}}\leq N.
\end{equation}
We keep this additional assumption on the initial condition throughout the convex integration step
in Proposition~\ref{p:iteration p}. In Theorem~\ref{thm:6.1} it is relaxed to $v_{0}\in L^{2}$ $\mathbf{P}$-a.s and, finally, Corollary~\ref{cor} proves the result if $v_{0}\in C^{-1+\kappa}$ $\mathbf{P}$-a.s.
We also suppose that there is a deterministic constant $M_{L}(N)^{1/2}\geq L^6N+L^{29}$. In the following we write $M_L$ instead of $M_L(N)$ for simplicity.

Let $\alpha\in(0,1)$ be a small parameter to be chosen below.
By induction on $q$ we  assume the following bounds for the iterations $v^{2}_q$: if $t\in[0, T_L]$ then for $p=\frac{32}{32-7\alpha}$
\begin{equation}\label{inductionv2p}
\aligned
\|v^2_q\|_{C_{t}W^{2/3,p}}&\leq {a^{-\alpha/2}}M_L^{1/2}(1+\sum_{1\leq r\leq q}\delta_{r}^{1/2})\leq 3M_L^{1/2}{a^{-\alpha/2}},
%
\\\|v^2_q\|_{C_{t}^{1/10}L^{5/3}}+\|v^2_q\|_{C_{t}W^{1/5,5/3}}&\leq {a^{-\alpha/2}}M_L^{1/2}(1+\sum_{1\leq r\leq q}\delta_{r}^{1/2})\leq 3M_L^{1/2}{a^{-\alpha/2}}.
\endaligned
\end{equation}
Later on, we use the factor  $a^{-\alpha/2}$ to absorb an implicit constant.
Here we defined $\sum_{1\leq r\leq 0}:=0$. In addition, we used $\sum_{r\geq 1}\delta_{r}^{1/2}\leq \sum_{r\geq1}a^{b\beta-rb\beta}=\frac{1}{1-a^{-\beta b}}\leq 2$ which boils down  to the requirement
\begin{equation}\label{aaa1}
a^{\beta b}\geq 2,
\end{equation}
which we assume from now on. We also assume that $L$ is large enough such that the implicit constant in \eqref{R0} and \eqref{R1} below can be absorbed by $L$. Moreover, for such $L$ we can always choose $a$ large enough such that $L\leq a^{\alpha/16}$.
%
%
%
We denote  $\sigma_{q}=2^{-q}$, $q\in\mathbb{N}_{0}\cup\{-1\}$,   $\gamma_q=2^{-q}$, $q\in\mathbb{N}_{0}\setminus\{3\}$, and  $\gamma_3=K>0$ arbitrary. This constant will be used in order to distinguish different solutions.

The key result is the following iterative proposition, which we prove below in Section~\ref{ss:it2}. 

\bp\label{p:iteration p}
Let $N\geq 1$ and let $L>1$ sufficiently large. There exists a choice of parameters $a$, $\alpha$,  $b$, $\beta$ such that the following holds true: Let $(v^1_{q},v^2_q,\mathring{R}_{q})$ for some $q\in\mathbb{N}_{0}$ be an $(\mathcal{F}_{t})_{t\geq 0}$-adapted solution to \eqref{induction3}, \eqref{vsharp} satisfying \eqref{inductionv2p} and
\begin{equation}\label{inductionv psp}
\|v_{q}^2(t)\|_{L^{2}}\leq\begin{cases}
M_0 (M_L^{1/2}\sum_{ r=1}^{q}\delta_{r}^{1/2}+\sum_{r=1}^{q}\gamma_{r}^{1/2})+3M_0(M_L(1+3q))^{1/2},&t\in (\frac{\sigma_{q-1}}2\wedge T_L,  T_L],\\
0, &t\in [0,\frac{\sigma_{q-1}}2\wedge T_L],
\end{cases}
\end{equation}
for a universal constant $M_0$,
\begin{align}\label{inductionv C1p}
\|v_q^2\|_{C^1_{t,x}}&\leq \lambda_q^4M_L^{1/2},\quad t\in[0, T_L],
\end{align}
\begin{align}\label{eq:Rp}
\|\mathring{R}_q(t)\|_{L^1}\leq
\delta_{q+1}M_L,\quad t\in (\sigma_{q-1}\wedge T_L,T_L],
\end{align}
\begin{align}\label{bd:Rp}
\|\mathring{R}_{q}(t)\|_{L^1}\leq M_L(1+3q),\quad t\in[0, T_L].
\end{align}
Then    there exists an $(\mathcal{F}_{t})_{t\geq 0}$-adapted process $(v^1_{q+1},v^2_{q+1},\mathring{R}_{q+1})$ which solves \eqref{induction3}, \eqref{vsharp} on the level $q+1$ and satisfies
\begin{equation}\label{iteration psp}
\|v_{q+1}^2(t)-v_q^2(t)\|_{L^2}\leq \begin{cases}
M_0 (M_L^{1/2}\delta_{q+1}^{1/2}+\gamma_{q+1}^{1/2}),& t\in (4\sigma_q\wedge T_L,T_L],\\
M_0 ((M_L(1+3q))^{1/2}+\gamma_{q+1}^{1/2}),& t\in (\frac{\sigma_q}2\wedge T_L,4\sigma_q\wedge T_L],\\
0,
&t\in [0,\frac{\sigma_q}2\wedge T_L],
\end{cases}
\end{equation}
\begin{equation}\label{iteration Rp}
\|\mathring{R}_{q+1}(t)\|_{L^1}\leq\begin{cases}
M_L\delta_{q+2},& t\in (\sigma_q\wedge T_L,T_L],\\
M_L\delta_{q+2}+\sup_{s\in[(t-\sigma_{q}/2)\vee 0,t]}\|\mathring{R}_{q}(s)\|_{L^1},&t\in (\frac{\sigma_q}2\wedge T_L,\sigma_q\wedge T_L],\\
\sup_{s\in[(t-\sigma_{q}/2)\vee 0,t]}\|\mathring{R}_{q}(s)\|_{L^1}+3M_L&t\in [0,\frac{\sigma_q}2\wedge T_L].
\end{cases}
\end{equation}
Consequently, $( v_{q+1}^2,\mathring{R}_{q+1})$ obeys \eqref{inductionv2p}, \eqref{inductionv psp},  \eqref{inductionv C1p}, \eqref{eq:Rp} and \eqref{bd:Rp} at the level $q+1$.
Furthermore, for $1<p=\frac{32}{32-7\alpha}$, $t\in [0,T_L]$ it holds
\begin{align}\label{induction wp ps}
\|v_{q+1}^2(t)-v_q^2(t)\|_{W^{2/3,p}}\leq
M_L^{1/2}\delta_{q+1}^{1/2}a^{-\alpha/2},
\end{align}
\begin{align}\label{induction wtp}
\|v_{q+1}^2-v_q^2\|_{C_t^{1/{10}}L^{5/3}}+\|v_{q+1}^2-v_q^2\|_{C_tW^{1/5,5/3}}\leq
M_L^{1/2}\delta_{q+1}^{1/2}a^{-\alpha/2},
\end{align}
and for $t\in (4\sigma_q\wedge T_L,T_L]$ we have
\begin{align}\label{p:gammap}
\big|\|v_{q+1}^2\|_{L^2}^2-\|v_q^2\|_{L^2}^2-3\gamma_{q+1}\big|\leq 7M_L\delta_{q + 1}.
\end{align}
\ep

Note that no bounds on $v^{1}_{q}$, $v^{1}_{q+1}$ were included in the statement of Proposition~\ref{p:iteration p}. Indeed, the definition of the new velocity $v^{2}_{q+1}$ does not require $v^{1}_{q+1}$. Then, having $v^{2}_{q+1}$ at hand, all the necessary bounds for $v^{1}_{q+1}$, $v^{\sharp}_{q+1}$ follow from  Section~\ref{s:v1vs} below. In particular, in Section~\ref{s:v1vs} we prove the following.

\bp\label{p:v1}
Under the assumptions of Proposition~\ref{p:iteration p}, it holds for $\kappa>0$ and $t\in[0,T_L]$
\begin{align}\label{estimatev12}
\|v_q^1\|_{C_{t,1/6}B_{5/3,\infty}^{1/3-2\kappa}}+\|v_q^1\|_{C^{1/6-\kappa}_{t,1/6}L^{5/3}}+\|v_q^1\|_{C_tL^2}\lesssim  a^{-\alpha/4}M_L^{1/2}+L^3N+L^4,
\end{align}
\begin{align}\label{vsharp1}
\|v_q^\sharp\|_{C_{t,3/10}B_{5/3,\infty}^{3/5-\kappa}}+\|v_q^\sharp\|_{C^{1/20}_{t,3/{10}} B_{5/3,\infty}^{11/20-2\kappa}}+\|v_q^\sharp\|_{C_tL^2}
\lesssim a^{-\alpha/8}M_L^{1/2}+L^5N+L^6,
\end{align}
\begin{align*}
\|v_{q+1}^1-v_q^1\|_{C_tL^2}+\|v_{q+1}^1-v_{q}^1\|_{C_tB_{5/3,\infty}^{1/3-\kappa}}+\|v_{q+1}^1-v_{q}^1\|_{C^{1/6-\kappa}_tL^{5/3}}\lesssim a^{-\alpha/2}M_{L}^{1/2}\delta^{1/2}_{q+1}+M_L^{1/2}\lambda_{q}^{-\theta/20},
\end{align*}
\begin{align*}
\|v_{q+1}^\sharp-v_q^\sharp\|_{C_{t,3/{10}}B_{5/3,\infty}^{3/5-\kappa}}+\|v_{q+1}^\sharp-v_q^\sharp\|_{C_tL^{2}}\lesssim a^{-\alpha/2} M_{L}^{1/2}\delta^{1/2}_{q+1}+M_L^{1/2}\lambda_{q}^{-\theta/20}.
\end{align*}
Here $\theta={10}/{21}$ and the implicit constants are always  universal and independent of $q$.
\ep

In the following we always use $\kappa>0$ to denote a small constant.

\begin{remark}
The best regularity we can expect for $v^1$ is $1/2-\kappa$ whereas for $v^\sharp$ it is $1-\kappa$. It will be seen in Section~\ref{s:v1vs} that their integrability is determined by $v^{2}$ and hence it comes from the convex integration argument. Here, we observe a competition between regularity and integrability, cf. \eqref{induction wp ps} and \eqref{induction wtp} and their proofs in Section~\ref{s:743} and Section~\ref{s:745}. For convenience, we have chosen integrability  $5/3$ and space regularity $1/3-2\kappa$ for $v^1$ and $3/5-\kappa$ for $v^\sharp$. The time weights are dictated by the desired space regularity as the initial value for $v^{1}_{q}$ and $v^{\sharp}_{q}$ only belongs to  $L^{2}$. We also note that  the bounds in Proposition~\ref{p:v1} do not  rely on the $W^{2/3,p}$ estimate of $v^{2}_{q}$ or the difference $v^{2}_{q+1}-v^{2}_{q}$. Indeed, this is only needed to make sense of the resonant product $v^{2}\varocircle z$ in the limit and to control the corresponding part of the Reynolds stress.
\end{remark}

We intend to start the iteration from $v_{0}^2\equiv 0$ on $[0,T_{L}]$. Then \eqref{inductionv2p}, \eqref{inductionv psp} and \eqref{inductionv C1p} hold. In that case, $\mathring{R}_0$ is the trace-free part of the matrix
\[ v_0^1  \otimes v_0^1 +V^{2}_{0} +V_{0}^{2,*}\]
where
\[
V^{2}_{0} =  v_0^1  \otimes \Delta_{\leqslant R} z^{\<20>}+ v_0^1  \left( \oprec + \osucc
\right) \Delta_{\leqslant R} z
-\mathbb{P} [v_0^1 \prec \Delta_{\leqslant R} \mathcal{I} \nabla z]
\varocircle z  - ([\mathbb{P}, v_0^1 \prec] \mathcal{I} \nabla z)
\varocircle \Delta_{\leqslant R} z \]
\[ - \tmop{com} (v_0^1, \mathbb{P}\mathcal{I} \nabla z, \Delta_{\leqslant R}
z)
- v_0^1 \cdot \Delta_{\leqslant R} z^{\<101>} + v_0^{\sharp}
\varocircle \Delta_{\leqslant R} z.
\]
By \eqref{estimatev12} and \eqref{vsharp1}, paraproduct estimates Lemma \ref{lem:para}, commutator estimates Lemmas \ref{lem:com1}, \ref{lem:com2}, we have
	\begin{align}\label{R0}
	\|R_{0}(t)\|_{L^1}&\lesssim \|v_0^1\|_{L^2}(\|\Delta_{\leq R}z^{\<20>}\|_{C^\kappa}+\|\Delta_{\leq R}z^{\<101>}\|_{C^\kappa}+\|\Delta_{\leq R}z\|_{C^{\kappa}})\\
	&\quad+\|v_0^1\|_{L^2}^2+\|v_0^\sharp\|_{L^{2}}\|\Delta_{\leq R}z\|_{C^{\kappa}}+\|v_0^1\|_{L^2}\|z\|_{C^{-1/2-\kappa}}\|\Delta_{\leq R}z\|_{C^{-1/2+2\kappa}}\nonumber
	\\&\lesssim a^{-\alpha/8}M_L+L^6N^2+L^8+L(a^{-\alpha/8}M_L^{1/2}+L^5N+L^6)2^{(1/2+2\kappa)R}\nonumber
	\\&\leq M_L.\nonumber
	\end{align}
Here we used the value of $R$ from \eqref{c:R} in Section \ref{s:v1vs} and the implicit constant can be absorbed by taking $a$ and $L$ large enough. Thus \eqref{eq:Rp} as well as \eqref{bd:Rp} are satisfied on the level  $q=0$, since $\delta_{1}=1$.

We deduce the following result.

\bt\label{thm:6.1}
There exists a $\mathbf{P}$-a.s. strictly positive stopping time $T_L$, arbitrarily large by choosing $L$ large, such that for any $\mathcal{F}_0$-measurable divergence free initial condition $v_0\in L^2$ $\mathbf{P}$-a.s.
the following holds true: There exists an $(\mathcal{F}_t)_{t\geq0}$-adapted process $(v^1,v^2,v^\sharp)$ such that for $\kappa>0$
$$
	\begin{aligned}
	v^{1}&\in C([0,T_L];L^2)\cap L^1(0,T_L;B_{5/3,\infty}^{1/3-2\kappa})\cap C^{1/6-\kappa}_{T_L,1/6}L^{5/3},\\
	v^{2}& \in L^p(0,T_L;L^2)\cap C([0,T_L],W^{2/3,1}\cap W^{1/5,5/3})\cap C^{1/10}_{T_L}L^{5/3},\\
	v^{\sharp}& \in C([0,T_L];L^2)\cap L^1(0,T_L;B^{3/5-\kappa}_{5/3,\infty}),
	\end{aligned}
	$$  $\mathbf{P}$-a.s. for all $p\in[1,\infty)$, and it is an analytically weak solution to \eqref{main4} with $v^1(0)=v_0,$ $ v^2(0)=0$ and satisfying \eqref{eq:para2}.  Furthermore, there are infinitely many such solutions and also infinitely many paracontrolled solutions $(h,\vartheta)=(v^1+v^2,v^\sharp+v^2+\mathbb{P}[(v^1+v^2)\prec \Delta_{\leq R}\mathcal{I}\nabla z])$ on $[0,T_L]$ satisfying
	\begin{align*}
	h&\in C([0,T_L];L^{5/3})\cap L^1(0,T_L;B_{5/3,\infty}^{1/5})\cap C^{1/10}_{T_L,1/6}L^{5/3}\cap L^p(0,T_L;L^2),\\
	\vartheta& \in C([0,T_L];L^{5/3})\cap L^1(0,T_L;B^{3/5-\kappa}_{1,\infty}),
	\end{align*}
	$\mathbf{P}$-a.s. for all $p\in[1,\infty)$, where $R$  depends on $L$ and is chosen in \eqref{c:R} below.
\et

\begin{proof}
	Letting $v_{0}^2\equiv 0$, we repeatedly apply Proposition~\ref{p:iteration p} and obtain $(\mathcal{F}_{t})_{t\geq0}$-adapted processes $(v_q^1,v_{q}^2,\mathring{R}_{q})$, $q\in\mathbb{N}$, such that  $$v_q^2 \to v^2\quad\mbox{in}\quad C([0,T_L],W^{2/3,1}\cap W^{1/5,5/3})\cap C^{1/10}_{T_L}L^{5/3}$$ as a consequence of  \eqref{induction wp ps}, \eqref{induction wtp} and \eqref{inductionv2p}. In view of Proposition~\ref{p:v1}, it follows that $$v_q^1\to v^1 \quad\mbox{in}\quad C([0,T_L];L^2)\cap L^1(0,T_L;B_{5/3,\infty}^{1/3-2\kappa})\cap C^{1/6-\kappa}_{T_L,1/6}L^{5/3},$$
	$$
	v_q^\sharp \to v^\sharp\quad\mbox{in}\quad C([0,T_L];L^2)\cap L^1(0,T_L;B^{3/5-\kappa}_{5/3,\infty}).
	$$
Then $(v^1,v^2,v^\sharp)$ are $(\mathcal{F}_t)_{t\geq0}$-adapted.
	Moreover, using {\eqref{iteration psp}} we
	have for every $p\in[1,\infty)$
	$$
	\int_0^{T_L}  \| v_{q + 1}^2 - v_q^2 \|_{L^2}^p d t
	\leqslant \int_{\sigma_q / 2 \wedge T_L}^{4 \sigma_q \wedge
		T_L} \| v_{q + 1}^2- v_q^2 \|_{L^2}^p d t +
	\int_{4 \sigma_q \wedge T_L}^{T_L} \| v_{q + 1}^2 - v_q^2 \|_{L^2}^p d t
	$$
	$$ \lesssim \int_{\sigma_q / 2 \wedge T_L}^{4 \sigma_q \wedge T_L}
	M^p_0 ((M_L(1+3q))^{1/2} + \gamma^{1/2}_{q + 1})^{p} d t
	+ \int_{4 \sigma_q \wedge T_L}^{T_L} M_0^p (M_L^{1/2} \delta^{1/2}_{q + 1}
	+ \gamma^{1/2}_{q + 1})^{p} d t
	$$
	$$ \lesssim M_{0}^{p}\left( 2^{- q} ((M_L(1+3q))^{1/2}  + \gamma_{q+1}^{1/2})^{p}
	+ T_{L}(M_L^{1 / 2}\delta_{q+1}^{1/2} + \gamma_{q+1}^{1/2})^{p} \right).
	$$
	Thus, the sequence $v^{2}_{q}$, $q\in\mathbb{N}$, is  Cauchy hence converging in $L^{p}(0,T_{L};L^{2})$ for all $p\in[1,\infty)$.
	Accordingly, $v^{2}_q\to v^{2}$ also in $L^p(0,T_{L};L^{2})$.
	Furthermore, by \eqref{eq:Rp}, \eqref{bd:Rp} we know for all $p\in[1,\infty)$
	$$
	\int_0^{T_L}\|\mathring{R}_q(t)\|^{p}_{L^1}d t\lesssim M_L^{p}\delta^{p}_{q+1}T_L+(M_L(1+3q))^{p}2^{-q} \to 0, \quad\mbox{as}\quad q\to\infty.
	$$
	Thus, the process $(v^1,v^2,v^\sharp)$ satisfies \eqref{main4} and \eqref{eq:para2} before $T_L$ in the analytically weak sense. Since $v^{1}_{q}(0)=v_{0}$ and $v_q^2(0)=0$ for all $q\in\mathbb{N}_{0}$ we deduce that $v^1(0)=v_0$ and $v^2(0)=0$. Thus $(h,\vartheta)$ defined above solves \eqref{solution} in the sense of Definition \ref{def:sol}.
	
	Next, we prove non-uniqueness of  the constructed solutions.	In view of  \eqref{p:gammap}, we have on $t\in (4\sigma_{0}\wedge T_L,T_L]$
	\begin{align}\label{eq:Kp}
	\begin{aligned}
	\big|\|v^2\|_{L^2}^2-3K\big|&\leq \left|\sum_{q=0}^\infty(\|v_{q+1}^2\|_{L^2}^2-\|v_{q}^2\|_{L^2}^2-3\gamma_{q+1})\right|+3\sum_{q\neq2}\gamma_{q+1}
	\\
	&\leq 7M_L\sum_{q=0}^{\infty}\delta_{q+1}+3\sum_{q\neq 2}\gamma_{q+1}\leq 7M_L\sum_{q=0}^{\infty}\delta_{q+1}+3\sum_{q\neq 2}\gamma_{q+1}\leq c,
	\end{aligned}
	\end{align}
	where the constant $c>0$ is independent of $K$ and the parameters $a$, $\alpha$.
	This implies non-uniqueness by choosing different $K$.  More precisely, for a given  $L\geq 1$ {sufficiently large} it holds $\mathbf{P}(4\sigma_{0}<T_{L})>0$. The parameters $L,N$  determine $M_{L}(N)$ and consequently by choosing different $K=K(L,N)$ and $K'=K'(L,N)$ so that $3|K-K'|>2c$ we deduce that the corresponding solutions $v^{2}_{K}$ and $v^{2}_{K'}$ have different $L^{2}$-norms on the set
	$\{4\sigma_0<T_L\}.$ We claim that  the sums $v^1_K+v^{2}_{K}$ and $v^1_{K'}+v^{2}_{K'}$ are different as well. Indeed, it is easy to see from \eqref{estimatev12} and \eqref{eq:Kp} that
	$$\sqrt{3K-c}-M_L^{1/2}\leq \|v^1+v^2\|_{L^2}\leq M_L^{1/2}+\sqrt{3K+c}.$$
	Choosing $K'$ such that $\sqrt{3K'-c}-M_L^{1/2}> M_L^{1/2}+\sqrt{3K+c}$ gives  different solutions.
	
	For a general divergence free initial condition $v_{0}\in L^{2}$ $\mathbf{P}$-a.s., we define
	$$\Omega_{N}:=\{N-1\leq \|v_{0}\|_{L^{2}}< N\}\in\mathcal{F}_{0}.$$ Then the first part of this proof gives the existence of infinitely many paracontrolled  solutions $(h^{N},\vartheta^{N})$ on each $\Omega_{N}$. Letting
	$$
	h:=\sum_{N\in\mathbb{N}}h^{N}1_{\Omega_{N}},\qquad \vartheta:=\sum_{N\in\mathbb{N}}\vartheta^{N}1_{\Omega_{N}}
	$$ concludes the proof. Note that this also uses the fact that the stochastic  objects are defined in advance and then the rest of the construction proceeds  pathwise.
\end{proof}

By an argument similar to \cite[Theorem 1.1]{HZZ21} we may extend the paracontrolled solutions obtained in Theorem~\ref{thm:6.1} by other paracontrolled solutions in order to obtain global existence and non-uniqueness.

\bt\label{th:ma4}
Let $v_{0}\in L^2$ $\mathbf{P}$-a.s. be an $\mathcal{F}_0$-measurable divergence free initial condition.
There exist infinitely many  paracontrolled solutions $(h,\vartheta)$ to \eqref{solution}  on $[0,\infty)$. Moreover, it holds
$$
h\in L^p_{\rm loc}([0,\infty);L^2)\cap C([0,\infty),L^{5/3})\quad \mathbf{P}\mbox{-a.s. for all }\,p\in[1,\infty).
$$
\et

\begin{proof}
	By Theorem~\ref{thm:6.1} we constructed a paracontrolled solution $h=v^{1}+v^{2}$ before the stopping time $T_L$ starting from the given initial condition $h(0)=v_0\in L^{2}$ $\mathbf{P}$-a.s. Since $T_L>0$ $\mathbf{P}$-a.s., we know that for $\mathbf{P}$-a.e. $\omega$ there exists $q_0(\omega)$ such that $4\sigma_{q_0(\omega)}<T_L(\omega)$.
	By \eqref{iteration psp} we find
	\begin{align*}
	\|v^2(T_L)\|_{L^2}&\leq \sum_{0\leq q< q_0}\|v_{q+1}^2(T_L)-v_q^2(T_L)\|_{L^2}+\sum_{ q\geq q_0}\|v_{q+1}^2(T_L)-v_q^2(T_L)\|_{L^2}
	\\
	&\lesssim M_0q_0(M_L(1+q_0))^{1/2}+M_0(K^{1/2}+1)+M_0M_L^{1/2}<\infty.
	\end{align*}
	This implies that $\|v^2(T_L)\|_{L^2}<\infty$ $\mathbf{P}$-a.s. Since also $v^{1}(T_{L})\in L^{2}$ $\mathbf{P}$-a.s., we can use the value $(v^1+v^2)(T_{L})$ as a new initial condition for $h$ in Theorem~\ref{thm:6.1}.

	More precisely, we consider $\hat{h}(0)=(v^1+v^2)(T_{L})$ and define $\hat{z}(t)=z(t+T_L)$ and similarly we define the stochastic objects
	$$
	\hat{z}_1(t):=z^{\<20>}(t+T_{L}),\quad \hat{z}_1\otimes \hat{z}(t):=z^{\<21>}(t+T_{L}),\quad \hat{z}\otimes \hat{z}_1 (t):=z^{\<21l>}(t+T_{L}),
	$$
	$$
	\hat{z}^{\<2020>}(t):=z^{\<2020>}(t+T_{L}),\quad \cI(\nabla \hat{z})(t):=\cI_{T_{L},T_{L}+t}(\nabla z),
	$$
	$$
	\mathbb{P}\cI(\nabla \hat{z})\varocircle \hat{z}(t)=\mathbb{P}\cI_{T_{L},T_{L}+t}(\nabla {z})\varocircle {z}(t+T_{L})= z^{\<101>}(t+T_{L};T_{L})
	,$$
	$$
	\mathbb{P}\mathcal{I}(\div( \hat{z}\otimes \hat{z}_1+\hat{z}_1\otimes \hat{z}))\varocircle \hat{z}(t)=\mathbb{P}\mathcal{I}_{T_{L},T_{L}+t}(\div( z^{\<21l>}+z^{\<21>}))\varocircle {z}(t+T_{L})=z^{\<2111>}(t+T_{L};T_{L}).
	$$
	Then we define stopping time $\hat{T}_L$ similar as in \eqref{stopping time3} with $T^4_L$ and $T^{5}_{L}$ replaced by
\begin{align*}
\hat{T}^4_L&:=\inf\left\{t\geq0,\|z^{\<101>}(T_L+t;T_L)\|_{C^{-\kappa}}\geq L\right\}\wedge\inf\left\{t\geq0, \|z^{\<101>}\|_{C_{t+T_L}^{1/10}C^{-1/5-\kappa}}\geq L\right\},
\\\hat{T}^5_L&:= \inf\left\{t\geq0, \|z^{\<2111>}(T_L+t;T_L)\|_{C^{-\kappa}}\geq L\right\}.
\end{align*}
Then $\hat{T}_{L+1}\geq T_{L+1}-T_L$.

	Consequently, we obtain solutions
	$$
	\left(\hat{h}=\hat{v}^1+\hat{v}^2, \hat{\vartheta}=\hat{v}^\sharp+\hat{v}^2+\mathbb{P}[(\hat{v}^1+\hat{v}^2)\prec\Delta_{\leq R}\cI(\nabla \hat{z})]\right)
	$$ before $\hat{T}_{L+1}$ adapted to $\mathcal{F}_{t+T_L}$. Here $R$ is chosen as in \eqref{c:R} in Section~\ref{s:v1vs} but in terms of  $L+1$ instead of $L$. Moreover, by Proposition \ref{p:v1} and \eqref{inductionv2p} it holds that $\hat{h}(0)\in B_{5/3,\infty}^{1/5}$. Hence, there is no singularity near zero of $\hat{h}$  and similarly as in Section \ref{s:v1vs} we obtain $\hat{h}\in C_T^{1/10}L^{5/3}$. Then we set $h(t)=(v^1+v^2)1_{t\leq T_L}+\hat{h}(t-T_L)1_{T_{L+1}\geq t>T_L}$.
Then for $p\geq1$
\begin{align*}
	h&\in C([0,T_{L+1}];L^{5/3})\cap L^1(0,T_{L+1};B_{5/3,\infty}^{1/5})\cap C^{1/10}_{T_{L+1},1/6}L^{5/3}\cap L^p(0,T_{L+1};L^2).
	\end{align*} By the same argument as in the proof of \cite[Theorem 1.1]{HZZ21}, $h$ is adapted to $(\mathcal{F}_t)_{t\geq0}$ and satisfies the equation \eqref{solution} before $[0,T_{L+1}]$. Indeed, we have for $t>T_L$
	\begin{align*}h(t)=(v_1+v_2)(T_L)-\mathbb{P}\int_{T_L}^te^{(t-s)\Delta}\Big[\div\Big(&h\otimes h+z^{\<20>}\otimes h+h\otimes z^{\<20>}+z^{\<2020>}\\&+h\otimes z+z\otimes h+z^{\<21l>}+z^{\<21>}\Big)\Big]\dif s,
	\end{align*}
	with the paracontrolled ansatz
	\begin{align*}
	h(t)=- \mathbb{P}[h(t) \prec  \mathcal{I}_{T_L,t} \nabla z] +
	\hat{\vartheta}(t-T_L) -\mathcal{I}_{T_L,t}(z^{\<21l>}+z^{\<21>}),
	\end{align*}
	and for $s>T_L$
	\begin{align*}
	h \varocircle z(s) &= - \mathbb{P} [h\prec  \mathcal{I}_{T_L,s}
	\nabla z] \varocircle z(s) + \hat{\vartheta}(s-T_L)\varocircle z(s) - z^{\<2111>}(s;T_{L})
	\\ &=-  ([\mathbb{P}, h \prec] \mathcal{I}_{T_L,s} \nabla z) \varocircle z(s) -
	\tmop{com} (h(s), \mathbb{P}\mathcal{I}_{T_L,s} \nabla z, z(s)) -  h(s) \cdot z^{\<101>}(s;T_{L}) \nonumber
	\\& \qquad+\hat{ \vartheta}(s-T_L) \varocircle z(s) -z^{\<2111>}(s;T_{L}).
	%
	\end{align*}
	Now, we define
	\begin{align*}
	\vartheta& =\left(v^\sharp+v^2+ \mathbb{P}[(v^1+v^2)\prec\Delta_{\leq R}\cI(\nabla z)]\right)1_{t\leq T_L}
	\\&\qquad+\left(\hat{\vartheta}(t-T_L)+  \mathbb{P}[h(t) \prec e^{(t-T_L)\Delta} \mathcal{I}\nabla z(T_L)] +e^{(t-T_L)\Delta} (z^{\<21l1>}+z^{\<211>})(T_L)\right)1_{t>T_L}.
	\end{align*}
	It is easy to see that $\vartheta\in C([0,T_{L+1}],L^{5/3})\cap L^1(0,T_{L+1},B^{3/5-\kappa}_{1,\infty})$.
	Then, for $t,s>T_L$	it holds
	\begin{align*}
	h(t)=-  \mathbb{P}[h(t) \prec  (\mathcal{I}\nabla z)(t)] +
	\vartheta(t) -(z^{\<21l1>}+z^{\<211>})(t),
	\end{align*}
	\begin{align*}
	h \varocircle z(s)  =& -  ([\mathbb{P}, h(s) \prec] (\mathcal{I} \nabla z)(s)) \varocircle z(s) -
	\tmop{com} (h(s), \mathbb{P}(\mathcal{I} \nabla z)(s), z(s)) -  h(s)\cdot z^{\<101>}(s) \nonumber
	\\& + \vartheta(s)\varocircle z(s)- z^{\<2111>}(s).
	\end{align*}
Here we used that from the renormalization it holds
$$
z^{\<2111>}(s)=z^{\<2111>}(s,T_L)+e^{(s-T_L)\Delta}(z^{\<21l1>}+z^{\<211>})(T_L)\varocircle z(s)
$$ and similarly for other terms.

	Thus,  $(h,\vartheta)$ satisfies the equation \eqref{solution}, as well as \eqref{anastz} and \eqref{product} before $T_{L+1}$.
	Now, we can iterate the above steps, i.e. starting from $h(T_{L+k})$ and constructing solutions $( h_{k+1},\vartheta_{k+1})$ before the stopping time $T_{L+k+1}$.
	Define $ h=h_11_{t\leq T_L}+\sum_{k=1}^\infty h_k1_{\{T_{L+k-1}< t\leq T_{L+k} \}},$
	and obtain that $h\in L^p_{\rm{loc}}([0,\infty); L^2)\cap C([0,\infty);L^{5/3}) $, for all $p\in[1,\infty)$. Similarly, we define $\vartheta$ and we obtain that $(h,\vartheta)$ is a paracontrolled solution. We emphasize that $h$ does not blow up at any finite time $T$ since for any time $T$ we could find $k_0$ such $T\leq T_{L+k_0}$ and the infinite sum becomes  a finite sum. The desired norm of $h$ only depends on $L$, $k_0$ and the initial data.
	Furthermore, as in the proof of Theorem~\ref{thm:6.1} we obtain infinitely many such solutions by choosing different $K$.
\end{proof}

\begin{corollary}\label{cor}
	Let $v_{0}\in L^2\cup C^{-1+\kappa}$ with $\kappa>0$ $\mathbf{P}$-a.s. be a $\mathcal{F}_0$-measurable divergence free  initial condition. Then there exist infinitely many  paracontrolled solutions $(h,\vartheta)$ to \eqref{solution}  on $[0,\infty)$.
\end{corollary}
\begin{proof} If $v_{0}\in C^{-1+\kappa}$ with $\kappa>0$, by \cite{ZZ15} there exists a stopping time $0<\sigma\leq T_L$ and a local paracontrolled solution $(h,\vartheta)$  to \eqref{solution} in the sense of Definition \ref{def:sol}. Now, $h(\sigma)\in C^{1/2-\kappa}$, $ \kappa>0$, and we can start from $h(\sigma)$ and obtain infinitely many global paracontrolled solutions by using Theorem~\ref{th:ma4}. Moreover, there is no singularity at $h(\sigma)$ and $h\in C_{T,1/2-\kappa}^{1/10}L^{5/3}$. Note that due to singularity at zero, we only have $h\in L^2(0,T,L^2)\cap C^{1/10}_{T,1/2-\kappa}L^{5/3}\cap C_TH^{-1}\cap L^1(0,T,B^{1/5}_{5/3,\infty})$.
\end{proof}

Accordingly, Theorem~\ref{thm:1.1} is proved.

Finally, by exactly the same argument as in \cite[Corollary 1.2]{HZZ21}, Corollary~\ref{cor:1.2} follows.

\section{Estimate of $v_{q}^1$ and $v^\sharp_q$} \label{s:v1vs}

In this section, we work under the assumptions of Proposition~\ref{p:iteration p}. The main aim is to prove the  bounds \eqref{estimatev12}, \eqref{vsharp1} as well as for $\kappa>0$ and $t\in[0,T_L]$
\begin{align}
\label{est:vq}
&\|v_{q+1}^1-v_q^1\|_{C_tL^2}+\|v_{q+1}^1-v_{q}^1\|_{C_tB_{5/3,\infty}^{1/3-2\kappa}}+\|v_{q+1}^1-v_{q}^1\|_{C^{1/6-\kappa}_tL^{5/3}}\nonumber
\\&\lesssim\|v_{q+1}^2-v_q^2\|_{C_tB^{1/5}_{5/3,\infty}}+\|v_{q+1}^2-v_q^2\|_{C^{1/10}_tL^{5/3}}+M_L^{1/2}\lambda_{q}^{-\theta/20},
\end{align}
\begin{align}\label{vsharp2}
\begin{aligned}
&\|v_{q+1}^\sharp-v_q^\sharp\|_{C_{t,3/{10}}B_{5/3,\infty}^{3/5-\kappa}}+\|v_{q+1}^\sharp-v_q^\sharp\|_{ C_tL^{2}}
\\&\lesssim \|v_{q+1}^2-v_q^2\|_{C_tB^{1/5}_{5/3,\infty}}+\|v_{q+1}^2-v_q^2\|_{C^{1/10}_tL^{5/3}}+M_L^{1/2}\lambda_{q}^{-\theta/20},
\end{aligned}
\end{align}
and
\begin{align}
\|v_q^1\|_{C_{t,3/8}L^4}
\lesssim M_L^{1/2}\lambda_{q}^{\theta(7/10+2\kappa)}.
\label{estimatev42}
\end{align}
As a consequence of \eqref{induction wtp}, this proves Proposition~\ref{p:v1}.
Moreover, we recall that the equation for $v^{1}_{q}$ is linear.
Hence, for a given $v_q^2$ we obtain the existence and uniqueness of solution $v^{1}_{q}$ to \eqref{induction3} by a fixed point argument together with the uniform estimate derived in the sequel. Also, if $v^{2}_{q}$ is $(\mathcal{F}_t)_{t\geq0}$-adapted, so are $(v^1_q,v^\sharp_q)$. This in particular gives the existence of $v^{1}_{q+1}$ in Proposition~\ref{p:iteration p}, once the new velocity $v^{2}_{q+1}$ was constructed in Section~\ref{ss:it2}.

In the following, we make use of  the localizers $\Delta_{>R}$ present in the equation for $v^{1}_{q}$ in \eqref{induction3}. Namely, by an appropriate choice of $R$ we can always apply \eqref{eq:loc} to get a small constant in front of terms which contain $v^{1}_{q}$. We are therefore able to absorb them into the left hand sides of the estimates without a Gronwall argument. Due to singularity at $t=0$, we use the weighted in time norms $C_{t,\gamma}$ for several different $\gamma\geq 0$, see Section~\ref{sec:pre} for their definition. In the following, all the estimates are pathwise and valid  before the stopping time $T_L$.

\subsection{Estimate of $v^{1}_{q}$ in $C_{t,1/6}B^{1/3-2\kappa}_{5/3,\infty}$ and $ C^{1/6-\kappa}_{t,1/6}L^{5/3}$}
\label{s:6.1}

We intend to apply  Lemma \ref{l:sch2} and notice that by Remark~\ref{r:b.5} each application yields a factor $L$, independently of the chosen time weights in the range $\gamma,\delta\in \{0,1/6,3/10\}$. However, we note that the difficult terms for Lemma~\ref{l:sch2} are those where we need to {decrease} the weight, i.e. $\gamma<\delta$. For those we need to make sure that the condition
\begin{equation}\label{eq:c1}
\gamma-\delta-\alpha/2+\beta/2+1> 0
\end{equation}
from Lemma~\ref{l:sch2} is satisfied. It will be seen below that this is always achieved since these terms do not require such a gain in space regularity, i.e. the difference $\alpha-\beta$  compensates the negativity of the difference $\gamma-\delta$.

Hence, we shall bound each term appearing in $V_q^1$ as well as $z^{\<2020>}$ in appropriate (possibly time-weighted) function spaces with spatial regularity at least $B^{-2/3-\kappa}_{5/3,\infty}$. The terms in $V^{1,*}_{q}$ are estimated the same way.

Recall that by the definition of the stopping time \eqref{stopping time3} we have
\begin{align*}
\|z\|^{2}_{C_{t}C^{-1/2-\kappa}}+\|z^{\<2020>}\|_{C_tC^{-\kappa}}+\|z^{\<21>}\|_{C_tC^{-1/2-\kappa}}+\|z^{\<101>}\|_{C_{t}C^{-\kappa}}+\|z^{\<2111>}\|_{C_tC^{-\kappa}}\lesssim L.
\end{align*}
By the paraproduct estimate Lemma \ref{lem:para} we have
\begin{align*}
&\|(v_q^1+v_q^2)\oprec \Delta_{\leq f(q)}\Delta_{>R}z\|_{C_{t,1/6}B^{-2/3-\kappa}_{5/3,\infty}}
\\&\lesssim  \sup_{s\in [0,t]}s^{1/6}(\|v_q^1(s)\|_{L^{5/3}}+\|v_q^2(s)\|_{L^{5/3}})\|\Delta_{>R} z\|_{C_tC^{-2/3-\kappa}}
\\&\lesssim(\|v_q^1\|_{C_{t,1/6}L^{5/3}}+\|v_q^2\|_{C_{t,1/6}L^{5/3}})L2^{-R/6},
\end{align*}
\begin{align*}
&\|(v_q^1+v_q^2)\oprec \Delta_{\leq f(q)}\Delta_{>R}z^{\<20>}\|_{C_{t,1/6}B^{-1/6-\kappa}_{5/3,\infty}}+ \|(v_q^1+v_q^2)\prec \Delta_{\leq f(q)}\Delta_{>R}z^{\<101>}\|_{C_{t,1/6}B^{-1/6-\kappa}_{5/3,\infty}}
\\&\lesssim  \sup_{s\in [0,t]}s^{1/6}(\|v_q^1(s)\|_{L^{5/3}}+\|v_q^2(s)\|_{L^{5/3}})(\|\Delta_{>R} z^{\<20>}\|_{C_tC^{-1/6-\kappa}}+\|\Delta_{>R} z^{\<101>}\|_{C_tC^{-1/6-\kappa}})
\\&\lesssim(\|v_q^1\|_{C_{t,1/6}L^{5/3}}+\|v_q^2\|_{C_{t,1/6}L^{5/3}})L2^{-R/6},
\end{align*}
and
\begin{align*}
&\|v_q^1\osucccurlyeq \Delta_{>R}z^{\<20>}\|_{C_{t,1/6}B^{\kappa}_{5/3,\infty}}+\|v_q^1\succcurlyeq \Delta_{>R}z^{\<101>}\|_{C_{t,1/6}B^{\kappa}_{5/3,\infty}}+\| v_q^1\osucc\Delta_{\leq f(q)}\Delta_{>R}z\|_{C_{t,1/6}B^{-1/3-3\kappa}_{5/3,\infty}}
\\&\lesssim\sup_{s\in [0,t]}s^{1/6}\|v_q^1\|_{B_{5/3,\infty}^{1/3-2\kappa}}(\|\Delta_{>R}z^{\<20>}\|_{C_tC^{-1/6-\kappa}}+\|\Delta_{>R}z^{\<101>}\|_{C_tC^{-1/6-\kappa}}+\|\Delta_{>R} z\|_{C_tC^{-2/3-\kappa}})\\
&\lesssim\|v_q^1\|_{C_{t,1/6}B_{5/3,\infty}^{1/3-2\kappa}}L2^{-R/6}.
\end{align*}
For the two commutators we use the commutator estimates, Lemma~\ref{lem:com2} and Lemma~\ref{lem:com1}, to obtain
\begin{align*}
&\|([\mathbb{P},v_q^1\prec]\cI(\nabla z))\varocircle \Delta_{>R}z\|_{C_{t,1/6}B^{\kappa}_{5/3,\infty}}+\|\mathrm{com}(v_q^1,\mathbb{P}\cI(\nabla z),\Delta_{>R}z)\|_{C_{t,1/6}B^{\kappa}_{5/3,\infty}}
\\&\lesssim\|v_q^1\|_{C_{t,1/6}B_{5/3,\infty}^{1/3-2\kappa}}\|\Delta_{>R}z\|_{C_tC^{-2/3-\kappa}}\|z\|_{C_tC^{-1/2-\kappa}}\lesssim\|v_q^1\|_{C_{t,1/6}B_{5/3,\infty}^{1/3-2\kappa}}L2^{-{R}/6}.
\end{align*}

For the last term containing $v_q^\sharp$ we use the paraproduct estimate and Lemma~\ref{l:sch2}. In particular, since $v_{q}^{\sharp}$ requires a higher time weight $t^{3/10}$, we shall verify the condition \eqref{eq:c1}. It turns out that this is satisfied as
$ \gamma = 1 / 6$, $ \delta = 3 / 10$, $ \alpha = 1 / 3 - 2 \kappa$,
$ \beta = - 1 + 1 / 20 - 2 \kappa$
       hence \eqref{eq:c1} holds.
Accordingly, we obtain
\begin{align*}
& \| \mathcal{I} \tmop{div} (v^{\sharp}_q \varocircle \Delta_{> R}
   z) \|_{C_{t,1/6}B^{1 / 3 - 2 \kappa}_{5/3, \infty}}+\| \mathcal{I} \tmop{div} (v^{\sharp}_q \varocircle \Delta_{> R}
   z) \|_{C_{t,1/6}^{1/6-\kappa}L^{5/3}}\\
   & \lesssim L
   \| \tmop{div}
   (v^{\sharp}_q \varocircle \Delta_{> R} z) \|_{C_{t,3/10}B^{- 1 + 1 / 20 -
   2\kappa}_{5/3, \infty}} \\
& \lesssim L
   \| v^{\sharp}_q \varocircle \Delta_{> R} z
   \|_{C_{t,3/10}B^{1 / 20 - 2\kappa}_{5/3, \infty}} \\
& \lesssim L\| v^{\sharp}_q \|_{C_{t,3/10}B^{3 / 5-\kappa}_{5/3, \infty}} \|
   \Delta_{> R} z \|_{C_t C^{- 11 / 20 - \kappa}}\\
& \lesssim L^2\|v_q^\sharp\|_{C_{t,3/10}B_{5/3,\infty}^{3/5-\kappa}}2^{-R/20}.
\end{align*}
For the initial value part we have by Lemma \ref{lem:heat} and \eqref{eq:u0} for $t\in(0,T_{L}]$
\begin{align*}
&\|e^{t\Delta}v_0\|_{B_{5/3,\infty}^{1/3-2\kappa}}\lesssim L^{1/2}t^{-1/6+\kappa}N,
 \\&\|(e^{t\Delta}-e^{s\Delta})v_0\|_{L^{5/3}}\lesssim |t-s|^{1/6-\kappa}s^{-1/6+\kappa}N,\quad  |t-s|\leq1, \ 0<s<t.
\end{align*}

Summarizing all the above estimates and using Besov embedding Lemma \ref{lem:emb}, we obtain
\begin{equation}\label{vq1}
\aligned
&\|v_q^1\|_{C_{t,1/6}B_{5/3,\infty}^{1/3-2\kappa}}+\|v_q^1\|_{C^{1/6-\kappa}_{t,1/6}L^{5/3}}
\\&\lesssim L\Big(L+N+L2^{-R/20}(\|v_q^1\|_{C_{t,1/6}B_{5/3,\infty}^{1/3-2\kappa}}
+\|v_q^\sharp\|_{C_{t,3/10}B_{5/3,\infty}^{3/5-\kappa}})+L2^{-R/6}\|v_q^2\|_{C_{t,1/6}L^{5/3}}\Big).
\endaligned
\end{equation}
Here, we used \eqref{eq:u0} and, as mentioned above, the extra factor  $L$ comes from Lemma \ref{l:sch2}.

\subsection{Estimate of $v^{\sharp}_{q}$ in $C_{t,3/10}B^{3/5-\kappa}_{5/3,\infty}$ and $C^{1/20}_{t,3/10}B^{11/20-2\kappa}_{5/3,\infty}$}
\label{s:6.2}

Let us proceed with the estimate of $v^{\sharp}_{q}$. Here, there are no difficulties coming from changing the time weight as all the terms require either a lower or the same weight.
In view of \eqref{induction3}, the paracontrolled ansatz \eqref{vsharp}
and since
\[ \tmop{div} \left( (v_q^1 + v_q^2) \oprec \Delta_{\leqslant f (q)} \Delta_{>
   R} z \right) = \Delta_{\leqslant f (q)} \Delta_{> R} z \succ \nabla (v_q^1
   + v_q^2), \]
\[ \tmop{div} \left( \Delta_{\leqslant f (q)} \Delta_{> R} z \osucc (v_q^1 +
   v_q^2) \right) = (v_q^1 + v_q^2) \prec \nabla\Delta_{\leqslant f (q)} \Delta_{>
   R} z, \]
we obtain
\[ v^{\sharp}_q  = v^1_q +\mathbb{P} [(v^1_q + v^2_q) \prec \mathcal{I}
   (\nabla \Delta_{> R} \Delta_{\leqslant f (q)} z)] + \left(
   z^{\<21l1>} + z^{\<211>} \right) \]
\begin{equation}\label{eq:3}
 = v^1_q (0) -\mathcal{I}\mathbb{P} \tmop{div} \left( z^{\<2020>} +
   V^{\sharp}_q + V^{\sharp, \ast}_q \right)
   \end{equation}
\[ -\mathbb{P} [\mathcal{I}, (v^1_q + v^2_q) \prec] (\nabla \Delta_{> R}
   \Delta_{\leqslant f (q)} z) -\mathcal{I}\mathbb{P} [\nabla(v_q^1 + v_q^2) \prec
   \Delta_{\leqslant f (q)} \Delta_{> R} z], \]
where
\[ V_q^{\sharp} = (v_q^1 + v_q^2) \oprec \Delta_{\leqslant f (q)} \Delta_{> R}
   z^{\<20>} + v_q^1 \osucccurlyeq \Delta_{> R} z^{\<20>} + v_q^1 \osucc
   \Delta_{\leqslant f (q)} \Delta_{> R} z \]
\[ - ([\mathbb{P}, v_q^1 \prec]\mathcal{I} \nabla z) \varocircle \Delta_{> R} z
   - \tmop{com} (v_q^1, \mathbb{P}\mathcal{I} \nabla z, \Delta_{> R} z) - z^{\<2111>} \]
\[ - (v_q^1 + v_q^2) \prec \Delta_{\leqslant f (q)} \Delta_{> R} z^{\<101>}
   - v_q^1 \succcurlyeq \Delta_{> R} z^{\<101>} + v_q^{\sharp} \varocircle
   \Delta_{> R} z. \]
In the following we estimate each term on the right hand side of \eqref{eq:3}.

 From the above estimate we already know that $z^{\<2020>}$ as well as all the terms in $V^{\sharp}_{q}$ except for $v_q^\sharp\varocircle \Delta_{>R}z$ are bounded in  $B^{-1/3-3\kappa}_{5/3,\infty}$. %
We also showed by paraproduct estimates Lemma \ref{lem:para} that
\begin{align*}
\|v_q^\sharp\varocircle \Delta_{>R}z\|_{C_{t,3/10}B^{1/20-2\kappa}_{5/3,\infty}}
&\lesssim\|v_q^\sharp\|_{C_{t,3/10}B_{5/3,\infty}^{3/5-\kappa}}\|\Delta_{>R}z\|_{C_tC^{-11/20-\kappa}}\\
&\lesssim L^{1/2}\|v_q^\sharp\|_{C_{t,3/10}B_{5/3,\infty}^{3/5-\kappa}}2^{-{R}/{20}}.
\end{align*}
Moreover, Lemma \ref{lem:para} also implies
\begin{align*}
&\|\mathbb{P}[\nabla(v_q^1+v_q^2)\prec (\Delta_{> R}\Delta_{\leq f(q)}z)]\|_{C_{t,1/6}B^{-27/20-\kappa}_{5/3,\infty}}\\
&\lesssim\Big(\sup_{s\in [0,t]}s^{1/6}\|v_q^1(s)\|_{B_{5/3,\infty}^{1/3-2\kappa}}+L^{1/6}\|v_q^2\|_{C_tB^{1/5}_{5/3,\infty}}\Big)\|\Delta_{> R}z\|_{C_tC^{-11/20-\kappa}}
\\&\lesssim (\|v_q^1\|_{C_{t,1/6}B_{5/3,\infty}^{1/3-2\kappa}}L^{1/2}+L\|v_q^2\|_{C_tB^{1/5}_{5/3,\infty}})2^{-{R}/{20}},
\end{align*}
which can then be plugged in the Schauder estimate, Lemma~\ref{l:sch2}.
Next, we note that  Lemma \ref{commutator} can be applied to the remaining term in \eqref{eq:3} which also gives a factor of $L$. Hence, we use interpolation to have
\begin{align*}
&\|\mathbb{P}[\mathcal{I}, (v_q^1+v_q^2)\prec](\nabla\Delta_{>R}\Delta_{\leq f(q)}z)\|_{C_{t,3/10}B^{3/5}_{5/3,\infty}}\\
&\qquad+\|\mathbb{P}[\mathcal{I}, (v_q^1+v_q^2)\prec](\nabla\Delta_{>R}\Delta_{\leq f(q)}z)\|_{C^{1/20}_{t,3/10} B_{5/3,\infty}^{11/20-2\kappa}}
\\&\lesssim L(\|v_q^1\|_{C_{t,1/6}B^{1/3-2\kappa}_{5/3,\infty}}+\|v_q^2\|_{C_tB^{1/5}_{5/3,\infty}}+\|v_q^1\|_{C^{1/6-\kappa}_{t,1/6} L^{5/3}}+\|v_q^2\|_{C^{1/10}_tL^{5/3}})\|\Delta_{> R}z\|_{C_tC^{-11/20-\kappa}}
\\&\lesssim L^22^{-R/20}(\|v_q^1\|_{C_{t,1/6}B^{1/3-2\kappa}_{5/3,\infty}}+\|v_q^2\|_{C_tB^{1/5}_{5/3,\infty}}+\|v_q^1\|_{C^{1/6-\kappa}_{t,1/6} L^{5/3}}+\|v_q^2\|_{C^{1/10}_tL^{5/3}}).
\end{align*}
Combining the above estimates and using Schauder estimate Lemma \ref{l:sch2} and interpolation, we have
\begin{align*}
&\|v_q^\sharp\|_{C_{t,3/10}B_{5/3,\infty}^{3/5-\kappa}}+\|v_q^\sharp\|_{C^{1/20}_{t,3/10} B_{5/3,\infty}^{11/20-2\kappa}}\nonumber
\\&\lesssim L^2+LN+L^2 2^{-R/20}\Big(\|v_q^1\|_{C_{t,1/6}B^{1/3-2\kappa}_{5/3,\infty}}+\|v_q^2\|_{C_tB^{1/5}_{5/3,\infty}}\\
&\qquad+\|v_q^1\|_{C^{1/6-\kappa}_{t,1/6} L^{5/3}}+\|v_q^2\|_{C^{1/10}_tL^{5/3}}+\|v_q^\sharp\|_{C_{t,3/10}B_{5/3,\infty}^{3/5-\kappa}}\Big)+L^{2}2^{-R/6}\|v_q^2\|_{C_{t,1/6}L^{5/3}},
\end{align*}
which combined with \eqref{vq1} implies that
\begin{align*}
&\|v_q^1\|_{C_{t,1/6}B_{5/3,\infty}^{1/3-2\kappa}}+\|v_q^1\|_{C^{1/6-\kappa}_{t,1/6}L^{5/3}}+\|v_q^\sharp\|_{C_{t,3/10}B_{5/3,\infty}^{3/5-\kappa}}+\|v_q^\sharp\|_{C^{1/20}_{t,3/{10}} B_{5/3,\infty}^{11/20-2\kappa}}
\\&\lesssim L^2+LN
 +L^2 2^{-R/20}\Big(\|v_q^1\|_{C_{t,1/6}B^{1/3-2\kappa}_{5/3,\infty}}+\|v_q^2\|_{C_tB^{1/5}_{5/3,\infty}}\\
 &\qquad+\|v_q^1\|_{C^{1/6-\kappa}_{t,1/6} L^{5/3}}+\|v_q^2\|_{C^{1/10}_tL^{5/3}}+\|v_q^\sharp\|_{C_{t,3/10}B_{5/3,\infty}^{3/5-\kappa}}\Big)+L^{2}2^{-R/6}\|v_q^2\|_{C_{t,1/6}L^{5/3}}.
\end{align*}
Then we choose $R$ such that
\begin{align}\label{c:R}2^{R/20}= 4CL^2,\end{align} with $C$ being the implicit constant and use \eqref{inductionv2p} to obtain
\begin{align}\label{est:vq1}
&\|v_q^1\|_{C_{t,1/6}B_{5/3,\infty}^{1/3-2\kappa}}+\|v_q^1\|_{C^{1/6-\kappa}_{t,1/6}L^{5/3}}+\|v_q^\sharp(s)\|_{C_{t,3/10}B_{5/3,\infty}^{3/5-\kappa}}+\|v_q^\sharp\|_{C^{1/20}_{t,3/{10}} B_{5/3,\infty}^{11/20-2\kappa}}
\\&\lesssim L^2+LN+L^2a^{-\alpha/2}M_L^{1/2},\nonumber
\end{align}
 which implies the first part of \eqref{estimatev12} and \eqref{vsharp1}.

\subsection{Estimate of $v^{1}_{q}$ in $C_{t}L^{2}$}
\label{s:6.3}

Here, most of the terms are similar as in \eqref{vq1} but we need to be careful about the compatibility condition \eqref{eq:c1}. As mentioned above, terms except for $v_q^\sharp\varocircle \Delta_{>R}z$ and $(v_q^1+v_q^2)\oprec \Delta_{\leq f(q)}\Delta_{>R}z$ have spatial regularity  $$B^{-1/3-3\kappa}_{5/3,\infty}\subset  B^{-1/3-3/10-3\kappa}_{2,\infty}$$ by Lemma \ref{lem:emb}. Thus, the corresponding $C_{t}L^2$ norm can be bounded by Lemma~\ref{l:sch2} with $\gamma=0$, $\delta=1/6$, $\alpha=\kappa$, $\beta=-1/3-3/10-3\kappa-1$ as follows
\begin{align*}
&\| \mathcal{I} \tmop{div} (\cdots) \|_{C_t L^2} \lesssim \| \mathcal{I}
   \tmop{div} (\cdots) \|_{C_t B^{\kappa}_{2, \infty}}\\
   & \lesssim L \| \tmop{div} (\cdots) \|_{C_{t,1/6}B_{2,
   \infty}^{- 1 / 3 - 3 / 10 - 3 \kappa - 1}} \\
&\lesssim L^{2}+L^{2}\|v_q^1\|_{C_{t,1/6}B_{5/3,\infty}^{1/3-2\kappa}}+L^{2}\|v_q^2\|_{{C_{t}L^{5/3}}}.
\end{align*}
Here, in $\cdots$ we collected all the terms from $V^{1}_{q}+V^{1,*}_{q}$ except for the above mentioned two and $z^{\<21>}$ and their symmetric counterparts.

Next, using paraproduct estimates and  Lemma \ref{lem:emb}, the term corresponding to $v_q^\sharp\varocircle \Delta_{>R}z$ can be bounded in view of the embedding $B^{3 / 5 - \kappa - 1 / 2 - \kappa}_{5 / 3, \infty} \subset B_{2,
   \infty}^{3 / 5 - 1 / 2 - 2 \kappa - 3 / 10}=B^{-1/5-2\kappa}_{2,\infty}$ by Lemma~\ref{l:sch2} with $\gamma=0$, $\delta=3/10$, $\alpha=\kappa$, $\beta=-1/5-2\kappa-1$ as
   \[
\begin{aligned}
&\| \mathcal{I} \tmop{div} (v_q^\sharp\varocircle \Delta_{>R}z) \|_{C_t L^2} \lesssim \| \mathcal{I}
   \tmop{div} (v_q^\sharp\varocircle \Delta_{>R}z) \|_{C_t B^{\kappa}_{2, \infty}}
   \\& \lesssim L \| \tmop{div} (v_q^\sharp\varocircle \Delta_{>R}z) \|_{C_{t, 3 / 10} B_{2,
   \infty}^{-1/5-2\kappa-1}}
    \\& \lesssim L\| v_q^\sharp\varocircle \Delta_{>R}z \|_{C_{t, 3 / 10} B_{5/3,
   \infty}^{3 / 5 - 1 / 2 - 2 \kappa }}
   \\&\lesssim L\|v_q^\sharp\|_{C_{t,3/10}B_{5/3,\infty}^{3/5-\kappa}}\|z\|_{C_tC^{-1/2-\kappa}}
\\&\lesssim L^2\|v_q^\sharp\|_{C_{t,3/10}B_{5/3,\infty}^{3/5-\kappa}}.
   \end{aligned}
   \]

We also use Lemma \ref{lem:para} and embedding $B^{-1+3/{10}+\kappa}_{5/3,\infty}\subset H^{-1}$ to have
\begin{align*}
	\|(v_q^1+v_q^2)\oprec \Delta_{\leq f(q)}\Delta_{>R}z\|_{C_tH^{-1}}
&	\lesssim \|v_q^1\|_{C_tL^2}\|\Delta_{> R}z\|_{C_tC^{-2/3-\kappa}}+\|v_q^2\|_{C_tL^{5/3}}\|z\|_{C_tC^{-1/2-\kappa}}
	\\ &\lesssim L^{1/2}(\|v_q^1\|_{C_tL^2}2^{-R/6}+\|v_q^2\|_{C_tL^{5/3}}).
\end{align*}
Thus combining the above estimates and \eqref{inductionv2p}, \eqref{est:vq1} we obtain
\begin{align*}
&\|v_q^1\|_{C_tL^2}
\\&\lesssim L^2+LN+L^2(\|v_q^1\|_{C_{t,1/6}B_{5/3,\infty}^{1/3-2\kappa}}+\|v_q^2\|_{C_{t}L^{5/3}}+\|v_q^\sharp\|_{C_{t,3/10}B_{5/3,\infty}^{3/5-\kappa}})+L^2\|v_q^1\|_{C_tL^2}2^{-R/6}
\\&\lesssim L^{4}+L^{3}N +a^{-\alpha/2}L^{4}M_{L}^{1/2}+L^{2}a^{-\alpha/2}M_{L}^{1/2}+L^2\|v_q^1\|_{C_tL^2}2^{-R/6}
\\&\lesssim L^{4}+L^{3}N +a^{-\alpha/4}M_{L}^{1/2}+L^2\|v_q^1\|_{C_tL^2}2^{-R/6},
\end{align*}
using the fact that $L^{4}\leq a^{\alpha/4}$.  Hence the last part of \eqref{estimatev12} follows.

\subsection{Estimate of $v^{\sharp}_{q}$ in $C_{t}L^{2}$}
\label{s:6.3a}

We apply the paracontrolled ansatz \eqref{vsharp}, the Besov embedding Lemma~\ref{lem:emb}, and paraproduct estimates Lemma \ref{lem:para} as well as \eqref{inductionv2p} and \eqref{stopping time3} together with the Schauder estimate to control $z^{\<21l1>}$ and $z^{\<211>}$. We deduce
\begin{align*}
\|v_q^\sharp\|_{C_tL^2}&\lesssim L\|v_q^1\|_{C_tL^{2}}+\|v_q^2\prec\cI(\nabla \Delta_{>R}\Delta_{\leq f(q)}z)\|_{C_tB^{{3}/{10}+\kappa}_{5/3,\infty}}+L^{2}
\\
&\lesssim L(\|v_q^2\|_{C_tL^{5/3}}+\|v_q^1\|_{C_tL^{2}})+L^{2}
\\&\lesssim La^{-\alpha/2}M_{L}^{1/2}+L(L^{4}+L^{3}N +a^{-\alpha/4}M_{L}^{1/2}) +L^{2}
\\&\lesssim  M_L^{1/2}a^{-\alpha/8}+L^5N+L^6.
\end{align*}
The last part of \eqref{vsharp1} follows.

\subsection{Estimate of the difference $v_{q+1}^1-v_q^1$ in $C_{t}B^{1/3-2\kappa}_{5/3,\infty}$ and $C_{t}^{1/6-\kappa}L^{5/3}$}
\label{s:6.4}

Most of the  terms can be estimated similarly as in \eqref{vq1}. We do not have to consider the initial data as it vanishes and therefore we can even control directly the $C_t$-norms without any weight. The main change comes from the additional difference  $\Delta_{\leq f(q+1)}-\Delta_{\leq f(q)}$. First, we use Lemma \ref{lem:para} to bound the terms with paraproducts $\oprec, \prec$ containing $\Delta_{\leq f(q+1)}-\Delta_{\leq f(q)}$ by Schauder estimates as
\begin{align*}
&L\|v_q^1+v_q^2\|_{C_tL^{5/3}}\|(\Delta_{\leq f(q+1)}-\Delta_{\leq f(q)})\Delta_{>R}(z+z^{\<20>})\|_{C_tC^{-2/3-\kappa}}\lesssim M_L^{1/2}\lambda_{q}^{-\theta/6},
\end{align*}
\begin{align*}
&L\|v_q^1+v_q^2\|_{C_tL^{5/3}}\|(\Delta_{\leq f(q+1)}-\Delta_{\leq f(q)})\Delta_{>R}z^{\<101>})\|_{C_tC^{-2/3-\kappa}}\lesssim M_L^{1/2}\lambda_{q}^{-\theta/6}.
\end{align*}
The part of $\osucc$  containing $\Delta_{\leq f(q+1)}-\Delta_{\leq f(q)}$ can be bounded  by Schauder estimates Lemma~\ref{l:sch2} with $\gamma=0$, $\delta=1/6$, $\alpha= 1/3-2\kappa$, $\beta=1/3-2\kappa-11/20-\kappa-1$ as
\begin{align*}
&L\|v_q^1\|_{C_{t,1/6}B_{5/3,\infty}^{1/3-2\kappa}}\|(\Delta_{\leq f(q+1)}-\Delta_{\leq f(q)})z\|_{C_tC^{-11/20-\kappa}}\lesssim M_L^{1/2}\lambda_{q}^{-\theta/20}.
\end{align*}
Thus, in view of \eqref{vq1} we obtain
\begin{align}\label{est:vq2}
&\|v_{q+1}^1-v_{q}^1\|_{C_tB_{5/3,\infty}^{1/3-2\kappa}}+\|v_{q+1}^1-v_{q}^1\|_{C^{1/6-\kappa}_tL^{5/3}}\nonumber
\\&\lesssim L^{2}2^{-R/20}\Big(\|v_{q+1}^1-v_q^1\|_{C_tB_{5/3,\infty}^{1/3-2\kappa}}+\|v_{q+1}^2-v_q^2\|_{C_tL^{5/3}}+\|v_{q+1}^\sharp-v_q^\sharp\|_{C_{t,3/10}B_{5/3,\infty}^{3/5-\kappa}}\Big)\nonumber\\
&\quad+M_L^{1/2}\lambda_{q}^{-\theta/20},
\end{align}
where for $v^\sharp_q$ part we used $\gamma=0$, $\delta=3/10$, $\alpha=1/3-2\kappa$, $\beta=3/5-\kappa-11/20-1$. Therefore in order to deduce the first part of \eqref{est:vq}, we shall estimate the difference $v_{q+1}^\sharp-v_q^\sharp$.

\subsection{Estimate of the difference $v_{q+1}^\sharp-v_q^\sharp$ in $C_{t,3/10}B_{5/3,\infty}^{3/5-\kappa}$}
\label{s:6.5}

Terms in $V_q^{\sharp}$ can be bounded similarly as above. We only concentrate on  $$\mathbb{P} [\mathcal{I}, (v^1_q + v^2_q) \prec] (\nabla \Delta_{> R}
   \Delta_{\leqslant f (q)} z),\qquad\mathcal{I}\mathbb{P} [\nabla(v_q^1 + v_q^2) \prec
   \Delta_{\leqslant f (q)} \Delta_{> R} z].$$
   Similarly as before, most terms could be estimated as the estimates for $v_q^\sharp$ with $v_q$ replaced by $v_{q+1}-v_q$. We consider the terms containing $\Delta_{\leq f(q+1)}-\Delta_{\leq f(q)}$ and use
   $$\|(\Delta_{\leq f(q+1)}-\Delta_{\leq f(q)})z\|_{C_tC^{-11/20-\kappa}}\lesssim \lambda_{q}^{-\theta/20} L^{1/2}.$$
Thus, using \eqref{estimatev12} the $C_{t,3/10}B_{5/3,\infty}^{3/5-\kappa}$-norm of these terms can be bounded by
\begin{align*}
&\lambda_{q}^{-\theta/20}L^{2}(\|v_q^1\|_{C_{t,1/6}B^{1/3-2\kappa}_{5/3,\infty}}+\|v_q^2\|_{C_tB^{1/5}_{5/3,\infty}}+\|v_q^1\|_{C^{1/6-\kappa}_{t,1/6}L^{5/3}}+\|v_q^2\|_{C^{1/10}_tL^{5/3}})\\&\lesssim\lambda_{q}^{-\theta/20}M_L^{1/2}.
\end{align*}
Here we used Lemma \ref{commutator}.
Hence,
we obtain
\begin{align}\label{est:vsharp}
&\|v_{q+1}^\sharp-v_q^\sharp\|_{C_{t,3/10}B_{5/3,\infty}^{3/5-\kappa}}\nonumber
\\&\lesssim \lambda_{q}^{-\theta/20}M_L^{1/2}+L^2\Big(\|v_{q+1}^1-v_q^1\|_{C_tB^{1/3-2\kappa}_{5/3,\infty}}+\|v_{q+1}^2-v_q^2\|_{C_tB^{1/5}_{5/3,\infty}}
\\&\quad+\|v_{q+1}^1-v_q^1\|_{C^{1/6-\kappa}_tL^{5/3}}+\|v_{q+1}^2-v_q^2\|_{C^{1/10}_tL^{5/3}}+\|v_{q+1}^\sharp-v_q^\sharp\|_{C_{t,3/10}B_{5/3,\infty}^{3/5-\kappa}}\Big)2^{-R/20}.\nonumber
\end{align}
Combining this bound with \eqref{est:vq2} we deduce a first part of \eqref{est:vq} and \eqref{vsharp2} and it remains to estimate the differences $v^{1}_{q+1}-v^{1}_{q}$ and $v^{\sharp}_{q+1}-v^{\sharp}_{q}$ in  $C_{t}L^{2}$.

\subsection{Estimate of the difference $v_{q+1}^1-v_q^1$ and $v^{\sharp}_{q+1}-v^{\sharp}_{q}$ in  $C_{t}L^{2}$}
\label{s:6.6}

By the Besov embedding Lemma~\ref{lem:emb} it holds
\begin{align*}
\|v_{q+1}^1-v_q^1\|_{C_tL^2}&\lesssim \|v_{q+1}^1-v_q^1\|_{C_tB_{5/3,\infty}^{1/3-2\kappa}},
\end{align*}
which in view of the first part of \eqref{est:vq} implies the second part of \eqref{est:vq}. Moreover, by \eqref{vsharp}, the Besov embedding Lemma \ref{lem:emb}, and paraproduct estimates Lemma \ref{lem:para} we
 have
\begin{align*}
&\|(v_q^1+v_q^2)\prec\cI (\nabla \Delta_{>R}(\Delta_{\leq f(q)}-\Delta_{\leq f(q+1)})z)\|_{C_{t}B^{1/3-2\kappa}_{5/3,\infty}}
\\&\lesssim L(\|v^1_{q}\|_{C_tL^{5/3}}+\|v^2_{q}\|_{C_tL^{5/3}})\|(\Delta_{\leq f(q)}-\Delta_{\leq f(q+1)})z\|_{C_tC^{-2/3-\kappa}}
\\&\lesssim L^{{3/2}}(\|v^1_{q}\|_{C_tL^{2}}+\|v^2_{q}\|_{C_tL^{5/3}})\lambda_{q}^{-\theta/6},
\end{align*}
which due to  $B^{1/3-2\kappa}_{5/3,\infty}\subset L^2$ and \eqref{est:vq} implies
\begin{align*}
&\|v_{q+1}^\sharp-v_q^\sharp\|_{C_tL^{2}}
\\&\lesssim \|v_{q+1}^1-v_q^1\|_{C_tL^2}+\|v^2_{q+1}-v^2_q\|_{C_tL^{5/3}}+L^{3/2}(\|v^1_{q}\|_{C_tL^{2}}+\|v^2_{q}\|_{C_tL^{5/3}})\lambda_{q}^{-\theta/6}\nonumber
\\&\lesssim \|v_{q+1}^2-v_q^2\|_{C_tB^{1/5}_{5/3,\infty}}+\|v_{q+1}^2-v_q^2\|_{C^{1/10}_tL^{5/3}}+M_L^{1/2}\lambda_{q}^{-\theta/20}.
\end{align*}
This gives the remaining  estimate of  \eqref{vsharp2}.

\subsection{Estimate of $v^{1}_{q}$ in $C_{t,3/8}L^4$}
\label{s:6.7}

This norm  may a priori blow up during the  iteration. We also estimate each term separately and apply the Schauder estimate Lemma~\ref{l:sch2} which then gives an additional factor $L$. We have
$$\|\Delta_{>R}\Delta_{\leq f(q)} (z+z^{\<20>})\|_{C_tC^{1/5+\kappa}}+\|\Delta_{>R}\Delta_{\leq f(q)} z^{\<101>}\|_{C_tC^{1/5+\kappa}}\lesssim L\lambda_{q}^{(7/10+2\kappa)\theta}.$$
Hence, by Lemma~\ref{lem:para} and \eqref{estimatev12} we have
\begin{align*}
&\|(v_q^1+v_q^2)\oprec\Delta_{>R}\Delta_{\leq f(q)} (z+z^{\<20>})\|_{C_tB^{-1+\kappa}_{4,\infty}}+\|(v_q^1+v_q^2)\prec\Delta_{>R}\Delta_{\leq f(q)} z^{\<101>}\|_{C_tB^{-1+\kappa}_{4,\infty}}
\\&\lesssim
\|v_q^1+v_q^2\|_{C_tB^{-6/5}_{4,\infty}}\Big(\|\Delta_{>R}\Delta_{\leq f(q)} (z+z^{\<20>})\|_{C_tC^{1/5+\kappa}}+\|\Delta_{>R}\Delta_{\leq f(q)} z^{\<101>}\|_{C_tC^{1/5+\kappa}}\Big)
\\&\lesssim L\lambda_{q}^{(7/10+2\kappa)\theta}\|v_q^1+v_q^2\|_{C_{t}L^{5/3}},
\end{align*}
\begin{align*}
&\|v_q^1\osucc \Delta_{>R}\Delta_{\leq f(q)} z\|_{C_{t,1/6}B^{-1+\kappa}_{4,\infty}}\\
&\lesssim\|v_q^1\|_{C_{t,1/6}B^{-1+\kappa}_{4,\infty}}\|\Delta_{>R}\Delta_{\leq f(q)} z\|_{C_tL^\infty}
\\&\lesssim L^{{1/2}}\lambda_{q}^{(1/2+2\kappa)\theta}\|v_q^1\|_{C_{t,1/6}B^{1/3-2\kappa}_{5/3,\infty}},
\end{align*}
and
\begin{align*}
&\|v_q^1\osucccurlyeq \Delta_{>R}z^{\<20>}\|_{C_{t,1/6}B_{4,\infty}^{-1+\kappa}}+\|v_q^1\succcurlyeq \Delta_{>R}z^{\<101>}\|_{C_{t,1/6}B_{4,\infty}^{-1+\kappa}}
\\&\lesssim\|v_q^1\osucccurlyeq \Delta_{>R}z^{\<20>}\|_{C_{t,1/6}B_{5/3,\infty}^{1/3-3\kappa}}+\|v_q^1\succcurlyeq \Delta_{>R}z^{\<101>}\|_{C_{t,1/6}B_{5/3,\infty}^{1/3-3\kappa}}
\\&\lesssim\|v_q^1\|_{C_{t,1/6}B_{5/3,\infty}^{1/3-2\kappa}}(\|z^{\<20>}\|_{C_tC^{-\kappa}}+\|z^{\<101>}\|_{C_tC^{-\kappa}})
\\&\lesssim\|v_q^1\|_{C_{t,1/6}B_{5/3,\infty}^{1/3-2\kappa}}L.
\end{align*}
Using Lemma \ref{l:sch2}  and \eqref{vsharp1} we have
\begin{align*}
&\|\cI\div(v_q^\sharp\varocircle\Delta_{>R}z)\|_{C_{t,3/10}L^4}
\lesssim\|\cI\div(v_q^\sharp\varocircle\Delta_{>R}z)\|_{C_{t,3/10}B^\kappa_{4,\infty}}
\\&\lesssim L\|v_q^\sharp\varocircle\Delta_{>R}z\|_{C_{t,3/10}B^{-1+2\kappa}_{4,\infty}}
\\&\lesssim L\|v_q^\sharp\|_{C_{t,3/10}B_{5/3,\infty}^{3/5{-\kappa}}}\|\Delta_{>R}z\|_{C_tC^{-1/2-\kappa}}.
\end{align*}
By commutator estimates Lemma \ref{lem:com1} we have
\begin{align*}
&\|([\mathbb{P},v_q^1\prec]\cI(\nabla z))\varocircle \Delta_{>R}z\|_{C_{t,1/6}B^{1/3-4\kappa}_{5/3,\infty}}+\|\mathrm{com}(v_q^1,\mathbb{P}\cI(\nabla z),\Delta_{>R}z)\|_{C_{t,1/6}B^{1/3-4\kappa}_{5/3,\infty}}
\\&\lesssim\|v_q^1\|_{C_{t,1/6}B_{5/3,\infty}^{1/3-2\kappa}}\|z\|_{C_tC^{-1/2-\kappa}}^2\lesssim\|v_q^1\|_{C_{t,1/6}B_{5/3,\infty}^{1/3-2\kappa}}L.%
\end{align*}
Finally, for the initial value part we apply Lemma 9 in \cite{DV15}
to obtain
$$\|e^{t \Delta}v_{0}\|_{L^4}\lesssim t^{-3/8}\|v_{0}\|_{L^2}.$$
Combining the above estimates and applying the Schauder estimate Lemma \ref{l:sch2}, the Besov embedding Lemma \ref{lem:emb} as well as \eqref{estimatev12} and \eqref{vsharp1} and the definition of $M_{L}$ we obtain \eqref{estimatev42}.

Note that we  only control the  $L^4$-norm  instead of e.g. $L^\infty$ because the paraproduct $$v_q^1\osucc \Delta_{\leq f(q)}\Delta_{>R}z$$ only belongs to $B^{1/3-2\kappa}_{5/3,\infty}$.

\section{The main iteration -- proof of Proposition \ref{p:iteration p}}
\label{ss:it2}

\subsection{Choice of parameters}\label{s:c2}

In the sequel, additional parameters will be indispensable and their value has to be carefully chosen in order to respect all the compatibility conditions appearing in the estimations below. First, for a sufficiently small  $\alpha\in (0,1)$ to be chosen below,  we let $\ell\in (0,1)$ be  a small parameter  satisfying
\begin{equation}\label{ell}
 {\ell^{4/5} }\lambda_q^4\leq \lambda_{q+1}^{-\alpha},\quad \ell^{-1}\leq \lambda_{q+1}^{2\alpha},
\end{equation}
In particular, we define
\begin{equation}\label{ell1}\ell:=\lambda_{q+1}^{-\frac{3\alpha}{2}}\lambda_q^{-2}.\end{equation}

In the sequel, we  use the following bounds
\begin{align}\label{parameter2}
	\alpha>244\beta b,\quad 1> {168}\beta b^2,\quad \frac1{35}-33\alpha>2\beta b,\quad \alpha b>128
\end{align}
	which can be obtained by choosing $\alpha$ small  such that
$\frac1{35}-33\alpha>\alpha,$
	and choosing $b\in  \mathbb{N}$ large enough such that
	$\alpha b>128$ and finally choosing $\beta$ small such that
	$\alpha>244\beta b$, $1> {168}\beta b^2$.  Various estimates of this form are needed for the final  control the new stress $\mathring{R}_{q+1}$. Hence, we shall choose $\alpha$ small first and $b$ large, then $\beta$  small enough. The last free parameter is $a$  which is power of $2^{21}$ and satisfies the lower bounds given through
$$a>4M_L+K,\quad  L\leq a^{\alpha/16}.$$
Then by our condition we have
\begin{align}\label{parameter1}
M_L(1+3q)+K\leq \lambda_q^{ {1/42}}< \ell^{-2/183}<\lambda_{q+1}^{\alpha-2\beta},\quad \sigma_q^{-1}<\ell^{-1/61}.\end{align}
 In the sequel, we increase $a$ in order to absorb various implicit and universal constants.

We may freely increase the value of $a$  provided we make $\beta$ smaller at the same time. %

\subsection{Mollification}\label{s:pp}

We intend to replace $v^2_q$ by a mollified velocity field $v_\ell$. To this end, we extend
\begin{align*}z(t)=z(0),\quad z^{\<20>}(t)=z^{\<20>}(0)=0,\quad z^{\<101>}(t)=0,
\\\mathcal{I}(\nabla z)(t)=0, \quad v_q^i(t)=v^i_q(0),\quad\mathring{R}_q(t)= \mathring{R}_q(0)\qquad \mbox{for } t<0.\end{align*}
 As $v_q^2$ equals to zero near zero, $\partial_tv_q^2(0)=0$, which implies by our extension that the equation holds also for $t<0$. Let $\{\phi_\varepsilon\}_{\varepsilon>0}$ be a family of standard mollifiers on $\mathbb{R}^3$, and let $\{\varphi_\varepsilon\}_{\varepsilon>0}$ be a family of  standard mollifiers with support on $\mathbb{R}^+$. We define a mollification of $v_q$, $\mathring{R}_q$  in space and time by convolution as follows
$$v_\ell=(v^2_q*_x\phi_\ell)*_t\varphi_\ell,\qquad
\mathring{R}_\ell=(\mathring{R}_q*_x\phi_\ell)*_t\varphi_\ell,$$
where $\phi_\ell=\frac{1}{\ell^3}\phi(\frac{\cdot}{\ell})$ and $\varphi_\ell=\frac{1}{\ell}\varphi(\frac{\cdot}{\ell})$.
Since the mollifier $\varphi_\ell$ is supported on $\mathbb{R}^+$, it is easy to see that $z_\ell$ is $(\mathcal{F}_t)_{t\geq0}$-adapted and so are $v_\ell$ and $\mathring{R}_\ell$.
Then using the equation for $v_q^2$ we obtain that $(v_\ell,\mathring{R}_\ell)$ satisfies
\begin{equation}\label{mollification3p}
\aligned
\partial_tv_\ell -\Delta v_\ell+\div N+\nabla p_\ell&=\div (\mathring{R}_\ell+R_{\textrm{com}})
\\\div v_\ell&=0,
\endaligned
\end{equation}
where $N$ is the trace-free part of the following matrix
$$
N=(v_q^1+v_q^2)\otimes (v_q^1+v_q^2)+V_{q}^{2}+V^{2,*}_{q}
$$
and $$R_{\textrm{com}}=N-N*_x\phi_\ell*_t\varphi_\ell.$$

By using \eqref{inductionv2p} \eqref{inductionv C1p} and \eqref{ell1} we know for $t\in[0, T_L]$
\begin{align}\label{error}
\|v^2_q-v_\ell\|_{C_{t}^{1/10}L^{5/3}}+\|v^2_q-v_\ell\|_{C_{t}W^{1/5,5/3}}&\lesssim\ell^{4/5}\|v^2_q\|_{C^1_{t,x}}
\leq \ell^{4/5}\lambda_q^4M_L^{1/2}
\\&\leq M_L^{1/2}\lambda_{q+1}^{-\alpha}
\leq\frac{1}{4} M_L^{1/2}\delta_{q+1}^{1/2}a^{-\alpha/2},\nonumber
\end{align}
\begin{equation}\label{error1p}
\|v^2_q-v_\ell\|_{C_tL^2}\lesssim \ell\|v^2_q\|_{C^1_{t,x}}\leq \ell\lambda_q^4 M_L^{1/2}\leq M_L^{1/2} \lambda_{q+1}^{-\alpha}\leq \frac{1}{4} M_L^{1/2}\delta_{q+1}^{1/2},
\end{equation}
and for $p\in [1,\infty]$
\begin{equation}\label{error11p}
\|v^2_q-v_\ell\|_{C_tW^{2/3,p}}\lesssim \ell^{1/3}\|v^2_q\|_{C^1_{t,x}}\leq \ell^{1/3}\lambda_q^4 M_L^{1/2}\leq M_L^{1/2} \lambda_{q+1}^{-{15\alpha}/{32}}\leq \frac{1}{4} M_L^{1/2}\delta_{q+1}^{1/2}a^{-\alpha/2},
\end{equation}
where we  used the fact that $\alpha b>128$ and $\alpha>3\beta$ and we  chose $a$ large enough in order to absorb the implicit constant.
In addition,
\begin{equation}\label{error13p}
\|v_\ell\|_{C^{N}_{t,x}}\lesssim\ell^{-N+1}\|v_q\|_{C^1_{t,x}}\leq \ell^{-N+1}\lambda_{q}^{4}M_{L}^{1/2}\lesssim M_L^{1/2}\ell^{-N}\lambda_{q+1}^{-\alpha}
\end{equation}
holds for $t\in[0, T_L]$ and
it holds for $t\in(\frac{\sigma_q}2\wedge T_L, T_L]$
\begin{equation}\label{eq:vl psp}
\|v_\ell(t)\|_{L^2}\leq \|v_q^2\|_{C_{t}L^2}\lesssim  M_0(M_L^{1/2}+K^{1/2})+3M_0M_L^{1/2}(1+3q)^{1/2}
\end{equation}
with some universal implicit constant.

\subsection{Construction of $v^{2}_{q+1}$.}
Let us proceed with the construction of the perturbation $w_{q+1}$, which then defines the next iteration by $v_{q+1}^2:=v_\ell+w_{q+1}$.
To this end, we make use of the  intermittent jets \cite[Section 7.4]{BV19}, which we recall in Appendix~\ref{s:B}. In particular, the building blocks $W_{(\xi)}=W_{\xi,r_\perp,r_\|,\lambda,\mu}$ for $\xi\in\Lambda$ are defined in (\ref{intermittent}) and the set $\Lambda$ is introduced in Lemma \ref{geometric}.
The necessary estimates are collected  in \eqref{bounds}.  For the intermittent jets
we choose the following parameters
\begin{equation}\label{parameter}
\aligned
\lambda&=\lambda_{q+1},
\qquad
r_\|=\lambda_{q+1}^{-4/7},
\qquad r_\perp=r_\|^{-1/4}\lambda_{q+1}^{-1}=\lambda_{q+1}^{-6/7},
\qquad
\mu=\lambda_{q+1}r_\|r_\perp^{-1}=\lambda_{q+1}^{9/7}.
\endaligned
\end{equation}
Since $a$ is power of $2^{21}$, $\lambda_{q+1}r_\perp= a^{(b^{q+1})/7}\in\mathbb{N}$.

 Now we follow \cite[Section 5.2]{HZZ21} and introduce  $\rho$ as follows
$$
\rho:=2\sqrt{\ell^2+|\mathring{R}_\ell|^2}+\frac{\gamma_{q+1}}{(2\pi)^3},
$$
which implies for $p\geq1$
\begin{equation}\label{rho psp}
\|\rho(t)\|_{L^p}\leq 2\ell (2\pi)^{3/p}+2\|\mathring{R}_\ell(t)\|_{L^p}+\gamma_{q+1}.
\end{equation}
In view of \eqref{eq:Rp} which holds on $(2\sigma_q\wedge T_L,T_L]$ and since $\mathrm{supp}\varphi_\ell\subset [0,\ell]$, we obtain for $t\in (4\sigma_q\wedge T_L,T_L]$
\begin{equation}\label{rho0 spp}
\|\rho\|_{C^{0}_{[4\sigma_q\wedge T_L,t],x}}\lesssim \ell^{-4}\delta_{q+1}M_L+\gamma_{q+1},
\end{equation}
where we also used the embedding $W^{4,1}\subset L^\infty$.
Then, we deduce similarly as \cite[(3.25)]{HZZ21} for $N\geq1$ and $t\in (4\sigma_q\wedge T_L,T_L]$
\begin{equation}\label{rhoN spp}
\aligned
\|\rho\|_{C^N_{[4\sigma_q\wedge T_L,t],x}}& \lesssim \ell^{-4-N}M_L\delta_{q+1}+\ell^{-N+1}(\ell^{-5}M_L\delta_{q+1})^N+\gamma_{q+1}\lesssim
\ell^{2 - 7 N} \delta_{q + 1} M_L+\gamma_{q+1}.
\endaligned
\end{equation}
For a general $t\in [0,T_L]$, we have by \eqref{bd:Rp}
\begin{equation}\label{rho0 sp1p}
\|\rho\|_{C^{0}_{t,x}}\lesssim \ell^{-4}M_L(1+3q)+\gamma_{q+1},
\end{equation}
and for $N\geq 1$
\begin{equation}\label{rhoN sp1p}
\aligned
\|\rho\|_{C^N_{t,x}}& \lesssim
\ell^{2 - 7 N}  M_L(1+3q)+\gamma_{q+1},
\endaligned
\end{equation}
where we used $M_L(1+3q)\leq \ell^{-1}$.

Next, we define the amplitude functions
\begin{equation}\label{amplitudes}a_{(\xi)}(\omega,t,x):=a_{\xi,q+1}(\omega,t,x):=\rho(\omega,t,x)^{1/2}\gamma_\xi\left(\Id
-\frac{\mathring{R}_\ell(\omega,t,x)}{\rho(\omega,t,x)}\right)(2\pi)^{-{3}/{4}},\end{equation}
where $\gamma_\xi$ is introduced in  Lemma \ref{geometric}.
Since $\rho$ and $\mathring{R}_\ell$ are $(\mathcal{F}_t)_{t\geq0}$-adapted, we know that also $a_{(\xi)}$ is $(\mathcal{F}_t)_{t\geq0}$-adapted.
By  (\ref{geometric equality}) we have
\begin{equation}\label{cancellation}(2\pi)^{-\frac{3}{2}}\sum_{\xi\in\Lambda}a_{(\xi)}^2\int_{\mathbb{T}^3}W_{(\xi)}\otimes W_{(\xi)}\dif x=\rho \Id-\mathring{R}_\ell,\end{equation}
By using (\ref{rho psp}) for $t\in(4\sigma_q\wedge T_L, T_L]$
\begin{equation}\label{estimate a psp}
\begin{aligned}
\|a_{(\xi)}(t)\|_{L^2}&\leq \|\rho(t)\|_{L^1}^{1/2}\|\gamma_\xi\|_{C^0(B_{1/2}(\Id))}\\
&\leq\frac{M}{8|\Lambda|(1+8\pi^{3})^{1/2}}\left(2(2\pi)^3\ell+2\delta_{q+1}M_L+\gamma_{q+1}\right)^{1/2}\\
&\leq\frac{M}{4|\Lambda|}(M_L^{1/2}\delta_{q+1}^{1/2}+\gamma_{q+1}^{1/2}),
\end{aligned}
\end{equation}
and for $t\in [0,T_L]$
\begin{equation*}
	\begin{aligned}
	\|a_{(\xi)}(t)\|_{C_tL^2}
	&\leq\frac{M}{4|\Lambda|}(M_L^{1/2}(1+3q)^{1/2}+\gamma_{q+1}^{1/2}),
	\end{aligned}
	\end{equation*}
where   $M$ denotes the universal constant from Lemma~\ref{geometric}.
From \eqref{rho0 spp}, \eqref{rhoN spp} and similarly to \cite[(3.30)]{HZZ21} we deduce for $t\in (4\sigma_q\wedge T_L,T_L]$
\begin{align*}
\|\rho^{1/2}\|_{C^0_{[4\sigma_q\wedge T_L,t],x}}\lesssim \ell^{-2}\delta_{q+1}^{1/2}M_L^{1/2}+\gamma_{q+1}^{1/2},
\end{align*}
and for $m=1,\dots, N$ using $K\leq \ell^{-1}$
\begin{align*}
\|\rho^{1/2}\|_{C^m_{[4\sigma_q\wedge T_L,t],x}}\lesssim \ell^{1-7m}\delta_{q+1}^{1/2}M_L^{1/2}+\ell^{1/2-m}(\gamma_{q+1}+\ell^{-5}\delta_{q+1}M_L)^{m}\leq \ell^{1-7m}(\delta_{q+1}^{1/2}M_L^{1/2}+\gamma_{q+1}^{1/2}).
\end{align*}
This implies  for $N\in \mN_{0}$ as in \cite[(3.34)]{HZZ21}
\begin{equation}\label{estimate aN psp}
\begin{aligned}
\|a_{(\xi)}\|_{C^N_{[4\sigma_q\wedge T_L,t],x}}
&\lesssim \ell^{-8-7N}(\delta_{q+1}^{1/2}M_L^{1/2}+\gamma_{q+1}^{1/2}).
\end{aligned}
\end{equation}
For a general $t\in [0,T_L]$ we have for $N\in \mN_{0}$
\begin{equation}\label{estimate aN ps1p}
\begin{aligned}
\|a_{(\xi)}\|_{C^N_{t,x}}
&\lesssim \ell^{-8-7N}(M_L^{1/2}(1+3q)^{1/2}+\gamma_{q+1}^{1/2}),
\end{aligned}
\end{equation}
where we used $M_L(1+3q)+K\leq \ell^{-1}$.

Let us introduce a smooth cut-off function
\begin{align*}
\chi(t)=\begin{cases}
0,& t\leq \frac{\sigma_q}2,\\
\in (0,1),& t\in (\frac{\sigma_q}2,{\sigma_q} ),\\
1,&t\geq {\sigma_q}.
\end{cases}
\end{align*}
Note that
$\|\chi'\|_{C^{0}_{t}}\leq 2^{q+1}$ which has to be taken into account in the estimates of the $C_t^{1/10}L^{5/3}$ and $C^{1}_{t,x}$-norms in \eqref{principle est1np}- \eqref{temporal est2 psp} below.

With these preparations in hand,  we define the principal part $w_{q+1}^{(p)}$ of the perturbation $w_{q+1}$ as
\begin{equation}\label{principle}
w_{q+1}^{(p)}:=\sum_{\xi\in\Lambda} a_{(\xi)}W_{(\xi)}.
\end{equation}
Since the coefficients $a_{(\xi)}$ are $(\mathcal{F}_t)_{t\geq0}$-adapted and $W_{(\xi)}$ is a deterministic function we deduce that
$w_{q+1}^{(p)}$ is also $(\mathcal{F}_t)_{t\geq0}$-adapted.
Moreover, according to (\ref{cancellation}) and (\ref{Wxi}) it follows that
\begin{equation}\label{can}w_{q+1}^{(p)}\otimes w_{q+1}^{(p)}+\mathring{R}_\ell=\sum_{\xi\in \Lambda}a_{(\xi)}^2 \mathbb{P}_{\neq0}(W_{(\xi)}\otimes W_{(\xi)})+\rho \Id,
\end{equation}
where we use the notation $\mathbb{P}_{\neq0}f:=f-\mathcal{F}f(0)=f-(2\pi)^{-3/2}\int_{\mathbb{T}^3}f\dif x$.

We also define the incompressibility corrector by
\begin{equation}\label{incompressiblity}
w_{q+1}^{(c)}:=\sum_{\xi\in \Lambda}\textrm{curl}(\nabla a_{(\xi)}\times V_{(\xi)})+\nabla a_{(\xi)}\times \textrm{curl}V_{(\xi)}+a_{(\xi)}W_{(\xi)}^{(c)},\end{equation}
with $W_{(\xi)}^{(c)}$ and $V_{(\xi)}$ being given in (\ref{corrector}).
Since $a_{(\xi)}$ is $(\mathcal{F}_t)_{t\geq0}$-adapted and $W_{(\xi)}, W_{(\xi)}^{(c)}$ and $V_{(\xi)}$ are  deterministic functions we know that
$w_{q+1}^{(c)}$ is also $(\mathcal{F}_t)_{t\geq0}$-adapted. 
By a direct computation we deduce that
\begin{equation*}
w_{q+1}^{(p)}+w_{q+1}^{(c)}=\sum_{\xi\in\Lambda}\textrm{curl}\,\textrm{curl}(a_{(\xi)}V_{(\xi)}),
\end{equation*}
hence
\begin{equation*}\div(w_{q+1}^{(p)}+w_{q+1}^{(c)})=0.\end{equation*}
We also introduce a temporal corrector
\begin{equation}\label{temporal}w_{q+1}^{(t)}:=-\frac{1}{\mu}\sum_{\xi\in \Lambda}\mathbb{P}\mathbb{P}_{\neq0}\left(a_{(\xi)}^2\phi_{(\xi)}^2\psi_{(\xi)}^2\xi\right),\end{equation}
where $\mathbb{P}$ is the Helmholtz projection. 
Similarly to above $w_{q+1}^{(t)}$ is $(\mathcal{F}_t)_{t\geq0}$-adapted and by similar computation as \cite[(7.38)]{BV19} we obtain
\begin{equation}\label{equation for temporal}
\aligned
&\partial_t w_{q+1}^{(t)}+\sum_{\xi\in\Lambda}\mathbb{P}_{\neq0}\left(a_{(\xi)}^2\div(W_{(\xi)}\otimes W_{(\xi)})\right)
\\
&\qquad= -\frac{1}{\mu}\sum_{\xi\in\Lambda}\mathbb{P}\mathbb{P}_{\neq0}\partial_t\left(a_{(\xi)}^2\phi_{(\xi)}^2\psi_{(\xi)}^2\xi\right)
+\frac{1}{\mu}\sum_{\xi\in\Lambda}\mathbb{P}_{\neq0}\left( a^2_{(\xi)}\partial_t(\phi^2_{(\xi)}\psi^2_{(\xi)}\xi)\right)
\\&\qquad= (\Id-\mathbb{P})\frac{1}{\mu}\sum_{\xi\in\Lambda}\mathbb{P}_{\neq0}\partial_t\left(a_{(\xi)}^2\phi_{(\xi)}^2\psi_{(\xi)}^2\xi\right)
-\frac{1}{\mu}\sum_{\xi\in\Lambda}\mathbb{P}_{\neq0}\left(\partial_t a^2_{(\xi)}(\phi^2_{(\xi)}\psi^2_{(\xi)}\xi)\right).
\endaligned
\end{equation}
Note that the first term on the right hand side can be viewed as a pressure term.

We define the truncated perturbations $\tilde w^{(p)}_{q+1},$ $ \tilde w_{q+1}^{(c)},$ $ \tilde w_{q+1}^{(t)}$ as follows
$$\tilde w^{(p)}_{q+1}:=w^{(p)}_{q+1}\chi,\quad \tilde w^{(c)}_{q+1}:=w^{(c)}_{q+1}\chi,\quad \tilde w^{(t)}_{q+1}:=w^{(t)}_{q+1}\chi^2.$$
Define ${w}_{q+1}:=\tilde{w}_{q+1}^{(p)}+\tilde{w}_{q+1}^{(p)}+\tilde{w}_{q+1}^{(t)}$ and
$$
v^{2}_{q+1}=v_{\ell}+w_{q+1}=v_{\ell}+\tilde{w}^{(p)}_{q+1}+\tilde{w}^{(c)}_{q+1}+\tilde{w}^{(t)}_{q+1}.
$$
We note that by construction $v^2_{q+1}$ is $(\mathcal{F}_t)_{t\geq0}$ adapted.

\subsection{Verification of the inductive estimates for $v^2_{q+1}$}

 By \eqref{estimate a psp} and \eqref{estimate aN psp} and {similar argument as \cite{HZZ21}}  we obtain
for $t\in(4\sigma_q\wedge T_L, T_L]$ and some universal constant $M_0\geq 1$
\begin{equation}\label{estimate wqp psp}
\|\tilde w_{q+1}^{(p)}(t)\|_{L^2}\lesssim \sum_{\xi\in\Lambda}\frac{1}{4|\Lambda|}M(M_L^{1/2}\delta_{q+1}^{1/2}+\gamma_{q+1}^{1/2})\|W_{(\xi)}\|_{C_tL^2}\leq \frac{M_0}{2}(M_L^{1/2}\delta_{q+1}^{1/2}+\gamma_{q+1}^{1/2}),
\end{equation}
where we %
used $150\alpha<\frac17$
and for $t\in (\frac{\sigma_q}2\wedge T_L,4\sigma_q\wedge T_L]$
\begin{equation}\label{estimate wqp ps1p}
\|\tilde w_{q+1}^{(p)}(t)\|_{L^2}\leq \frac{M_0}{2}((M_L(1+3q))^{1/2}+\gamma_{q+1}^{1/2}).
\end{equation}

Similarly as in \cite[(3.43)-(3.46)]{HZZ21},  we apply (\ref{bounds}) and (\ref{estimate aN psp}) for general $L^p$-norms to deduce for $t\in(4\sigma_q\wedge T_L, T_L]$, $p\in(1,\infty)$
\begin{equation}\label{principle est1 psp}
\aligned
\|\tilde w_{q+1}^{(p)}(t)\|_{L^p}&\lesssim \ell^{-8}(M_L^{1/2}\delta_{q+1}^{1/2}+\gamma_{q+1}^{1/2})r_\perp^{2/p-1}r_\|^{1/p-1/2},
\endaligned
\end{equation}
\begin{equation}\label{correction est psp}
\aligned
\|\tilde w_{q+1}^{(c)}(t)\|_{L^p}&\lesssim \ell^{-22}(M_L^{1/2}\delta_{q+1}^{1/2}+\gamma_{q+1}^{1/2})r_\perp^{2/p}r_\|^{1/p-3/2},
\endaligned
\end{equation}
and
\begin{equation}\label{temporal est1 psp}
\aligned
\|\tilde w_{q+1}^{(t)}(t)\|_{L^p}&\lesssim  \ell^{-16}(M_L\delta_{q+1}+\gamma_{q+1})r_\perp^{2/p-1}r_\|^{1/p-2}\lambda_{q+1}^{-1},
\endaligned
\end{equation}
\begin{equation}\label{corr temporal psp}
\aligned
&\|\tilde w_{q+1}^{(c)}(t)\|_{L^p}+\|\tilde w_{q+1}^{(t)}(t)\|_{L^p}\lesssim\ell^{-8} (M_L^{1/2}\delta_{q+1}^{1/2}+\gamma_{q+1}^{1/2})r_\perp^{2/p-1}r_\|^{1/p-1/2}.
\endaligned
\end{equation}
For a general $t\in (\frac{\sigma_q}2\wedge T_L,4\sigma_q\wedge T_L]$ we have
\begin{equation}\label{principle est1 ps1p}
\aligned
\|\tilde w_{q+1}^{(p)}(t)\|_{L^p}&\lesssim \ell^{-8}((M_L(1+3q))^{1/2}+\gamma_{q+1}^{1/2})r_\perp^{2/p-1}r_\|^{1/p-1/2},
\endaligned
\end{equation}
\begin{equation}\label{correction est ps1p}
\aligned
\|\tilde w_{q+1}^{(c)}(t)\|_{L^p}&\lesssim \ell^{-22}((M_L(1+3q))^{1/2}+\gamma_{q+1}^{1/2})r_\perp^{2/p}r_\|^{1/p-3/2},
\endaligned
\end{equation}
and
\begin{equation}\label{temporal est1 ps1p}
\aligned
\|\tilde w_{q+1}^{(t)}(t)\|_{L^p}&\lesssim  \ell^{-16}((M_L(1+3q))+\gamma_{q+1})r_\perp^{2/p-1}r_\|^{1/p-2}\lambda_{q+1}^{-1},
\endaligned
\end{equation}
\begin{equation}\label{corr temporal ps1p}
\aligned
&\|\tilde w_{q+1}^{(c)}(t)\|_{L^p}+\|\tilde w_{q+1}^{(t)}(t)\|_{L^p}\lesssim \ell^{-8}((M_L(1+3q))^{1/2}+\gamma_{q+1}^{1/2})r_\perp^{2/p-1}r_\|^{1/p-1/2}.
\endaligned
\end{equation}

Combining (\ref{estimate wqp psp}), (\ref{correction est psp}) and (\ref{temporal est1 psp}) we obtain for $t\in(4\sigma_q\wedge T_L, T_L]$
\begin{equation}\label{estimate wq psp}
\aligned
\|w_{q+1}(t)\|_{L^2}&\leq (M_L^{1/2}\delta_{q+1}^{1/2}+\gamma_{q+1}^{1/2})\left(\frac{M_0}{2}+C\lambda_{q+1}^{44\alpha-2/7}+C(M_L^{1/2}\delta_{q+1}^{1/2}+\gamma_{q+1}^{1/2})\lambda_{q+1}^{32\alpha-1/7}\right)
\\&\leq \frac34M_{0}(M_L^{1/2}\delta_{q+1}^{1/2}+\gamma_{q+1}^{1/2}),
\endaligned
\end{equation}
and
for $t\in(\frac{\sigma_q}2\wedge T_L, 4\sigma_q\wedge T_L]$
\begin{equation}\label{estimate wq ps1p}
\aligned
\|w_{q+1}(t)\|_{L^2}
&\leq \frac34M_{0}((M_L(1+3q))^{1/2}+\gamma_{q+1}^{1/2}),
\endaligned
\end{equation}
where we used $M_L(1+3q)+K<\ell^{-1}$ and the condition on $\alpha$.

With these bounds, we have all in hand to complete the proof of Proposition~\ref{p:iteration p}. We split the details  into several subsections.

\subsubsection{Proof of \eqref{iteration psp}}
First, \eqref{estimate wq psp} together with \eqref{error1p} yields for $t\in(4\sigma_q\wedge T_L, T_L]$
$$
\|v_{q+1}^2(t)-v_{q}^2(t)\|_{L^{2}}\leq \|w_{q+1}(t)\|_{L^{2}}+\|v_{\ell}(t)-v_{q}^2(t)\|_{L^{2}}\leq M_0(M_L^{1/2} \delta_{q+1}^{1/2}+\gamma_{q+1}^{1/2}).
$$
For $t\in (\frac12\sigma_q\wedge T_L, 4\sigma_q\wedge T_L]$
we use \eqref{estimate wq ps1p}, \eqref{error1p} to obtain
$$
\|v_{q+1}^2(t)-v_{q}^2(t)\|_{L^{2}}\leq \|w_{q+1}(t)\|_{L^{2}}+\|v_{\ell}(t)-v_{q}^2(t)\|_{L^{2}}\leq M_0((M_L(1+3q))^{1/2}+\gamma_{q+1}^{1/2}).
$$
For $t\in  [0,\frac12\sigma_q\wedge T_L]$ it holds $\chi(t)=0$ as well as $v_{q}^2(t)=0$ by \eqref{inductionv psp} implying
$$
\|v_{q+1}^2-v_{q}^2\|_{C_{t}L^{2}}=\|v_{\ell}-v_{q}^2\|_{C_{t}L^{2}}=0.%
$$
Hence \eqref{iteration psp} follows.

\subsubsection{Proof that \eqref{iteration psp} implies \eqref{inductionv psp} on the level $q+1$}
From \eqref{iteration psp} we find for $t\in [0,\frac12\sigma_q\wedge T_L]$ %
$$
\|v^{2}_{q+1}\|_{C_{t}L^{2}}\leq \sum_{r=0}^{q}\|v^{2}_{r+1}-v^{2}_{r}\|_{C_{t}L^{2}}=0,%
$$
proving the second bound in \eqref{inductionv psp} on the level $q+1$.
For the first bound in \eqref{inductionv psp} on the level $q+1$, we obtain in view of \eqref{iteration psp} for $t\in (\frac12\sigma_q\wedge T_L,T_{L}]$
$$
\begin{aligned}
\|v^{2}_{q+1}(t)\|_{L^{2}}&\leq \sum_{0\leq r\leq q}\|v^{2}_{r+1}(t)-v^{2}_{r}(t)\|_{L^{2}}\\
&\leq M_{0} \left(M_{L}^{1/2}\sum_{0\leq r\leq q}\delta_{r+1}^{1/2}+\sum_{0\leq r\leq q}(M_L(1+3r))^{1/2}\mathbf{1}_{t\in (\frac{\sigma_{r}}{2}\wedge T_{L},4\sigma_{r}\wedge T_{L}]}+\sum_{0\leq r\leq q}\gamma_{r+1}^{1/2}\right)\\&\leq M_{0} \left(M_{L}^{1/2}\sum_{0\leq r\leq q}\delta_{r+1}^{1/2}+3(M_L(1+3q))^{1/2}+\sum_{0\leq r\leq q}\gamma_{r+1}^{1/2}\right),
\end{aligned}
$$
where we used the fact that by the definition of $\sigma_{r}=2^{-r}$ each $t\in [0,T_L]$ only belongs to three intervals $(\frac{\sigma_{r}}{2}\wedge T_{L},4\sigma_{r}\wedge T_{L}]$.  Hence \eqref{inductionv psp} follows.

\subsubsection{Proof of \eqref{induction wtp}  and the second inequality in \eqref{inductionv2p} on the level $q+1$}
\label{s:743}

In this section, we see in particular how the definition of intermittent jets determines the integrability $5/3$ which we use throughout the paper.
It holds by \eqref{estimate aN ps1p}, \eqref{bounds} and the choice of parameters in \eqref{parameter2}
\begin{equation}\label{principle est1np}
\aligned
&\|\tilde w_{q+1}^{(p)}\|_{C^{1/{10}}_tL^{5/3}}+\|\tilde w_{q+1}^{(p)}\|_{C_tW^{1/5,5/3}}
\\&\lesssim \sum_{\xi\in \Lambda}\|a_{(\xi)}\|_{C^1_{t,x}}(\|W_{(\xi)}\|_{C^{1/{10}}_tL^{5/3}}+\|W_{(\xi)}\|_{C_tW^{1/5,5/3}})2^{q+1}
\\&\lesssim (M_L^{1/2}(1+3q)^{1/2}+\gamma_{q+1}^{1/2})\ell^{-15}r_\perp^{1/5}r_\|^{1/10}((r_\perp\lambda_{q+1}\mu/r_\|)^{1/10}+\lambda_{q+1}^{1/5})2^{q+1}
\\&\lesssim M_L^{1/2}\lambda_{q+1}^{32\alpha-1/35},
\endaligned
\end{equation}
\begin{equation}\label{correction estnp}
\aligned
&\|\tilde w_{q+1}^{(c)}\|_{C^{1/10}_tL^{5/3}}+\|\tilde w_{q+1}^{(c)}\|_{C_tW^{1/5,5/3}}\\&\lesssim\sum_{\xi\in \Lambda}2^{q+1}\Big(\|a_{(\xi)}\|_{C^1_{t,x}}(\|W_{(\xi)}^{(c)}\|_{C^{1/10}_tL^{5/3}}+\|W_{(\xi)}^{(c)}\|_{C_tW^{1/5,5/3}})
\\&\quad+\|a_{(\xi)}\|_{C^3_{t,x}}(\|V_{(\xi)}\|_{C^{1/10}_tW^{1,5/3}}+\|V_{(\xi)}\|_{C_tW^{6/5,5/3}})\Big)
\\&\lesssim (M_L^{1/2}(1+3q)^{1/2}+\gamma_{q+1}^{1/2})\ell^{-29}r_\perp^{1/5}r_\|^{1/10}\left(r_\perp r_\|^{-1}+\lambda_{q+1}^{-1}\right)\Big((r_\perp\lambda_{q+1}\mu/r_\|)^{1/10}+\lambda_{q+1}^{1/5}\Big)2^{q+1}
\\&\lesssim M_L^{1/2}\ell^{-30}r_\perp^{6/5}r_\|^{-9/10}\Big((r_\perp\lambda_{q+1}\mu/r_\|)^{1/10}+\lambda_{q+1}^{1/5}\Big)
\\&\lesssim M_L^{1/2}\lambda_{q+1}^{60\alpha-11/35},
\endaligned
\end{equation}
\begin{equation}\label{temporal est1np}
\aligned
&\|\tilde w_{q+1}^{(t)}\|_{C^{1/10}_tL^{5/3}}+\|\tilde w_{q+1}^{(t)}\|_{C_tW^{1/5,5/3}}
\\&\leq\frac{1}{\mu}2^{q+1}\sum_{\xi\in\Lambda}
\Big(\|a_{(\xi)}\|_{C^0_{t,x}}\|a_{(\xi)}\|_{C^1_{t,x}}\|\phi_{(\xi)}\|_{L^{10/3}}^2\|\psi_{(\xi)}\|_{C_tL^{10/3}}^2\\
&\quad+\|a_{(\xi)}\|_{C^0_{t,x}}^2\|\phi_{(\xi)}\|_{L^{10/3}}\| \phi_{(\xi)}\|_{W^{1/5,10/3}}\|\psi_{(\xi)}\|_{C_tL^{10/3}}^2\\
&\quad+\|a_{(\xi)}\|_{C^0_{t,x}}^2\|\phi_{(\xi)}\|_{L^{10/3}}^2\|\psi_{(\xi)}\|_{C_tW^{1/5,10/3}}\|\psi_{(\xi)}\|_{C_tL^{10/3}}
\\
&\quad+\|a_{(\xi)}\|_{C^0_{t,x}}^2\|\phi_{(\xi)}\|_{L^{10/3}}^2\|\psi_{(\xi)}\|_{C^{1/10}_tL^{10/3}}\|\psi_{(\xi)}\|_{C_tL^{10/3}}\Big)
\\&\lesssim (M_L(1+3q)+\gamma_{q+1}) \ell^{-23}r_\perp^{1/5}r_\|^{-7/5}(\mu^{-1}r_\perp^{-1}r_\|)\\
&\qquad\times\Big((r_\perp\lambda_{q+1}\mu/r_\|)^{1/10}+\lambda_{q+1}^{1/5}+(\lambda_{q+1}r_\perp r_\|^{-1})^{1/5}\Big)2^{q+1}
\\&\lesssim M_L\lambda_{q+1}^{{48}\alpha-6/35},
\endaligned
\end{equation}
and
\begin{equation}\label{principle est1nnp}
\aligned
\|w_{q+1}\|_{C^{1/10}_tL^{5/3}}+\| w_{q+1}\|_{C_tW^{1/5,5/3}}
&\lesssim M_L^{1/2}\lambda_{q+1}^{32\alpha-1/{35}}\leq \frac34M_L^{1/2}\delta_{q+1}a^{-\alpha/2}.
\endaligned
\end{equation}
In the last inequality above we used \eqref{parameter2}.
Hence, \eqref{induction wtp} follows from \eqref{error}. The second inequality  in \eqref{inductionv2p} on the level $q+1$ follows as well.

\subsubsection{Proof of \eqref{inductionv C1p} on the level $q+1$}

Using \eqref{estimate aN ps1p} and similar as \cite[Section 3.1.4]{HZZ21} we find for $t\in [0,T_L]$
\begin{equation}\label{principle est2 psp}
\aligned
\|\tilde w_{q+1}^{(p)}\|_{C^1_{t,x}}%
&\lesssim \ell^{-15}((M_L(1+3q))^{1/2} +\gamma_{q+1}^{1/2})r_\perp^{-1}r_\|^{-1/2}\lambda_{q+1}^2,
\endaligned
\end{equation}
\begin{equation}\label{correction est2 psp}
\aligned
\|\tilde w_{q+1}^{(c)}\|_{C^1_{t,x}}
&\lesssim \ell^{-29}((M_L(1+3q))^{1/2}+\gamma_{q+1}^{1/2})r_\|^{-3/2}\lambda_{q+1}^2,
\endaligned
\end{equation}
and
\begin{equation}\label{temporal est2 psp}
\aligned
\|\tilde w_{q+1}^{(t)}\|_{C^1_{t,x}}%
\lesssim \ell^{-24}(M_L(1+3q)+\gamma_{q+1})r_\perp^{-1}r_\|^{-2}\lambda_{q+1}^{1+\alpha}.
\endaligned
\end{equation}
In particular, we see that the fact that the time derivative of $\chi$  behaves like $2\sigma_{q}^{-1}\lesssim \ell^{-1}$ does not pose any problems as the $C^{0}_{t,x}$-norms of $\tilde{w}_{q+1}^{(p)}$, $\tilde{w}_{q+1}^{(c)}$ and $\tilde{w}_{q+1}^{(t)}$ always contain smaller powers of $\ell^{-1}$.

Combining   \eqref{error13p} and \eqref{principle est2 psp}, \eqref{correction est2 psp}, \eqref{temporal est2 psp} with (\ref{ell}) we obtain for $t\in[0, T_L]$
\begin{equation*}
\aligned
\|v^{2}_{q+1}\|_{C^1_{t,x}}&\leq \|v_\ell\|_{C^1_{t,x}}+\|w_{q+1}\|_{C^1_{t,x}}\\
&\leq  (M_L(1+3q)+\gamma_{q+1})^{1/2}\left(\lambda_{q+1}^\alpha+C\lambda_{q+1}^{30\alpha+22/7}+C\lambda_{q+1}^{58\alpha+20/7}+C\lambda_{q+1}^{50\alpha+3}\right)
\leq M_L^{1/2}\lambda_{q+1}^4,
\endaligned
\end{equation*}
where we used $M_L(1+3q)+\gamma_{q+1}\leq \ell^{-1}$.
This implies  \eqref{inductionv C1p}.

\subsubsection{Proof of \eqref{induction wp ps} and the first inequality in \eqref{inductionv2p} on the level $q+1$}
\label{s:745}

Similarly, we derive  the following estimates: for $t\in[0, T_L]$ it follows from (\ref{ell}), (\ref{estimate aN ps1p}) and \eqref{bounds} that
\begin{equation}\label{principle est22 psp}
\aligned
\|\tilde w_{q+1}^{(p)}+\tilde w_{q+1}^{(c)}\|_{C_tW^{1,p}}
&\leq\sum_{\xi\in\Lambda}
\|\textrm{curl\,}\textrm{curl}(a_{(\xi)}V_{(\xi)})\|_{C_tW^{1,p}}
\\
&\lesssim \sum_{\xi\in \Lambda}\|a_{(\xi)}\|_{C^3_{t,x}}\|V_{(\xi)}\|_{C_tW^{3,p}}
\\
&\lesssim  \ell^{-29}((M_L(1+3q))^{1/2}+\gamma_{q+1}^{1/2})r_\perp^{2/p-1}r_\|^{1/p-1/2}\lambda_{q+1},
\endaligned
\end{equation}
and
\begin{equation}\label{corrector est2 psp}
\aligned
\|\tilde w_{q+1}^{(t)}\|_{C_tW^{1,p}}
&\leq\frac{1}{\mu}\sum_{\xi\in\Lambda}
\Big(\|a_{(\xi)}\|_{C^0_{t,x}}\|a_{(\xi)}\|_{C^1_{t,x}}\|\phi_{(\xi)}\|_{L^{2p}}^2\|\psi_{(\xi)}\|_{C_tL^{2p}}^2\\
&\quad+\|a_{(\xi)}\|_{C^0_{t,x}}^2\|\phi_{(\xi)}\|_{L^{2p}}\|\nabla \phi_{(\xi)}\|_{L^{2p}}\|\psi_{(\xi)}\|_{C_tL^{2p}}^2\\
&\quad+\|a_{(\xi)}\|_{C^0_{t,x}}^2\|\phi_{(\xi)}\|_{L^{2p}}^2\|\nabla \psi_{(\xi)}\|_{C_tL^{2p}}\|\psi_{(\xi)}\|_{C_tL^{2p}}\Big)
\\
&\lesssim \ell^{-23}(M_L(1+3q)+\gamma_{q+1})r_\perp^{2/p-2}r_\|^{1/p-1}\lambda_{q+1}^{-2/7}.
\endaligned
\end{equation}

We also have for $p=\frac{32}{32-7\alpha}$, $r_\perp^{2/p-2}r_{\|}^{1/p-1}\leq \lambda_{q+1}^\alpha$ and
\begin{align}\label{corrector est3 psp}
\|w_{q+1}\|_{C_tW^{1,p}}\lesssim (M_L(1+3q)+\gamma_{q+1})\ell^{-29}\lambda_{q+1}^{\alpha-1/7}\lesssim M_L^{1/2}\lambda_{q+1}^{60\alpha-1/7}\leq \frac34M_L^{1/2}\delta_{q+1}^{1/2}a^{-\alpha/2},
\end{align}
where we used the condition for $\alpha, \beta$ and \eqref{parameter2} in the second step, which combined with \eqref{error11p} implies
\eqref{induction wp ps} and hence the first inequality of \eqref{inductionv2p}.

\subsection{Proof of  \eqref{p:gammap}}

We control the energy similarly as in \cite[Section 3.1.5]{HZZ21}.
By definition, we find
\begin{align}\label{eq:deltaE psp}
\begin{aligned}
&\big|\|v^2_{q+1}\|_{L^2}^2-\|v^2_q\|_{L^2}^2-3\gamma_{q+1}\big|
\leq \big|\|\tilde w_{q+1}^{(p)}\|_{L^2}^2-3\gamma_{q+1}\big|+\|\tilde w_{q+1}^{(c)}+\tilde w_{q+1}^{(t)}\|_{L^2}^2\\
&\ +2\|v_\ell(\tilde w_{q+1}^{(c)}+\tilde w_{q+1}^{(t)})\|_{L^1}
+2\|v_\ell \tilde w_{q+1}^{(p)}\|_{L^1}+2\|\tilde w_{q+1}^{(p)}(w_{q+1}^{(c)}+\tilde w_{q+1}^{(t)})\|_{L^1}+|\|v_\ell\|_{L^2}^2-\|v^{2}_q\|_{L^2}^2|.
\end{aligned}
\end{align}
Let us begin with the bound of the first term on the right hand side of \eqref{eq:deltaE psp}. We use \eqref{can} and the fact that $\mathring{R}_\ell$ is traceless to deduce for $t\in (4\sigma_q\wedge T_L,T_L]$
\begin{align*}
|\tilde w_{q+1}^{(p)}|^2-\frac{3\gamma_{q+1}}{(2\pi)^3}=6\sqrt{\ell^2+|\mathring{R}_\ell|^2}+\sum_{\xi\in \Lambda}a_{(\xi)}^2P_{\neq0}|W_{(\xi)}|^2
,
\end{align*}
hence
\begin{align}\label{eq:gg psp}
|\|\tilde w_{q+1}^{(p)}\|_{L^2}^2-3\gamma_{q+1}|\leq 6\cdot(2\pi)^3\ell+6\|\mathring{R}_\ell\|_{L^1}+\sum_{\xi\in \Lambda}\Big|\int a_{(\xi)}^2P_{\neq0}|W_{(\xi)}|^2\Big|.
\end{align}
Here we estimate each term separately. Using \eqref{ell1}  we find
\begin{align*}
6\cdot(2\pi)^3\ell\leq 6\cdot(2\pi)^3\lambda_{q+1}^{-{3\alpha}/2}\leq  \frac{1}{ 48}\lambda_{q+1}^{-2\beta }M_L\leq\frac{1}{ 48}\delta_{q+1}M_L,
\end{align*}
which requires $2\beta <{3\alpha}/2$ and choosing $a$ large to absorb the constant.
Using \eqref{eq:Rp} on $\mathring{R}_q$ and $\mathrm{supp}\varphi_\ell\subset [0,\ell]$ we know for $t\in (4\sigma_q\wedge T_L,T_L]$
\begin{align*}
6\|\mathring{R}_\ell(t)\|_{L^1}\leq  6\delta_{q+1}M_L.
\end{align*}
For the last term in \eqref{eq:gg psp} we use similar argument as in \cite[Section 3.1.5]{HZZ21} to have since $M_{L}\geq 1$
\begin{align*}
\sum_{\xi\in\Lambda}&\Big|\int a_{(\xi)}^2\mathbb{P}_{\neq0}|W_{(\xi)}|^2\Big|\lesssim\lambda_{q+1}^{158\alpha-1/7}(M_L(1+3q)+K)\lesssim\lambda_{q+1}^{160\alpha-1/7}
\leq\frac1{24}\lambda_1^{2\beta}\lambda_{q+1}^{-2\beta} M_L=\frac1{24}\delta_{q+1}M_L,
\end{align*}
where we used $M_{L}(1+3q)+K\leq\ell^{-1}\leq\lambda_{q+1}^{2\alpha}$ as well as $160\alpha+2\beta<1/7$.
This completes the bound of \eqref{eq:gg psp}.

Going back to \eqref{eq:deltaE psp}, we control the  remaining terms as follows.
Using  the estimates \eqref{correction est psp}, \eqref{temporal est1 psp} and \eqref{ell} we  have for $t\in (4\sigma_q\wedge T_L,T_L]$
\begin{align*}
\|\tilde w_{q+1}^{(c)}+\tilde w_{q+1}^{(t)}\|_{L^2}^2&
\lesssim (M_L+\gamma_{q+1})\lambda_{q+1}^{88\alpha-4/7}
+(M_L^2+\gamma_{q+1}^2)\lambda_{q+1}^{64\alpha-2/7}\leq \frac{1}{48}
\lambda_{q+1}^{-2\beta }M_L\leq \frac{\delta_{q+1}}{48}M_L,
\end{align*}
where we use  $M_L+\gamma_{q+1}\leq \ell^{-1}$ to control $M_L+\gamma_{q+1}$. Similarly we use \eqref{eq:vl psp} together with \eqref{estimate wqp psp} to have for $t\in (4\sigma_q\wedge T_L,T_L]$
\begin{align*}
&2\|v_\ell(\tilde w_{q+1}^{(c)}+\tilde w_{q+1}^{(t)})\|_{L^1}+2\|\tilde w_{q+1}^{(p)}(\tilde w_{q+1}^{(c)}+\tilde w_{q+1}^{(t)})\|_{L^1}
\lesssim M_{0}((M_L(1+3q))^{1/2}+K^{1/2})\|\tilde w_{q+1}^{(c)}+\tilde w_{q+1}^{(t)}\|_{L^2}\\
&\qquad \lesssim M_{0}((M_L(1+3q))^{1/2}+K^{1/2})\left((M_L^{1/2}+\gamma_{q+1}^{1/2})\lambda_{q+1}^{44\alpha-2/7}
+(M_L+\gamma_{q+1})\lambda_{q+1}^{32\alpha-1/7}\right)
\\
&\qquad\leq \frac{1}{48}
\lambda_{q+1}^{-2\beta }M_L\leq \frac{\delta_{q+1}}{48}M_L,
\end{align*}
where we used $M_L(1+3q)+K\leq \ell^{-1}$ and we possibly increased $a$ to absorb $M_{0}$.
We use \eqref{ell} and \eqref{principle est1 psp} and $\|v_\ell\|_{C^1_{t,x}}\leq \|v_q^2\|_{C^1_{t,x}} $ to have for every $\kappa>0$
\begin{align*}
2\|v_\ell \tilde w_{q+1}^{(p)}\|_{L^1}
&\lesssim\|v_\ell\|_{L^\infty}\|\tilde w_{q+1}^{(p)}\|_{L^1}\lesssim M_L^{1/2}\lambda_q^4\ell^{-8}(M_L^{1/2}\delta_{q+1}^{1/2}+\gamma_{q+1}^{1/2})r_{\perp}^{1-\kappa}r_{\|}^{\frac12(1-\kappa)}\\
&\lesssim (M_L+K)\lambda_{q+1}^{17\alpha-\frac87(1-\kappa)}\leq \frac{1}{96}
\lambda_{q+1}^{-2\beta }M_L\leq \frac{\delta_{q+1}}{96}M_L.
\end{align*}
For the last  terms by \eqref{eq:vl psp} we have
\begin{align*}
|\|v_\ell\|_{L^2}^2-\|v^{2}_q\|_{L^2}^2|&\leq \|v_\ell-v^{2}_q\|_{L^2}(\|v_\ell\|_{L^2}+\|v^{2}_q\|_{L^2})\\
&\lesssim \ell\lambda_q^4 M_L^{1/2} M_{0}(M_L(1+3q)+K)^{1/2}\\
&\leq \frac{1}{96}
\lambda_{q+1}^{-2\beta }M_L\leq \frac{\delta_{q+1}}{96}M_L .
\end{align*}
which requires
$M_L(1+3q)+K< \lambda_{q+1}^{\alpha-2\beta}$ as in \eqref{parameter1} and $a$ large enough to absorb the extra constant.

Combining the above estimate \eqref{p:gammap} follows.

\subsection{Definition of the Reynolds stress $\mathring{R}_{q+1}$}\label{s:def1p}

Considering the equation for the difference $v^{2}_{q+1}-v_{\ell}$, we obtain the formula for the new Reynolds stress
\begin{equation}\label{stressp}
\aligned
&\div\mathring{R}_{q+1}-\nabla p^{2}_{q+1}\\
&=\underbrace{-\Delta w_{q+1}+\partial_t(\tilde{w}_{q+1}^{(p)}+\tilde{w}_{q+1}^{(c)})+\div((v_\ell+v_{q+1}^1)\otimes {w}_{q+1}+{w}_{q+1}\otimes (v_\ell+v_{q+1}^1))}_{\div(R_{\textrm{lin}})+\nabla p_{\textrm{lin}}}
\\&\quad+\underbrace{\div\left((\tilde{w}_{q+1}^{(c)}+\tilde{w}_{q+1}^{(t)})\otimes w_{q+1}+\tilde{w}_{q+1}^{(p)}\otimes (\tilde{w}_{q+1}^{(c)}+\tilde{w}_{q+1}^{(t)})\right)}_{\div(R_{\textrm{cor}})+\nabla p_{\textrm{cor}}}
\\&\quad+\underbrace{\div(\tilde{w}_{q+1}^{(p)}\otimes \tilde{w}_{q+1}^{(p)}+\mathring{R}_\ell)+\partial_t\tilde{w}_{q+1}^{(t)}}_{\div(R_{\textrm{osc}})+\nabla p_{\textrm{osc}}}
\\&\quad +\underbrace{\div((V^{2}_{q+1}-V^{2}_{q})+(V^{2,*}_{q+1}-V^{2,*}_{q}))}_{\div R_{\text{com}1}+\nabla p_{\textrm{com}1}}
\\&\quad+\underbrace{\div\left((v_\ell+v_{q+1}^1)\otimes(v_\ell+v_{q+1}^1)-(v_q^1+v_{q}^2)\otimes(v_q^1+v_q^2)
	\right)}_{\div(R_{\textrm{com}2})+\nabla p_{\textrm{com}2}}
\\&\quad+\div(R_{\textrm{com}})-\nabla p_\ell,
\endaligned
\end{equation}
where, using the notation $v_q=v_q^1+v_q^2$,
\begin{equation*}
\begin{aligned}
&V^{2}_{q+1}-V^{2}_{q}
\\&= (v^{2}_{q+1}-v_q^2) \osucccurlyeq \Delta_{> R} z^{\<20>}  + (v_{q+1}-v_q) \otimes \Delta_{\leqslant R} z^{\<20>}
\\&\quad  +(v^{2}_{q+1}-v_q^2) \osucc \Delta_{> R} z + (v_{q+1}-v_q) \left( \oprec + \osucc
   \right) \Delta_{\leqslant R} z + (v^{2}_{q+1}-v_q^2) \varocircle z
\\&\quad -\mathbb{P} [(v_{q+1}-v_q )\prec \Delta_{\leqslant R} \mathcal{I} \nabla z]
   \varocircle z - ([\mathbb{P}, (v^{2}_{q+1}-v_q^2) \prec] \mathcal{I} \nabla z)
   \varocircle z
\\&\quad   - ([\mathbb{P}, (v^{1}_{q+1}-v_q^1) \prec] \mathcal{I} \nabla z)
   \varocircle \Delta_{\leqslant R} z
\\&\quad   - \tmop{com} (v^{1}_{q+1}-v_q^1, \mathbb{P}\mathcal{I} \nabla z, \Delta_{\leqslant R}
   z) - \tmop{com} (v^{2}_{q+1}-v_q^2, \mathbb{P}\mathcal{I} \nabla z, z) -(v^{2}_{q+1}-v_q^2) \succcurlyeq \Delta_{> R} z^{\<101>}
\\&\quad    - (v_{q+1}-v_q) \cdot \Delta_{\leqslant R} z^{\<101>} + (v^{\sharp}_{q+1}-v_q^{\sharp})
   \varocircle \Delta_{\leqslant R} z.
\end{aligned}
\end{equation*}

Applying the inverse divergence operator $\mathcal{R}$ we define
\begin{equation*}
R_{\textrm{lin}}:=-\mathcal{R}\Delta w_{q+1}+\mathcal{R}\partial_t(\tilde{w}_{q+1}^{(p)}+\tilde{w}_{q+1}^{(c)})
+(v_\ell+v_{q+1}^1)\,\mathring{\otimes}\, w_{q+1}+w_{q+1}\,\mathring{\otimes}\, (v_\ell+v_{q+1}^1),
\end{equation*}
\begin{equation*}
R_{\textrm{cor}}:=(\tilde{w}_{q+1}^{(c)}+\tilde{w}_{q+1}^{(t)})\,\mathring{\otimes}\, w_{q+1}+\tilde{w}_{q+1}^{(p)}\,\mathring{\otimes}\, (\tilde{w}_{q+1}^{(c)}+\tilde{w}_{q+1}^{(t)}),
\end{equation*}
$R_{\textrm{com}1}$ is the trace-free part of the matrix
$$
(V^{2}_{q+1}-V^{2}_{q})+(V^{2,*}_{q+1}-V^{2,*}_{q}),
$$
and
\begin{equation*}
R_{\textrm{com}2}:=(v_\ell+v_{q+1}^1)\,\mathring{\otimes}\, (v_\ell+v_{q+1}^1)-(v_q^2+v_{q}^1)\,\mathring{\otimes}\, (v_q^2+v_q^1).
\end{equation*}
Similarly as in \cite{HZZ21}, using \eqref{can}, \eqref{equation for temporal}, the oscillation error is given by
\begin{equation*}
\aligned
R_{\textrm{osc}}&:=\chi^2\sum_{\xi\in\Lambda}\mathcal{R}
\left(\nabla a_{(\xi)}^2\mathbb{P}_{\neq0}(W_{(\xi)}\otimes W_{(\xi)}) \right)-\frac{\chi^2}{\mu}\sum_{\xi\in\Lambda}\mathcal{R}
\left(\partial_t a_{(\xi)}^2(\phi_{(\xi)}^2\psi_{(\xi)}^2\xi) \right)
\\&\qquad+\mathcal{R}w_{q+1}^{(t)}\p_t\chi^2+(1-\chi^2)\mathring{R}_\ell
\\&=:R_{\textrm{osc}}^{(x)}+R_{\textrm{osc}}^{(t)}+R_{\textrm{cut}}^{(1)}.
\endaligned
\end{equation*}
Finally we define the Reynolds stress on the level $q+1$ by
\begin{equation*}\aligned
\mathring{R}_{q+1}:=R_{\textrm{lin}}+R_{\textrm{cor}}+R_{\textrm{osc}}+R_{\textrm{com}}+R_{\textrm{com}1}+R_{\textrm{com}2}.
\endaligned
\end{equation*}
We observe that %
by  construction, $\mathring{R}_{q+1}$ is $(\mathcal{F}_t)_{t\geq0}$-adapted.

\subsection{Verification of the inductive estimates  for $\mathring{R}_{q+1}$}
\label{sss:Rp}

We shall establish the three bounds in \eqref{iteration Rp}. As the oscillations $w_{q+1}$ were  fully added for $t\in(\sigma_{q}\wedge T_{L},T_{L}]$, this is the good interval where  the desired smallness of $\mathring{R}_{q+1}$ is achieved. In the middle interval $t\in(\frac{\sigma_{q}}{2}\wedge T_{L},\sigma_{q}\wedge T_{L}]$, there is a part of $\mathring{R}_{q+1}$ involving the cut-off $1-\chi^{2}$, which can only be bounded by the previous stress $\mathring{R}_{q}$. In the time interval $t\in[0,\frac{\sigma_{q}}2\wedge T_{L}]$, there are no oscillations which could decrease the Reynolds stress and hence we can only prove a polynomial blow-up. Nevertheless, this eventually leads to convergence in $L^{p}$ in time as this bad time interval  is shrinking exponentially, cf. the proof of Theorem~\ref{thm:6.1}. Here, it is essential that we do not use regularity of $v^{1}_{q}$ and $v^{\sharp}_{q}$ to avoid the blow-up in time.

\medskip

\textbf{Case I.} Let $t\in (\sigma_q\wedge T_L,T_L]$. If $T_{L}\leq \sigma_{q}$ then there is nothing to estimate here, hence we assume that $\sigma_{q}<T_{L}$ and $t\in (\sigma_{q},T_{L}]$. In this regime, it holds  $\chi=1$ and so the truncation does not play any role in the estimates.
We estimate each term in the definition of $\mathring{R}_{q+1}$ separately.
We choose $p=\frac{32}{32-7\alpha}>1$ so that in particular that $r_\perp^{2/p-2}r_\|^{1/p-1}\leq \lambda_{q+1}^\alpha$.
For the linear  error we apply \eqref{inductionv C1p} to obtain
\begin{equation*}
\aligned
\|R_{\textrm{lin}}(t)\|_{L^1}
&\lesssim\|\mathcal{R}\Delta w_{q+1}\|_{L^p}+\|\mathcal{R}\partial_t(w_{q+1}^{(p)}+w_{q+1}^{(c)})\|_{L^p}+\|(v_\ell+v_{q+1}^1)\mathring{\otimes}w_{q+1}+w_{q+1}\mathring{\otimes}(v_\ell+v_{q+1}^1)\|_{L^1}
\\
&\lesssim\|w_{q+1}\|_{W^{1,p}}+\sum_{\xi\in\Lambda}\|\partial_t\textrm{curl}
(a_{(\xi)}V_{(\xi)})\|_{L^p}+(M_L^{1/2}\lambda_{q}^4+\|v_{q+1}^1(t)\|_{L^4})\|w_{q+1}\|_{L^{4/3}},
\endaligned
\end{equation*}
where by \eqref{bounds} and \eqref{estimate aN ps1p}
\begin{equation*}
\aligned
&\sum_{\xi\in\Lambda}\|\partial_t\textrm{curl}
(a_{(\xi)}V_{(\xi)})\|_{C_tL^p}\leq \sum_{\xi\in\Lambda}\left(\|
a_{(\xi)}\|_{C_tC^1_x}\|\partial_t V_{(\xi)}\|_{C_tW^{1,p}}+\|\partial_ta_{(\xi)}\|_{C_tC^1_x}\| V_{(\xi)}\|_{C_tW^{1,p}}\right)\\
&\lesssim (M_L(1+3q)+\gamma_{q+1})^{1/2} \ell^{-15}r_{\perp}^{2/p}r_{\|}^{1/p-3/2}\mu+(M_L(1+3q)+\gamma_{q+1})^{1/2} \ell^{-22}r_{\perp}^{2/p-1}r_{\|}^{1/p-1/2}\lambda_{q+1}^{-1}.
\endaligned
\end{equation*}

\br
For the  product $v^{1}_{q+1}\otimes w_{q+1}$ we used the $L^{4}$-norm of $v^{1}_{q}$ instead of $L^{2}$  in order to lower the required integrability of $w_{q+1}$. Indeed, $w_{q+1}$ is not  small in $L^{2}$ for $t\in(\sigma_{q},T_{L}]$, cf. \eqref{estimate wq psp}, \eqref{estimate wq ps1p}. On the other hand, as a consequence of \eqref{estimatev42}, increasing the integrability of $v^{1}_{q+1}$ leads to a blow-up in two respects: there is a blow up as $t\to0$ but also the time-weighted norm in $C_{t,3/8}L^{4}$  has only a diverging bound as $q\to\infty$. We show below that both these divergencies are compensated by the smallness of $w_{q+1}$ in $L^{4/3}$ and by using the fact that $t>\sigma_{q}$.
\er
In view of \eqref{corrector est3 psp}, \eqref{principle est1 ps1p}, \eqref{corr temporal ps1p} as well as \eqref{estimatev42} applied on the level $q+1$,  we  deduce  for $t\in(\sigma_q,T_L]$
\begin{equation*}
\aligned
\|R_{\textrm{lin}}(t)\|_{L^p}
&\lesssim M_L^{1/2}\lambda_{q+1}^{60\alpha-1/7}
 +(M_L(1+3q)+\gamma_{q+1})^{1/2}\Big(
\ell^{-22}\lambda_{q+1}^{\alpha-1/7}
\\&\ +M_L^{1/2}\ell^{-8}r_\perp^{1/2}r_\|^{1/4}(\lambda^4_{q}+\sigma_q^{-3/{8}}\lambda_{q+1}^{1/3+\kappa})\Big)
\\&\lesssim M_L^{1/2}\lambda_{q+1}^{60\alpha-1/7}+(M_L(1+3q)+\gamma_{q+1})^{1/2}M_L^{1/2}\lambda_{q+1}^{16\alpha-5/21+\kappa}\sigma_q^{-3/{8}}
\\
&\leq\frac{M_L\delta_{q+2}}{5}.
\endaligned
\end{equation*}
Here, we have taken $a$ sufficiently large and $\beta$ sufficiently small.

The estimates of $R_{\textrm{cor}}$ and   $R_{\textrm{osc}}$ are the same as the corresponding bounds in \cite{HZZ21}. The strongest requirement comes from the bound of $R_{\rm{cor}}$, namely, we have
$$
\|R_{\rm cor}(t)\|_{L^{p}}\lesssim((M_L(1+3q))^{1/2}+\gamma_{q+1}^{1/2})^3\lambda_{q+1}^{49\alpha-1/7}\leq \ell^{-3/2}\lambda_{q+1}^{49\alpha-1/7}
\leq\frac{M_{L}}{5}\lambda_{q+1}^{-2\beta b}\leq \frac{M_{L}}{5}\delta_{q+2},
$$
which is satisfied provided $ \lambda_{q+1}^{52\alpha-1/7}\leq\frac{M_{L}}{5}\lambda_{q+1}^{-2\beta b}.$

We use  standard mollification estimates in order to bound $R_{\textrm{com}}$. More precisely,  $R_{\rm{com}}$ has to vanish sufficiently fast in order to fulfill the first bound in \eqref{iteration Rp}. Hence, we need to use regularity of each term in $N$ which then by mollification estimates leads to the desired decay. To this end, there is a number of terms which require spatial regularity of $v^{1}_{q}$ as well as $v^{\sharp}_{q}$. But these norms  blow up as $t\to0$, cf. Proposition~\ref{p:v1}. Thus, we make use of the fact that $t>\sigma_{q}$ and that the corresponding blow-up of order  $\sigma_{q}^{-1/6}$ and $\sigma_{q}^{-3/10}$, respectively, can be absorbed by the smallness of $\ell$ due to our choice of the parameters in \eqref{parameter1}.

Let us now consider each term in $R_{\textrm{com}}$ separately.
For $I_1:=(v_q^1+v_q^2)\otimes \Delta_{\leq R}z^{\<20>}-(v_q^1+v_q^2)\cdot \Delta_{\leq R}z^{\<101>}$ we use Proposition \ref{p:v1} and \eqref{inductionv2p} to have
\begin{align*}
&\|I_1-I_1*_x\phi_\ell*_{t}\varphi_\ell\|_{L^1}
\\&\lesssim
\ell^{1/10}(\|v_q^2\|_{C_t^{1/10}L^{5/3}}+\|v_q^1\|_{C_{[\sigma_q/2,t]}^{1/6-\kappa}L^{5/3}})(\|\Delta_{\leq R} z^{\<20>}\|_{C_t^{1/10}L^\infty}+\|\Delta_{\leq R} z^{\<101>}\|_{C_t^{1/10}L^\infty})
\\&\quad+\ell^{1/10}(\|v_q^2\|_{C_tW^{1/5,5/3}}+\|v_q^1\|_{C_{[\sigma_q/2,t]}B^{1/6-\kappa}_{5/3,\infty}})(\|\Delta_{\leq R} z^{\<20>}\|_{C_tC^{1/10}}+\|\Delta_{\leq R} z^{\<101>}\|_{C_tC^{1/10}})
\\&\lesssim \ell^{1/10} \sigma_q^{-1/6}(M_L^{1/2}a^{-\alpha/4}+L^3N+L^4)L^{10},
\end{align*}
where we used
$$\|\Delta_{\leq R} z^{\<20>}\|_{C_tC^{1/10}}+\|\Delta_{\leq R} z^{\<20>}\|_{C_t^{1/10}L^{\infty}}\lesssim 2^{(1/5+2\kappa)R}(\| z^{\<20>}\|_{C^{1/10}_tC^{-1/5-\kappa}}+\| z^{\<20>}\|_{C_tC^{-\kappa}})\lesssim L^{10}$$
and similarly for $z^{\<101>}$.
The same estimate also holds for the symmetric counterpart.
For $I_2:=(v_q^1+v_q^2)\otimes (v_q^1+v_q^2)$ we use \eqref{inductionv psp} and \eqref{estimatev12} to have
\begin{align*}
&\|I_2-I_2*_x\phi_\ell*_{t}\varphi_\ell\|_{L^1}
\\&\lesssim\Big(\ell\|v_q^2\|_{C^1_{t,x}}+\ell^{1/{61}}\|v_q^1\|_{C_{[\sigma_q/2,t]}^{1/61}L^{2}}+\ell^{1/{31}}\|v_q^1\|_{C_{[\sigma_q/2,t]}H^{1/{30}-3\kappa}}\Big)(\|v^1_q\|_{C_tL^2}+\|v^2_q\|_{C_tL^2})
\\&\lesssim \sigma_{q}^{-1/6}\ell^{1/{61}}(M_La^{-{\alpha}/{8}}+L^6N^2+L^8)+(\ell\lambda_{q}^4+\sigma_q^{-1/6}\ell^{1/{61}}) M_L^{1/2}(M_L^{1/2}(1+3q)^{1/2}+K^{1/2})
 \\&\lesssim M_L^{1/2}(M_L(1+3q)+K)^{1/2}(\ell\lambda_q^4+\ell^{1/{61}}\sigma_q^{-1/6}),
\end{align*}
where we used interpolation and Proposition \ref{p:v1} and the  embedding Lemma \ref{lem:emb} to have
\begin{align*}
\|v_q^1\|_{C_{t,1/6}^{1/{61}}L^2}&\lesssim \|v_q^1\|^{\frac9{10(1-9\kappa)}}_{C_{t,1/6}H^{1/30-3\kappa}}\|v_q^1\|^{\frac{1-90\kappa}{10(1-9\kappa)}}_{C_{t,1/6}^{1/6-\kappa}H^{-3/10}}
\lesssim \|v_q^1\|^{\frac9{10(1-9\kappa)}}_{C_{t,1/6}B^{1/3-2\kappa}_{5/3,\infty}}\|v_q^1\|^{\frac{1-90\kappa}{10(1-9\kappa)}}_{C_{t,1/6}^{1/6-\kappa}L^{5/3}}
\\&\lesssim M_L^{1/2}a^{-\alpha/4}+L^3N+L^4\lesssim M_L^{1/2},
\end{align*}
provided $\kappa$ was chosen sufficiently small.
For $I_3:=v_q^2\osucccurlyeq \Delta_{> R}z^{\<20>}-v_q^2\succcurlyeq \Delta_{> R}z^{\<101>}$ we use paraproduct estimates Lemma \ref{lem:para} and \eqref{inductionv C1p} to have
\begin{align*}
&\|I_3-I_3*_x\phi_\ell*_{t}\varphi_\ell\|_{L^1}
\\&\lesssim\ell^{1/10}\|v_q^2\|_{C_{t,x}^1}(\| z^{\<20>}\|_{C_tC^{-\kappa}}+\|z^{\<101>}\|_{C_tC^{-\kappa}}+\| z^{\<20>}\|_{C_t^{1/10}C^{-1/5-\kappa}}+\| z^{\<101>}\|_{C_t^{1/10}C^{-1/5-\kappa}})
\\&\lesssim \ell^{1/10}\lambda_q^4LM_L^{1/2}.
\end{align*}
For $I_4:=(v_q^1+v_q^2)(\oprec+\osucc)\Delta_{\leq R}z$ we use paraproduct estimates Lemma \ref{lem:para} and \eqref{inductionv C1p} to have
\begin{align*}
&\|I_4-I_4*_x\phi_\ell*_{t}\varphi_\ell\|_{L^1}
\\&\lesssim\ell^{1/10}(\|v_q^2\|_{C_{t,x}^1}+\|v_q^1\|_{C_{[\sigma_q/2,t]}^{1/10}B^{1/{10}}_{5/3,\infty}})(\|\Delta_{\leq R} z\|_{C_t^{1/10}L^\infty}+\|\Delta_{\leq R} z\|_{C_tC^{1/{10}}})
\\&\lesssim\ell^{1/10}L^{29}\Big(M_L^{1/2}\lambda_q^4+\sigma_q^{-1/6}(a^{-\alpha/4}M_L^{1/2}+L^3N+L^4)\Big),
\end{align*}
where we used
$$\|\Delta_{\leq R} z\|_{C_tC^{1/10}}+\|\Delta_{\leq R} z\|_{C_t^{1/10}L^\infty}\lesssim 2^{R(7/{10}+2\kappa)}L^{1/2}\leq L^{29},$$
and by interpolation and Proposition \ref{p:v1}
\begin{align}\label{estimatevq}\|v_q^1\|_{C_{t,1/6}^{1/10}B^{1/{10}}_{5/3,\infty}}\lesssim\|v_q^1\|_{C_{t,1/6}B^{1/3-2\kappa}_{5/3,\infty}}+\|v_q^1\|_{C_{t,1/6}^{1/6-\kappa}L^{5/3}}
\lesssim M_L^{1/2}a^{-\alpha/4}+L^3N+L^4.
\end{align}
For $I_5:=v^2_q\osucc\Delta_{> R}z +v_q^2\varocircle z+v_q^\sharp \varocircle \Delta_{\leq R}z$ we use paraproduct estimates Lemma \ref{lem:para}, Proposition~\ref{p:v1} and \eqref{inductionv C1p} to have
\begin{align*}
&\|I_5-I_5*_x\phi_\ell*_{t}\varphi_\ell\|_{L^1}\\
&\lesssim
\ell^{1/24}(\|v_q^2\|_{C^1_{t,x}}+\|v_q^\sharp\|_{C_{[\sigma_q/2,t]}B^{3/5-\kappa}_{5/3,\infty}}+\|v_q^\sharp\|_{C^{1/20}_{[\sigma_q/2,t]}B^{{11}/{20}-\kappa}_{5/3,\infty}})(\|z\|_{C_t^{1/24}C^{-7/12-\kappa}}+\|z\|_{C_tC^{-1/2-\kappa}})
\\&\lesssim\ell^{1/24}LM_L^{1/2}\Big(\lambda_q^4+\sigma_q^{-3/10}(a^{-\alpha/8}M_L^{1/2}+L^5N+L^6)\Big)\lesssim \ell^{1/24}M_L(\lambda_q^4+\sigma_q^{-3/10}).
\end{align*}
Here we used $L\leq a^{\alpha/16}$ by the choice of the parameters and interpolation to bound $\|z\|_{C_t^{1/24}C^{-7/12-\kappa}}$ by
$$\|z\|_{C_tC^{-1/2-\kappa}}+\|z\|_{C^{1/10}_tC^{-7/10-\kappa}}\lesssim L^{1/2}.$$
For $I_6:=\mathbb{P}[v_q\prec \Delta_{\leq R}\cI(\nabla z)]\varocircle z$ we use \eqref{inductionv2p}, \eqref{estimatev12} and Lemma \ref{lem:para} to have
\begin{align*}
&\|I_6-I_6*_x\phi_\ell*_t\varphi_\ell\|_{L^1}
\\&\lesssim \ell^{1/{10}}(\|v_q^2\|_{C_t^{1/{10}}L^{5/3}}+\|v_q^1\|_{C^{1/{10}}_{[{\sigma_q}/2,t]}L^{5/3}})\\
&\qquad\qquad\qquad\times\Big(\|z\|_{C_tC^{-1/2-\kappa}}(\|\Delta_{\leq R}\cI\nabla z\|_{C_tC^{3/5+2\kappa}}+\|\Delta_{\leq R}\cI\nabla z\|_{C_t^{1/{10}}C^{1/2+2\kappa}})
\\&\qquad\qquad\qquad\qquad+\|z\|_{C_t^{1/{10}}C^{-7/{10}-\kappa}}\|\Delta_{\leq R}\cI\nabla z\|_{C_tC^{7/{10}+2\kappa}}\Big)
\\&\lesssim\ell^{1/{10}}L^{10}\sigma_q^{-1/6}(M_L^{1/2}a^{-\alpha/4}+L^3N+L^4)\lesssim \ell^{1/{10}}M_L\sigma_q^{-1/6}.
\end{align*}
Here we used $$\|\Delta_{\leq R}\cI\nabla z\|_{C_tC^{7/10+2\kappa}}+\|\Delta_{\leq R}\cI\nabla z\|_{C_t^{1/{10}}C^{1/2+2\kappa}}\lesssim 2^{R(1/5+3\kappa)}\|z\|_{C_tC^{-1/2-\kappa}}\lesssim L^9.$$

In view of the  two commutator estimates, Lemma~\ref{lem:com1} and Lemma~\ref{lem:com2}, the remaining terms containing $v^2_q$ could be controlled by
$$
\ell^{1/10}\|v_q^2\|_{C^1_{t,x}}(\| z\|^2_{C_tC^{-1/2-\kappa}}+\| z\|_{C_tC^{-1/2-\kappa}}\| z\|_{C^{1/10}_tC^{-7/10-\kappa}})\lesssim \ell^{1/10}\lambda_q^4M_L^{1/2}L,$$
and the remaining terms containing $v^1_q$ could be controlled by
\begin{align*}
&\ell^{1/10}\|v_q^1\|_{C^{1/10}_{[\sigma_q/2,t]}B^{1/10}_{5/3,\infty}}\Big(\| z\|_{C_tC^{-1/2-\kappa}}\| \Delta_{\leq R}z\|_{C_tC^{-1/2+1/5+2\kappa}}+\| z\|_{C_tC^{-1/2-\kappa}}\| \Delta_{\leq R}z\|_{C^{1/10}_tC^{-1/2+2\kappa}}\Big)
\\&\lesssim \ell^{1/10}L^{10}\sigma_q^{-1/6}(M_L^{1/2}a^{-\alpha/4}+L^3N+L^4)\lesssim \ell^{1/10}M_L\sigma_q^{-1/6}.
\end{align*}
Here we used \eqref{estimatev12} and \eqref{estimatevq} and
$$
\begin{aligned}
&\| \Delta_{\leq R}z\|_{C_tC^{-1/2+1/5+2\kappa}}+\| \Delta_{\leq R}z\|_{C^{1/10}_tC^{-1/2+2\kappa}}
\\&\lesssim 2^{R(1/5+3\kappa)}(\|z\|_{C_tC^{-1/2-\kappa}}+\|z\|_{C^{1/10}_tC^{-7/10-\kappa}})\lesssim L^9.
\end{aligned}
$$
Therefore, we have that for $t\in(\sigma_q, T_L]$
\begin{equation*}
\aligned
\|R_{\textrm{com}}(t)\|_{L^1}
&\lesssim M_L^{1/2}(\ell^{1/61}\sigma_q^{-1/6}+\ell\lambda_q^4) (M_L(1+3q)+K)^{1/2}+\ell^{1/{24}}M_L\lambda_{q}^4+\ell^{1/{24}}M_L\sigma_q^{-3/{10}}
\\&\leq \frac{M_L\delta_{q+2}}{5},
\endaligned
\end{equation*}
where we used
the  choice of $\ell$ in \eqref{ell1} and the conditions
$$\alpha>244\beta b,\quad\alpha b>128,\quad(M_L(1+3q)+K)^{1/2}\leq \ell^{-1/183},\quad\sigma_{q}^{-1/3}<\ell^{-1/183},$$
which can indeed be achieved by our conditions on the parameters.

Next, by the choice of $\alpha, \beta$ and $b$ we use paraproduct estimates Lemma \ref{lem:para} we can bound $R_{\textrm{com1}}$ uniformly over the interval $[0,t]$ for $p=\frac{32}{32-7\alpha}$
\begin{align*}
\|R_{\textrm{com}1}\|_{{C_{t}}L^1}&\lesssim (\|v^2_{q+1}-v^2_q\|_{C_tW^{2/3,p}}+\|v^1_{q+1}-v^1_q\|_{C_tB_{5/3,\infty}^{1/3-2\kappa}}+\|v^\sharp_{q+1}-v^\sharp_q\|_{C_tL^{5/3}})L^{21}
\\&\lesssim L^{21}M_L^{1/2}(\lambda_{q}^{-1/42}+\lambda_{q+1}^{-15\alpha/32}+\lambda_{q+1}^{60\alpha-1/7}+\lambda_{q+1}^{32\alpha-1/35})
\\&\leq \frac{M_L\delta_{q+2}}{10},
\end{align*}
where we used \eqref{error11p} and \eqref{corrector est3 psp} to control $\|v^2_{q+1}-v^2_q\|_{C_tW^{2/3,p}}$, \eqref{est:vq}, \eqref{error} and \eqref{principle est1nnp} to control $\|v^1_{q+1}-v^1_q\|_{C_tB_{5/3,\infty}^{1/3-2\kappa}}$ and also  \eqref{vsharp2}, \eqref{error} and \eqref{principle est1nnp} to control $\|v^\sharp_{q+1}-v^\sharp_q\|_{C_tL^{5/3}}$. %
Moreover, we applied
$1>168\beta b^2$  and $-32\alpha+\frac1{35}>2\beta b$, the condition on $\alpha,\beta,b$ and
\begin{align}\label{eq:Rz}
\| \Delta_{\leq R}z\|_{C_tC^{\kappa}}\lesssim 2^{R(1/2+2\kappa)}\|z\|_{C_tC^{-1/2-\kappa}}\lesssim L^{21}.
\end{align}

Furthermore, we use Proposition \ref{p:v1}, \eqref{est:vq}, \eqref{inductionv psp} and \eqref{principle est1nnp} to estimate also $R_{\textrm{com}2}$ uniformly over $[0,t]$ to obtain
\begin{align*}
\|R_{\textrm{com}2}\|_{C_tL^1}&\lesssim (M_L(1+3q)+K)^{1/2}(\|v^1_{q+1}-v^1_q\|_{C_tL^2}+\|v_\ell-v^2_q\|_{C_tL^2})
\\&\lesssim (M_L(1+3q)+K)^{1/2}M_L^{1/2}(\lambda_{q}^{-1/42}+\lambda_{q+1}^{-\alpha}+\lambda_{q+1}^{32\alpha-{1}/{35}})
\\&\leq \frac{M_L\delta_{q+2}}{10},
\end{align*}
Here we used $\alpha>244\beta b,$ $1>168\beta b^2$  and  $(M_L(1+3q)+K)^{1/2}\leq  \lambda_q^{1/84}<\ell^{-1/183}$.

\br
We note that it was essential in the estimate of $R_{\textrm{com}1}$ and $R_{\textrm{com}2}$ that no time weights were required for the difference $v^{1}_{q+1}-v^{1}_{q}$ and $v^\sharp_{q+1}-v^\sharp_q$. Indeed,  there would be no way  to absorb the  weight as time approaches zero. We also note that the bounds of $R_{\textrm{com}1}$ and $R_{\textrm{com}2}$ hold directly for all $t\in[0,T_{L}]$.
\er

Summarizing all the above estimates we obtain
\begin{equation*}
\aligned
&\|\mathring{R}_{q+1}(t)\|_{L^1}\leq M_L\delta_{q+2},
\endaligned
\end{equation*}
which is the desired bound.

\medskip

\textbf{Case II.} Let  $t\in (\frac{\sigma_q}2\wedge T_L, \sigma_q\wedge T_L]$. If $T_{L}\leq \frac{\sigma_{q}}2$ then there is nothing to estimate, hence we may assume $\frac{\sigma_{q}}2<T_{L}$ and  $t\in (\frac{\sigma_q}2, \sigma_q\wedge T_L]$. Then
we decompose $\mathring{R}_\ell=\chi^2\mathring{R}_\ell+(1-\chi^2)\mathring{R}_\ell$. The first part $\chi^2\mathring{R}_\ell$ is canceled (up to the oscillation error $\chi^{2}R_{\rm osc}$) by
$\tilde w_{q+1}^{(p)}\otimes \tilde w_{q+1}^{(p)}=\chi^{2} w_{q+1}^{(p)}\otimes  w_{q+1}^{(p)}$
and $\chi^2\partial_tw_{q+1}^{(t)}=\partial_t\tilde w_{q+1}^{(t)}-(\chi^2)'w_{q+1}^{(t)}$. So in this case in the definition of $\mathring{R}_{q+1}$ most terms are similar to   \textbf{Case I.}  and can be estimated similarly as above.
Therefore, we only have to consider $(1-\chi^2)\mathring{R}_\ell$, and
$$\div R_{\mathrm{cut}}:=\chi'(t)(w_{q+1}^{(p)}(t)+w_{q+1}^{(c)}(t))+(\chi^2)'(t)w_{q+1}^{(t)}(t).$$
We know
$$\|(1-\chi^2)\mathring{R}_\ell(t)\|_{L^1}\leq\sup_{s\in [t-\ell,t]} \|\mathring{R}_q(s)\|_{L^1}.$$
For $R_{\mathrm{cut}}$ we realize that the bounds \eqref{principle est1 ps1p} and \eqref{corr temporal ps1p} also hold for $w_{q+1}^{(p)}$, $w_{q+1}^{(c)}$ and $w_{q+1}^{(t)}$. Then we have for $\kappa>0$
\begin{align*}
\|R_{\mathrm{cut}}(t)\|_{L^1}&\leq\|\chi'(t)(w_{q+1}^{(p)}(t)+w_{q+1}^{(c)}(t))\|_{L^1}+\|(\chi^2)'(t)w_{q+1}^{(t)}(t)\|_{L^1}\\
&\lesssim\frac1{\sigma_q}\|w_{q+1}^{(p)}(t)\|_{L^1}+\frac1{\sigma_q}\|w_{q+1}^{(c)}(t)\|_{L^1}+\frac1{\sigma_q}\|w_{q+1}^{(t)}(t)\|_{L^1}
\\
&\leq\frac1{\sigma_q}\ell^{-8}((M_L+3qM_L)^{1/2}+\gamma_{q+1}^{1/2})r_\perp^{1-2\kappa}r_{\|}^{1/2-\kappa} \leq \frac{M_L}{10}\delta_{q+2},
\end{align*}
where we use $\sigma_q^{-1}\leq \ell^{-1}$ and $M_L^{1/2}(1+3q)^{1/2}+\gamma_{q+1}^{1/2}\leq \ell^{-1}$. For $R_{\textrm{com}}$, where the required space regularity of $v^{1}_{q}$ and $v^{\sharp}_{q}$ leads to a blow-up in time, we again use the fact that $t\geq \sigma_q/2$ and $4\ell\leq \sigma_q$ to have a similar bound as in the first case.

\medskip

\textbf{Case III.}
For $t\in [0,\frac{\sigma_q}2\wedge T_L]$ we know $v_\ell(t)=v_{q+1}^2(t)=0$ and so the Reynolds stress reduces to
$$\mathring{R}_{q+1}=\mathring{R}_\ell+R_{\mathrm{com1}}+R_{\mathrm{com}}+R_{\mathrm{com2}}.$$
The bounds for ${R_{\mathrm{com1}}}$ and $R_{\mathrm{com2}}$  hold as in \textbf{Case I.}.
Unlike \textbf{Case I.} and \textbf{Case II.}, here we cannot use regularity of $v^{1}_{q}$ due to the blow-up at $t=0$ as $t$ is no longer bounded away from zero. On the other hand, we do not have to show smallness of $R_{\textrm{com}}$ but only a polynomial blow-up, see \eqref{iteration Rp}.
Therefore, we do not try to use the mollification estimates. Instead, we bound  $\|R_{\mathrm{com}}(t)\|_{L^1}$ directly using Lemma~\ref{lem:para} and \eqref{eq:Rz}, \eqref{estimatev12}, \eqref{vsharp1} as
\begin{align}\label{R1}
\|R_{\mathrm{com}}(t)\|_{L^1}&\lesssim \|v_q^1\|_{C_tL^2}(\|\Delta_{\leq R}z^{\<20>}\|_{C_tL^\infty}+\|\Delta_{\leq R}z^{\<101>}\|_{C_tL^\infty}+\|\Delta_{\leq R}z\|_{C_tC^{\kappa}})\nonumber\\&\quad+\|v_q^1\|_{C_tL^2}^2+\|v_q^\sharp\|_{C_tL^{5/3}}\|\Delta_{\leq R}z\|_{C_tC^{\kappa}}+\|v_q^1\|_{C_tL^2}\| z\|_{C_tC^{-1/2-\kappa}}\|\Delta_{\leq R}z\|_{C_tC^{-1/2+2\kappa}}\nonumber
\\&\leq 2M_L,
\end{align}
where we used \eqref{eq:Rz} and the implicit constant can be absorbed by taking $a$ and $L$ large enough.

\br
We point out that for the two commutators in $R_{\textrm{com}}$ we  did not apply the commutator estimates   Lemma \ref{lem:com1} and Lemma \ref{lem:com2}, as these would require regularity of $v^{1}_{q}$. Instead, we estimated each term by the paraproduct estimates directly.
\er

As a result, it follows
$$\|\mathring{R}_{q+1}(t)\|_{L^1}\leq \sup_{s\in [t-\ell,t]} \|\mathring{R}_q(s)\|_{L^1}
+3M_L,$$
which completes the proof of \eqref{iteration Rp} and also \eqref{eq:Rp} and \eqref{bd:Rp}.

 \appendix
\renewcommand{\appendixname}{Appendix~\Alph{section}}
  \renewcommand{\theequation}{A.\arabic{equation}}
  \section{Intermittent jets}
  \label{s:B}

In this  part we recall the construction of intermittent jets from \cite[Section 7.4]{BV19}.
We point out that the construction is entirely deterministic, that is, none of the functions below depends on $\omega$.
Let us begin with  the following geometric lemma which can be found in  \cite[Lemma 6.6]{BV19}.

\bl\label{geometric}
Denote by $\overline{B_{1/2}}(\mathrm{Id})$ the closed ball of radius $1/2$ around the identity matrix $\mathrm{Id}$, in the space of $3\times 3$ symmetric matrices. There
exists $\Lambda\subset \mathbb{S}^2\cap \mathbb{Q}^3$ such that for each $\xi\in \Lambda$ there exists a  $C^\infty$-function $\gamma_\xi:\overline{B_{1/2}}(\mathrm{Id})\rightarrow\mathbb{R}$ such that
\begin{equation*}
R=\sum_{\xi\in\Lambda}\gamma_\xi^2(R)(\xi\otimes \xi)
\end{equation*}
for every symmetric matrix satisfying $|R-\mathrm{Id}|\leq 1/2$.
For $C_\Lambda=8|\Lambda|(1+8\pi^3)^{1/2}$, where $|\Lambda|$ is the cardinality of the set $\Lambda$, we define
the constant
\begin{equation*}
M=C_\Lambda\sup_{\xi\in \Lambda}(\|\gamma_\xi\|_{C^0}+\sum_{|j|\leq N}\|D^j\gamma_\xi\|_{C^0}).
\end{equation*}
 For each $\xi\in \Lambda$ let us define $A_\xi\in \mathbb{S}^2\cap \mathbb{Q}^3$ to be an orthogonal vector to $\xi$. Then for each $\xi\in\Lambda$ we have that $\{\xi, A_\xi, \xi\times A_\xi\}\subset \mathbb{S}^2\cap \mathbb{Q}^3$ form an orthonormal basis for $\mathbb{R}^3$.
We label by $n_*$ the smallest natural such that
\begin{equation*}\{n_*\xi, n_*A_\xi, n_*\xi\times A_\xi\}\subset \mathbb{Z}^3\end{equation*}
for every $\xi\in \Lambda$.
\el

Let $\Phi:\mathbb{R}^2\rightarrow\mathbb{R}$ be a smooth function with support in a ball of radius $1$. We normalize $\Phi$ such that
$\phi=-\Delta \Phi$ obeys
\begin{equation}\label{eq:phi}
\frac{1}{4\pi^2}\int_{\mathbb{R}^2}\phi^2(x_1,x_2)\dif x_1\dif x_2=1.
\end{equation}
By definition we know $\int_{\mathbb{R}^2}\phi dx=0$. Define $\psi:\mathbb{R}\rightarrow\mathbb{R}$ to be a smooth, mean zero function with support in the ball of radius $1$ satisfying
\begin{equation}\label{eq:psi}
\frac{1}{2\pi}\int_{\mathbb{R}}\psi^2(x_3)dx_3=1.
\end{equation}
For parameters $r_\perp, r_\|>0$ such that
\begin{equation*}r_\perp\ll r_\|\ll1,\end{equation*}
we define the rescaled cut-off functions
\begin{equation*}\phi_{r_\perp}(x_1,x_2)=\frac{1}{r_\perp}\phi\left(\frac{x_1}{r_\perp},\frac{x_2}{r_\perp}\right),\quad
\Phi_{r_\perp}(x_1,x_2)=\frac{1}{r_\perp}\Phi\left(\frac{x_1}{r_\perp},\frac{x_2}{r_\perp}\right),\quad \psi_{r_\|}(x_3)=\frac{1}{r_\|^{1/2}}\psi\left(\frac{x_3}{r_\|}\right).\end{equation*}
We periodize $\phi_{r_\perp}, \Phi_{r_\perp}$ and $\psi_{r_\|}$ so that they are viewed as periodic functions on $\mathbb{T}^2, \mathbb{T}^2$ and $\mathbb{T}$ respectively.

Consider a large real number $\lambda$ such that $\lambda r_\perp\in\mathbb{N}$, and a large time oscillation parameter $\mu>0$. For every $\xi\in \Lambda$ we introduce
\begin{equation*}\aligned
\psi_{(\xi)}(t,x)&:=\psi_{\xi,r_\perp,r_\|,\lambda,\mu}(t,x):=\psi_{r_{\|}}(n_*r_\perp\lambda(x\cdot \xi+\mu t))
\\ \Phi_{(\xi)}(x)&:=\Phi_{\xi,r_\perp,\lambda}(x):=\Phi_{r_{\perp}}(n_*r_\perp\lambda(x-\alpha_\xi)\cdot A_\xi, n_*r_\perp\lambda(x-\alpha_\xi)\cdot(\xi\times A_\xi))\\
\phi_{(\xi)}(x)&:=\phi_{\xi,r_\perp,\lambda}(x):=\phi_{r_{\perp}}(n_*r_\perp\lambda(x-\alpha_\xi)\cdot A_\xi, n_*r_\perp\lambda(x-\alpha_\xi)\cdot(\xi\times A_\xi)),
\endaligned\end{equation*}
 where $\alpha_\xi\in\mathbb{R}^3$ are shifts to ensure that $\{\Phi_{(\xi)}\}_{\xi\in\Lambda}$ have mutually disjoint support.

The intermittent jets $W_{(\xi)}:\mathbb{T}^3\times \mathbb{R}\rightarrow\mathbb{R}^3$ are defined as in \cite[Section 7.4]{BV19}.
\begin{equation}\label{intermittent}W_{(\xi)}(t,x):=W_{\xi,r_\perp,r_\|,\lambda,\mu}(t,x):=\xi\psi_{(\xi)}(t,x)\phi_{(\xi)}(x).\end{equation}
By the choice of $\alpha_\xi$ we have that
\begin{equation}\label{Wxi}
W_{(\xi)}\otimes W_{(\xi')}\equiv0, \textrm{ for } \xi\neq \xi'\in\Lambda,
\end{equation}
and by the normalizations  \eqref{eq:psi} we obtain
$$
\frac1{(2\pi)^3}\int_{\mathbb{T}^3}W_{(\xi)}(t,x)\otimes W_{(\xi)}(t,x)dx=\xi\otimes\xi.
$$
These facts combined with Lemma \ref{geometric} imply that
\begin{equation}\label{geometric equality}
\sum_{\xi\in\Lambda}\gamma_\xi^2(R)\frac1{(2\pi)^3}\int_{\mathbb{T}^3}W_{(\xi)}(t,x)\otimes W_{(\xi)}(t,x)dx=R,
\end{equation}
for every symmetric matrix $R$ satisfying $|R-\textrm{Id}|\leq 1/2$. Since $W_{(\xi)}$ are not divergence free, we  introduce the corrector term
\begin{equation}\label{corrector}
W_{(\xi)}^{(c)}:=\frac{1}{n_*^2\lambda^2}\nabla \psi_{(\xi)}\times \textrm{curl}(\Phi_{(\xi)}\xi)
=\textrm{curl\,curl\,} V_{(\xi)}-W_{(\xi)}.
\end{equation}
with
\begin{equation*}
V_{(\xi)}(t,x):=\frac{1}{n_*^2\lambda^2}\xi\psi_{(\xi)}(t,x)\Phi_{(\xi)}(x).
\end{equation*}
Thus we have
\begin{equation*}
\div\left(W_{(\xi)}+W_{(\xi)}^{(c)}\right)\equiv0.
\end{equation*}

Next, we  recall the key   bounds from \cite[Section 7.4]{BV19}. For $N, M\geq0$ and $p\in [1,\infty]$ the following holds provided $r_{\|}^{-1}\ll r_{\perp}^{-1}\ll \lambda$
\begin{equation}\label{bounds}\aligned
\|\nabla^N\partial_t^M\psi_{(\xi)}\|_{C_tL^p}&\lesssim r^{1/p-1/2}_\|\left(\frac{r_\perp\lambda}{r_\|}\right)^N
\left(\frac{r_\perp\lambda \mu}{r_\|}\right)^M,\\
\|\nabla^N\phi_{(\xi)}\|_{L^p}+\|\nabla^N\Phi_{(\xi)}\|_{L^p}&\lesssim r^{2/p-1}_\perp\lambda^N,\\
\|\nabla^N\partial_t^MW_{(\xi)}\|_{C_tL^p}+\frac{r_\|}{r_\perp}\|\nabla^N\partial_t^MW_{(\xi)}^{(c)}\|_{C_tL^p}+\lambda^2\|\nabla^N\partial_t^MV_{(\xi)}
\|_{C_tL^p}&\lesssim r^{2/p-1}_\perp r^{1/p-1/2}_\|\lambda^N\left(\frac{r_\perp\lambda\mu}{r_\|}\right)^M,
\endaligned\end{equation}
where the implicit constants may depend on $p, N$ and $M$, but are independent of $\lambda, r_\perp, r_\|, \mu$.
\renewcommand{\theequation}{B.\arabic{equation}}

\section{Estimates for the heat operator}\label{s:sch}

To deal with the singularity at zero we introduce the following two norms: for $\alpha\in(0,1)$, $p\in[1,\infty]$, $\gamma\geq0$
$$
\|f\|_{C_{T,\gamma}^{\alpha}L^p}:=\sup_{0\leq t\leq T}t^{\gamma}\|f(t)\|_{L^{p}}+\sup_{0\leq s<t\leq T}s^\gamma \frac{\|f(t)-f(s)\|_{L^p}}{|t-s|^\alpha},
$$
$$
 \|f\|_{C_{T,\gamma}B^\alpha_{p,\infty}}:=\sup_{0\leq t\leq T}t^\gamma \|f(t)\|_{B^\alpha_{p,\infty}}.$$

First, we recall the basic estimates for the heat semigroup $P_t:=e^{t\Delta}$ from \cite[Lemma 2.8]{ZZZ20}. Let $T\geq1$.

\bl\label{lem:heat}
 For any $\theta>0$ and $\alpha\in\mR$, there is a constant $C=C(\alpha,\theta)>0$ such that for $p,q\in [1,\infty]$ and all $t\in (0,T]$,
		\begin{align}\label{E1}
		\|P_t f\|_{B^{\theta+\alpha}_{p,q}}\lesssim_C T^{\theta/2} t^{-\theta/2}\|f\|_{B^\alpha_{p,q}}.
		\end{align}
	 For any $0<\theta<2$, there  is a constant $C=C(\theta)>0$ such that for all $t\in [0,1]$,
		\begin{align}\label{E2}
		\|P_t f-f\|_{L^p}\lesssim_C  t^{\theta/2}\|f\|_{B^{\theta}_{p,\infty}}.
		\end{align}
	\el

Then we prove the following for $\mathcal{I}f=\int_0^\cdot P_{t-s}f\dif s$.

\begin{lemma}
  \label{l:sch2}Let $\alpha \in (0, 2)$, $\beta \in (\alpha - 2, 0)$,
  $\gamma, \delta \in [0, 1)$, $p \in [1, \infty]$ so that
  $$\gamma - \delta - \alpha / 2 + \beta / 2 + 1 >0.$$
  Then
  \[ \| \mathcal{I}f \|_{C_{T, \gamma}^{\alpha / 2} L^p} + \| \mathcal{I}f
     \|_{C_{T, \gamma} B^{\alpha}_{p, \infty}} \lesssim T^{\gamma - \delta +
     1} \| f \|_{C_{T, \delta} B^{\beta}_{p, \infty}} . \]
\end{lemma}

\begin{proof}
  By Lemma \ref{lem:heat}, we have for $0 \leqslant t \leqslant T$
  \[ t^{\gamma} \| \mathcal{I}f (t) \|_{B^{\alpha}_{p, \infty}} \lesssim
     T^{(\alpha - \beta) / 2} t^{\gamma} \int_0^t (t - s)^{- (\alpha - \beta)
     / 2} s^{- \delta} \mathd s \| f \|_{C_{T, \delta} B^{\beta}_{p, \infty}},
  \]
  where
  \[ T^{(\alpha - \beta) / 2} t^{\gamma} \int_0^t (t - s)^{- (\alpha - \beta)
     / 2} s^{- \delta} \mathd s \]
  \[ = T^{(\alpha - \beta) / 2} t^{\gamma} \int_0^{t / 2} (t - s)^{- (\alpha -
     \beta) / 2} s^{- \delta} \mathd s + T^{(\alpha - \beta) / 2} t^{\gamma}
     \int_{t / 2}^t (t - s)^{- (\alpha - \beta) / 2} s^{- \delta} \mathd s \]
  \[ \lesssim T^{(\alpha - \beta) / 2} t^{\gamma - (\alpha - \beta) / 2}
     \int_0^{t / 2} s^{- \delta} \mathd s + T^{(\alpha - \beta) / 2} t^{\gamma
     - \delta} \int_{t / 2}^t (t - s)^{- (\alpha - \beta) / 2} \mathd s \]
  \[ \lesssim T^{(\alpha - \beta) / 2} t^{\gamma - \delta - (\alpha - \beta) /
     2 + 1} \leqslant T^{\gamma - \delta + 1}, \]
  provided $\gamma - \delta - (\alpha - \beta) / 2 + 1 \geqslant 0$.

   To bound the other norm, we use Lemma \ref{lem:heat} again to get for $0
  \leqslant s \leqslant t \leqslant T$, $| t - s | \leqslant 1, \kappa>0$ small enough
  \[ s^{\gamma} \| \mathcal{I}f (t) -\mathcal{I}f (s) \|_{L^p} \lesssim
     s^{\gamma} \int_s^t \| P_{t - r} f (r) \|_{B^{\kappa}_{p, \infty}} \mathd
     r + s^{\gamma} \| (P_{t - s} - I) \mathcal{I}f (s) \|_{L^p} \]
  \[ \lesssim s^{\gamma} \int_s^t (t - r)^{-
     (\kappa - \beta) / 2} r^{- \delta} \mathd r \| f \|_{C_{T, \delta}
     B^{\beta}_{p, \infty}} + s^{\gamma} (t - s)^{\alpha / 2} \| \mathcal{I}f
     (s) \|_{B^{\alpha}_{p, \infty}} . \]

  For the first term, we write $r = s + x (t - s)$ and obtain if $\gamma
  \geqslant \delta$
  \[ s^{\gamma} \int_s^t (t - r)^{- (\kappa -
     \beta) / 2} r^{- \delta} \mathd r \leqslant s^{\gamma - \delta}  \int_s^t (t - r)^{- (\kappa - \beta) / 2}
     \mathd r \]
  \[ = s^{\gamma - \delta} (t - s)^{1-\kappa/2+\beta/2} \int_0^1 (1 - x)^{- (\kappa - \beta) / 2}
     \mathd x \lesssim T^{\gamma - \delta} (t - s)^{\alpha / 2}   \]
  since $\beta/ 2 + 1 > \alpha/2$; whereas if $\gamma < \delta$ we get
  \[ s^{\gamma}  \int_s^t (t - r)^{- (\kappa-
     \beta) / 2} r^{- \delta} \mathd r  \]
  \[ = s^{\gamma} (t - s)^{1-\kappa/2+\beta/2} \int_0^1 (1 - x)^{- (\kappa- \beta) / 2} (s + x (t
     - s))^{- \delta + \gamma} (s + x (t - s))^{- \gamma} \mathd x \]
  \[ \leqslant (t - s)^{1-\kappa/2+\beta/2} \int_0^1 (1 - x)^{- (\kappa - \beta) / 2} (s + x (t -
     s))^{- \delta + \gamma} \mathd x \]
  \[ \leqslant (t - s)^{1-\kappa/2+\beta/2 - \delta + \gamma} \int_0^1 (1 - x)^{- (\kappa -
     \beta) / 2} x^{- \delta + \gamma} \mathd x\] 
  \[ \lesssim (t - s)^{\alpha / 2}  ,
  \]
  provided $- \alpha / 2 + 1 - \delta + \gamma +\beta/2> 0$.

  For the second term we use the previous estimate. Therefore, for $|t-s|\leq1$
  \[ s^{\gamma} \| \mathcal{I}f (t) -\mathcal{I}f (s) \|_{L^p} \lesssim (t -
     s)^{\alpha / 2} {T^{ 1 - \delta + \gamma}}  \| f \|_{C_{T,
     \delta} B^{\alpha}_{p, \infty}} . \]
  Since $B^{\alpha}_{p, \infty} \subset L^p$, we estimate the remaining term
  in the H{\"o}lder norm in time using the first estimate above
  \[ t^{\gamma} \| \mathcal{I}f (t) \|_{L^p} \lesssim t^{\gamma} \|
     \mathcal{I}f (t) \|_{B^{\alpha}_{p, \infty}} \lesssim T^{\gamma - \delta
     + 1} \| f \|_{C_{T,
     \delta} B^{\beta}_{p, \infty}}\]
  and the proof is complete.
\end{proof}

\begin{remark}\label{r:b.5}
 We observe that in Section~\ref{s:v1vs}, we need to use Lemma~\ref{l:sch2} for various combinations of $\gamma,\delta\in \{0,1/6,3/10\}$. As a consequence, the power of $T$ in the statement of Lemma~\ref{l:sch2} is always bounded by  2. Since due to the definition of the stopping time \eqref{stopping time3} we have $T\leq L^{1/2}$ we obtain a factor $L$ for any application of the Schauder estimate.
\end{remark}

\bl\label{commutator}

Let $\alpha\in(0,1)$, $p\in[1,\infty]$, $\gamma\in (0,1)$, $\beta\in\mathbb{R}$. Then for any $\kappa>0$, $T\geq 1$
$$\|[\mathcal{I},f\prec]g\|_{C_{T,\gamma}^{\alpha/4}B_{p,\infty}^{\alpha/2+\beta+2-\kappa}}+ \|[\mathcal{I},f\prec]g\|_{C_{T,\gamma} B_{p,\infty}^{\alpha+\beta+2-\kappa}}\lesssim T^{1+\alpha/2}(\|f\|_{C_{T,\gamma}^{\alpha/2}L^p}+\|f\|_{C_{T,\gamma} B_{p,\infty}^{\alpha}})\|g\|_{C_TC^\beta}.$$

\el

\begin{proof} We have $$[\mathcal{I},f\prec]g(t)=\int_0^te^{(t-s)\Delta}[(f(s)-f(t))\prec g(s)]\dif s+\int_0^t[e^{(t-s)\Delta}, f(t)\prec] g(s)\dif s=I_1(t)+I_2(t).$$
	Then by \eqref{E1} we have
	\begin{align*}
	t^\gamma \|I_1(t)\|_{B_{p,\infty}^{\alpha+\beta+2-\kappa}}
	&\lesssim t^\gamma T^{1-\kappa/2+\alpha/2}\int_0^t(t-s)^{-1+\kappa/2}s^{-\gamma}\dif s\|f\|_{C_{T,\gamma}^{\alpha/2}L^p}\|g\|_{C_TC^\beta}\\
	&\lesssim T^{1+\alpha/2}\|f\|_{C_{T,\gamma}^{\alpha/2}L^p}\|g\|_{C_TC^\beta},\end{align*}
	and by \cite[Lemma A.1]{CC15}  (see also \cite{MW18})
	\begin{align*}&t^\gamma \|I_2(t)\|_{B_{p,\infty}^{\alpha+\beta+2-\kappa}}
	\lesssim t^\gamma\int_0^t(t-s)^{-1+\kappa/2}\dif s\|f(t)\|_{B_{p,
			\infty}^{\alpha}}\|g\|_{C_TC^\beta}
	\lesssim T^\kappa\|f\|_{C_{T,\gamma}B_{p,
			\infty}^{\alpha}}\|g\|_{C_TC^\beta}.\end{align*}
	Moreover, for $0\leq t_1<t_2\leq T$ satisfying $|t_2-t_1|\leq 1$
	\begin{align*}I_1(t_2)-I_1(t_1)
	&=\int_0^{t_1}\Big(e^{(t_2-s)\Delta}[(f(s)-f(t_2))\prec g(s)]-e^{(t_1-s)\Delta}[(f(s)-f(t_1))\prec g(s)]\Big)\dif s\\
	&\qquad+\int_{t_1}^{t_2}e^{(t_2-s)\Delta}[(f(s)-f(t_2))\prec g(s)]\dif s,
	\end{align*}
	which by Lemma \ref{lem:heat} implies that
\begin{align*}
	&t_1^\gamma \|I_1(t_2)-I_1(t_1)\|_{B_{p,\infty}^{\alpha/2+\beta+2-\kappa}}
		\\&\lesssim  T^{1-\kappa/2+\alpha/2}\|f\|_{C_{T,\gamma}^{\alpha/2}L^p}\|g\|_{C_TC^\beta}\\
		&\qquad\qquad \times t_1^\gamma\Big(|t_1-t_2|^{\alpha/4}\int_0^{t_1}(t_1-s)^{-1+\kappa/2}s^{-\gamma}\dif s+\int_{t_1}^{t_2}(t_2-s)^{-1+\alpha/4+\kappa/2}s^{-\gamma}\dif s\Big)
		\\&\lesssim T^{1+\alpha/2}\|f\|_{C_{T,\gamma}^{\alpha/2}L^p}\|g\|_{C_TC^\beta}|t_1-t_2|^{\alpha/4}.\end{align*}
	Also, it holds
\begin{align*}&I_2(t_2)-I_2(t_1)
		\\&=\int_0^{t_1}[e^{(t_2-s)\Delta}-e^{(t_1-s)\Delta}, f(t_1)\prec] g(s)\dif s+\int_0^{t_1}[e^{(t_2-s)\Delta}, (f(t_2)-f(t_1))\prec] g(s)\dif s
		\\&\quad+\int_{t_1}^{t_2}[e^{(t_2-s)\Delta}, f(t_2)\prec] g(s)\dif s=J_1+J_2+J_3.
		\end{align*}
		Then we obtain by \cite[Lemma A.1]{CC15}  (see also \cite{MW18})
		\begin{align*}&t_1^\gamma \|J_2+J_3\|_{B_{p,\infty}^{\alpha/2+\beta+2-\kappa}}
		\\&\lesssim  (\|f\|_{C_{T,\gamma}^{\alpha/4}B^{\alpha/2}_{p,\infty}}+\|f\|_{C_{T,\gamma}B^{\alpha}_{p,\infty}})\|g\|_{C_TC^\beta}\\
		&\qquad\qquad\times \Big(|t_1-t_2|^{\alpha/4}\int_0^{t_1}(t_1-s)^{-1+\kappa/2}\dif s+\int_{t_1}^{t_2}(t_2-s)^{-1+\alpha/4+\kappa}\dif s\Big)
		\\&\lesssim (\|f\|_{C_{T,\gamma}^{\alpha/2}L^p}+\|f\|_{C_{T,\gamma}B^{\alpha}_{p,\infty}})\|g\|_{C_TC^\beta}|t_1-t_2|^{\alpha/4}T^\kappa.
		\end{align*}
	Since $D^m ((1-e^{-2^{2j} (t_2-t_1)|\xi|^2})e^{-2^{2j}(t_1-s)|\xi|^2})\lesssim (|t_2-t_1| 2^{2j})^\delta e^{-c2^{2j}(t_1-s)}$ for $m\in \mathbb{N}_0$, $\delta\in(0,1)$ and $\xi$ in an annulus, by similar argument as \cite[Lemma A.1]{CC15} (see also \cite{MW18}) we obtain
	$$\|[e^{(t_2-s)\Delta}-e^{(t_1-s)\Delta}, f(t_1)\prec] g(s)\|_{B_{p,\infty}^{\alpha/2+\beta+2-\kappa}}\lesssim |t_2-t_1|^{\alpha/4}\|f(t_1)\|_{B_{p,\infty}^{\alpha}}(t_1-s)^{-1+\kappa/2}\|g(s)\|_{C^\beta}.$$
This implies
	\begin{align*}
	t_1^\gamma \|J_1\|_{B_{p,\infty}^{\alpha/2+\beta+2-\kappa}}
	&\lesssim  \|f\|_{C_{T,\gamma}B_{p,\infty}^{\alpha}}\|g\|_{C_TC^\beta}|t_1-t_2|^{\alpha/4}\int_0^{t_1}(t_1-r)^{-1+\kappa/2}\dif r
	\\&\lesssim \|f\|_{C_{T,\gamma}B_{p,\infty}^{\alpha}}\|g\|_{C_TC^\beta}|t_1-t_2|^{\alpha/4}T^\kappa
	\end{align*}
	and the proof is complete.
\end{proof}

\def\cprime{$'$} \def\ocirc#1{\ifmmode\setbox0=\hbox{$#1$}\dimen0=\ht0
	\advance\dimen0 by1pt\rlap{\hbox to\wd0{\hss\raise\dimen0
			\hbox{\hskip.2em$\scriptscriptstyle\circ$}\hss}}#1\else {\accent"17 #1}\fi}

\end{document}